\documentclass[10pt]{article}

\usepackage[margin=1in]{geometry}
\usepackage{bm}
\usepackage{comment}
\usepackage{latexsym, amsthm, amsmath, enumerate,mathrsfs}
\usepackage{graphicx, xspace, ifthen, rotating, pict2e}
\usepackage{multirow,multicol,mathtools}
\usepackage[svgnames]{xcolor}
\usepackage[charter]{mathdesign}
\usepackage{algpseudocode, algorithmicx}
\usepackage{tikz}
\usepackage{ytableau}
\usepackage[vcentermath,enableskew]{youngtab}
\usepackage[cmtip,all]{xy}
\newcommand{\longsquiggly}{\xymatrix{{}\ar@{~>}[r]&{}}}

\newtheorem{lemdef}{Lemma/Definition}[section]{\bf}{\it}
{\bf}{\rmfamily}
{\bf}{\rmfamily}
{\bf}{\rmfamily}
{\bf}{\rmfamily}

\newtheorem{thm}{Theorem}[section]
\newtheorem{proposition}[thm]{Proposition}
\newtheorem{lemma}[thm]{Lemma}
\newtheorem{corollary}[thm]{Corollary}
\newtheorem{conjecture}[thm]{Conjecture}

\theoremstyle{definition}

\newtheorem{definition}[thm]{Definition}
\newtheorem{remark}[thm]{Remark}
\newtheorem{example}[thm]{Example}

\newcommand{\newword}[1]{\textbf{\emph{#1}}}

\newcommand{\sh}{\mathrm{sh}}
\newcommand{\rsh}{\mathrm{rsh}}
\newcommand{\rectify}{\mathrm{rect}}

\newcommand{\esh}{\mathrm{esh}}
\newcommand{\partesh}{\mathrm{esh}}
\newcommand{\partsh}{\mathrm{sh}}
\newcommand{\composedpsh}{\smash{\overline{\mathrm{sh}}}}

\newcommand{\first}{\mathrm{first}}
\newcommand{\eset}{\varnothing}
\newcommand{\rect}{{\scalebox{.3}{\yng(3,3)}}}
\newcommand{\stair}{{\scalebox{.22}{\young(\hfil\hfil\hfil,:\hfil\hfil,::\hfil)}}}

\newcommand{\ybox}{{\boxtimes}}
\newcommand{\ebox}{\scalebox{.5}{\yng(1)}}
\newcommand{\tinybox}{{\scalebox{.3}{\yng(1)}}}
\newcommand{\tyng}[1]{\scalebox{.5}{\yng(#1)}}

\newcommand{\LRyb}{\LR(\alpha,{\scalebox{.5}{\yng(1)}},\beta, \gamma)}
\newcommand{\LRby}{\LR(\alpha,\beta,{\scalebox{.5}{\yng(1)}}, \gamma)}

\newcommand{\LR}{\mathrm{LR}}

\newcommand{\DE}{\mathrm{DE}}
\newcommand{\DEyb}{\DE(\alpha,{\scalebox{.5}{\yng(1)}},\beta, \gamma)}
\newcommand{\DEby}{\DE(\alpha,\beta,{\scalebox{.5}{\yng(1)}}, \gamma)}

\newcommand{\K}{\mathrm{K}}

\newcommand{\Gr}{\mathrm{Gr}}
\newcommand{\Grkn}{\Gr(k, \mathbb{C}^n)}
\newcommand{\OG}{\mathrm{OG}}
\newcommand{\OGn}{\OG(n, \mathbb{C}^{2n+1})}
\newcommand{\defn}[1]{{\bf #1}}

\newcommand{\east}[1]{\ensuremath{\xrightarrow{\ #1\ }}}
\newcommand{\west}[1]{\ensuremath{\xleftarrow{\ #1\ }}}
\newcommand{\north}[1]{\ensuremath{\big \uparrow \!\text{\raisebox{.1ex}{\scriptsize $#1$}}}}
\newcommand{\south}[1]{\ensuremath{\big \downarrow \!\text{\raisebox{.1ex}{\scriptsize $#1$}}}}

\newcommand{\stepnorth}[2]{\vector(0,1){.92}\put(-.23,.4){\scriptsize$#1$}\put(0,1){#2}}

\newcommand{\stepnorthshiftW}[2]{\put(-.08,0){\vector(0,1){.95}\put(-.23,.4){\scriptsize$#1$}}\put(0,1){#2}}

\newcommand{\stepsouth}[2]{\vector(0,-1){.92}\put(.05,-.6){\scriptsize$#1$}\put(0,-1){#2}}

\newcommand{\stepeast}[2]{\vector(1,0){.92}\put(.4,-.25){\scriptsize$#1$}\put(1,0){#2}}
\newcommand{\stepeastA}[2]{\vector(1,0){.92}\put(.4,.05){\scriptsize$#1$}\put(1,0){#2}}

\newcommand{\stepwest}[2]{\vector(-1,0){.92}\put(-.5,.05){\scriptsize$#1$}\put(-1,0){#2}}

\newcommand{\stepwestshiftS}[2]{\put(0,-.08){\vector(-1,0){.92}\put(-.5,-.25){\scriptsize$#1$}}\put(-1,0){#2}}
\newcommand{\tensor}{\otimes}
\newcommand{\GL}{\mathrm{GL}}
\newcommand{\Wr}{\mathrm{Wr}}

\newcommand{\ltst}{\prec}
\newcommand{\leqst}{\preceq}

\title{Schubert curves in the orthogonal Grassmannian}

\author{Maria Gillespie, Jake Levinson, Kevin Purbhoo}

\begin{document}
\maketitle{}

\begin{abstract}
We develop a combinatorial rule to compute the real geometry of type B \textit{Schubert curves} $S(\lambda_\bullet)$ in the orthogonal Grassmannian $\OGn$, which are one-dimensional Schubert problems defined with respect to orthogonal flags osculating the rational normal curve.  Our results are natural analogs of results previously known only in type A \cite{bib:GillespieLevinson}.  

First, using the type B Wronski map studied in \cite{bib:Pur10}, we show that the \textit{real} locus of the Schubert curve has a natural covering map to $\mathbb{RP}^1$, with monodromy operator $\omega$ defined as the commutator of jeu de taquin rectification and promotion on skew shifted semistandard tableaux.  We then introduce two different algorithms to compute $\omega$ without rectifying the skew tableau. The first uses the crystal operators introduced in \cite{bib:GLP}, while the second uses local switches much like jeu de taquin. The switching algorithm further computes the K-theory coefficient of the Schubert curve: its nonadjacent switches precisely enumerate Pechenik and Yong's \emph{shifted genomic tableaux}. The connection to K-theory also gives rise to a partial understanding of the \textit{complex} geometry of these curves.
\end{abstract}

\section{Introduction}

A \emph{Schubert curve} is a certain one-dimensional intersection
of Schubert varieties whose flags are maximally tangent to the
rational normal curve.  Concretely, the \emph{rational
normal curve} is the image of the Veronese embedding 
$\mathbb{P}^1
\hookrightarrow \mathbb{P}^{n-1} = \mathbb{P}(\mathbb{C}^n)$, 
defined by 
\[t \mapsto [1 : t : t^2 : \cdots : t^{n-1}] .\] 
The \emph{osculating} or \emph{maximally tangent flag} to this
curve at $t \in \mathbb{P}^1$,  is the complete flag $\mathcal{F}_t$
in $\mathbb{C}^n$ formed by the iterated derivatives of this map;
hence the $i$-th part of
the flag is spanned by the top $i$ rows of the matrix 
\begin{equation*}
\label{eqn:flag-matrix} \begin{bmatrix}
\big(\frac{d}{dt}\big)^{i-1}(t^{j-1}) \end{bmatrix} = \begin{bmatrix}
1 & t & t^2 & \cdots & t^{n-1} \\ 0 & 1 & 2t & \cdots & (n-1) t^{n-2}
\\ 0 & 0 & 2 & \cdots & (n-1)(n-2) t^{n-3} \\ \vdots & \vdots &
\vdots &\ddots & \vdots \\ 0 & 0 & 0 & \cdots & (n-1)!  \end{bmatrix}.
\end{equation*} 

Associated to the osculating flag $\mathcal{F}_t$ and a partition $\lambda$, 
we have a Schubert variety
$X_\lambda(\mathcal{F}_t)$ inside the Grassmannian $\Grkn$.
Schubert varieties with respect to osculating flags have been studied 
extensively in the context of degenerations of curves 
\cite{bib:Chan} \cite{bib:EH86} \cite{bib:Oss06}, Schubert calculus and 
the Shapiro-Shapiro Conjecture 
\cite{bib:MTV09} \cite{bib:Pur13} \cite{bib:Sot10}, and the geometry of 
the moduli space $\overline{M_{0,r}}(\mathbb{R})$ \cite{bib:Speyer}. 
They satisfy unusually strong transversality properties, particularly 
under the hypothesis that the osculation points $t$ are 
real; Mukhin, Tarasov and Varchenko \cite{bib:MTV09} showed
that every zero-dimensional intersection of the form
\begin{equation}
\label{eqn:schubertintersection}
  X_{\lambda^{(1)}} (\mathcal{F}_{t_1})
   \cap \cdots \cap X_{\lambda^{(r)}} (\mathcal{F}_{t_r}),
\end{equation}
with $t_1 < \dots < t_r$ real, is a transverse intersection.
The points in such an intersection are enumerated by certain chains of
Littlewood-Richardson tableaux.  Moreover,
the behaviour of such intersections under monodromy and degeneration
(as the osculating points collide) has a remarkable description in terms 
combinatorial operations on these tableaux
\cite{bib:Chan} \cite{bib:Levinson} \cite{bib:Pur10} \cite{bib:Speyer}.
As such, zero-dimensional intersections of the 
form~\eqref{eqn:schubertintersection} exhibit much deeper connections 
with tableau combinatorics and the Littlewood-Richardson rule than one 
finds using general flags.

Schubert curves are one-dimensional intersections of the form
\eqref{eqn:schubertintersection}.
More precisely, a \newword{Schubert curve} in $\Grkn$ is 
defined to be the intersection 
\[
S = S(\lambda^{(1)}, \ldots, \lambda^{(r)}) 
= X_{\lambda^{(1)}}(\mathcal{F}_{t_1})
\cap \cdots \cap X_{\lambda^{(r)}}(\mathcal{F}_{t_r}),
\]
where the osculation points $t_i$ are real numbers with $0 = t_1 <
t_2 < \cdots < t_r = \infty$, and $\lambda^{(1)},\ldots,\lambda^{(r)}$
are partitions for which $\sum |\lambda^{(i)}|=k(n-k)-1$.
Every such intersection is one-dimensional (if nonempty) and reduced 
\cite{bib:Levinson}, but not necessarily irreducible.  Moreover,
the \emph{real} connected components of $S$ can be described by combinatorial 
operations, related to jeu de taquin and Sch\"{u}tzenberger's promotion 
and evacuation, on chains of skew Young tableaux.  

In \cite{bib:GillespieLevinson}, a fast, local algorithm called \emph{evacuation-shuffling} was introduced to compute the real topology of $S$. Moreover, the particular (local) combinatorial structure of evacuation-shuffling resulted in new bijective connections to the structure coefficients arising in computing the K-class of the Schubert curve in the K-theory of the Grassmannian.  In this sense, these connections serve as a higher dimensional analog of the Littlewood-Richardson rule, enabling a combinatorial understanding of one-dimensional (rather than zero-dimensional) intersections of Schubert varieties via tableaux.

The purpose of this paper is to extend this story of Schubert curves
to the type B setting, specifically
to the orthogonal Grassmannian $\OGn$,
the variety of $n$-dimensional \emph{isotropic} subspaces 
of $\mathbb{C}^{2n+1}$ with respect to a fixed nondegenerate 
symmetric bilinear form.  
For both geometric and combinatorial reasons,
this is a natural place to extend the results of
\cite{bib:Levinson} and \cite{bib:GillespieLevinson}.
Like the Grassmannian, $\OGn$ is a minuscule flag
variety; in fact, apart from the simple family of even-dimensional quadrics,
this is essentially the only% 
\footnote{Strictly speaking, the classification of minuscule flag varieties 
includes two other infinite families: $\mathbb{P}^{2n-1}$ in type C, and 
$\OG(n+1,\mathbb{C}^{2n+2})$ in type D. As varieties, however, these are already accounted for: 
$\mathbb{P}^{2n-1}$ is a Grassmannian, and
$\OG(n+1, \mathbb{C}^{2n+2})$ is isomorphic to $\OGn$.
The difference in Lie type is important in some contexts, but
largely irrelevant to the work in this paper.}
other infinite minuscule family.
For all minuscule flag varieties, there is a combinatorial Schubert
calculus, based on Young tableaux, which computes cohomology 
and K-theory \cite{bib:ThomasYongMinuscule}.  In the case of 
$\OGn$, Schubert varieties are indexed 
by \newword{shifted partitions} $\mu=(\mu_1>\mu_2>\mu_3>\cdots)$,
with $\mu_1 \leq n$, whose diagram is obtained by shifting
the $i$th row $i$ units to the right.  For instance, the diagram 
for the shifted partition $(6,4,2,1)$ is shown below.
\begin{center}
\includegraphics{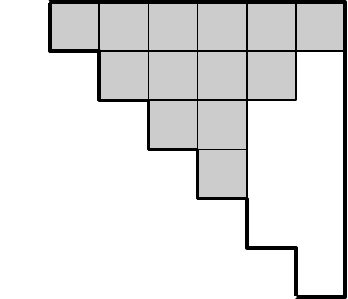}
\end{center}
There is well-established theory of shifted tableaux and Littlewood-Richardson
tableaux for these shapes, analogous to the theory in type A.

Schubert varieties $\Omega_\mu(\mathcal{F})$ in $\OGn$
are only defined with respect to \emph{orthogonal flags}, i.e. flags
for which 
the $i$-th part of the flag is orthogonal to the $(2n+1-i)$-th part for 
all $i$.  Fortunately, there is
a choice of symmetric bilinear form on $\mathbb{C}^{2n+1}$ such that
the flags $\mathcal{F}_t$
are in fact orthogonal flags.  We may therefore consider
Schubert varieties with respect to osculating flags in 
$\OGn$.   As in type A, intersections of such Schubert 
varieties have strong transversality properties:
any zero-dimensional intersection of the
form
\[\Omega_{\mu^{(1)}}(\mathcal{F}_{t_1}) \cap \cdots \cap \Omega_{\mu^{(r)}}(\mathcal{F}_{t_r})
\]
with $t_1 < \dots < t_r$, is a transverse intersection \cite{bib:Pur09},
and there are again connections between the geometric properties of
intersection points and the tableaux that enumerate them \cite{bib:Pur10B}.
It is therefore natural to wonder whether the geometric and combinatorial
properties of one-dimensional intersections 
also extend to $\OGn$. Accordingly, we define \newword{type B Schubert curves} 
\[S = S(\mu^{(1)}, \ldots, \mu^{(r)})
= \Omega_{\mu^{(1)}}(\mathcal{F}_{t_1}) \cap \cdots \cap \Omega_{\mu^{(r)}}(\mathcal{F}_{t_r}),
\]
where the osculation points $t_i$ are real numbers with 
$0 = t_1 < t_2 < \cdots < t_r = \infty$. Here $\mu^{(1)},\ldots,\mu^{(r)}$ 
are shifted partitions for which $\sum |\mu^{(i)}|=\frac{n(n-1)}{2}-1$.  In this paper, we study 
the geometry and associated combinatorics of these curves.

\subsection{Main results}

We describe the basic geometric properties of type B Schubert 
curves in Section~\ref{sec:geometry}.
As expected, the results we obtain here are precisely analogous the those 
in type A, and many of the results are proved either similarly, or 
deduced from the type A results.
We show that a type B Schubert curve $S$ is one-dimensional (if non-empty) 
and reduced.
A key step is the construction of a natural branched covering map $S \to \mathbb{P}^1$, whose fibers are in bijection with chains of skew, shifted \textit{Littlewood-Richardson} tableaux.  We write $\mathrm{LR}(\mu^{(1)},\ldots,\mu^{(r)})$ to denote the set of sequences $(T_1, \ldots, T_r)$ of skew shifted Littlewood-Richardson tableaux, filling the $n\times n$ triangle, such that the shape of $T_i$ extends that of $T_{i-1}$ and $T_i$ has content $\mu^{(i)}$ for all $i$. (The tableaux $T_1$ and $T_r$ are uniquely determined and may be omitted.)

\begin{figure}[t]
\begin{center}
\includegraphics[width=10cm]{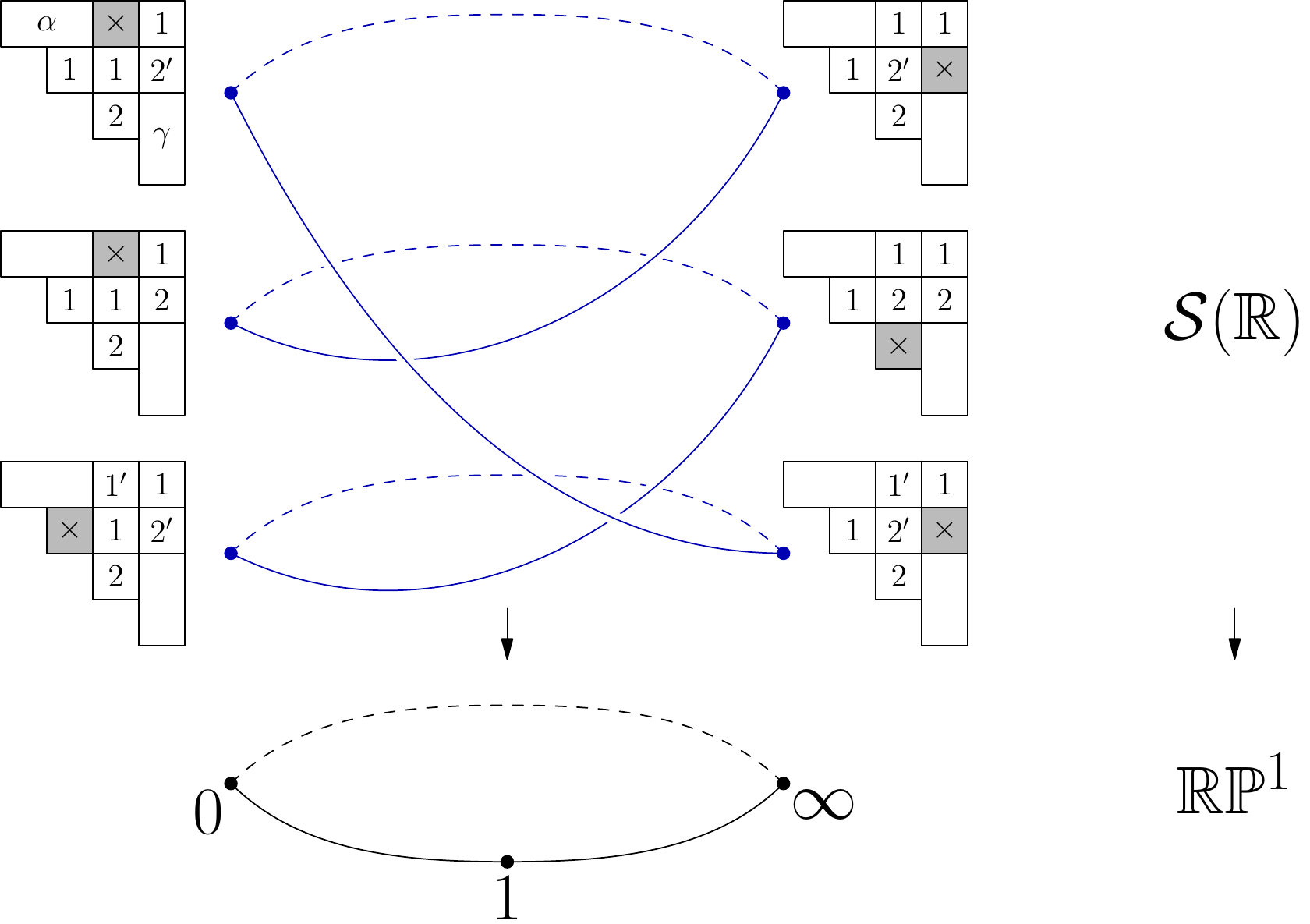}
\end{center}
\caption{\label{fig:covering-space} The covering map $S(\mathbb{R})\to \mathbb{RP}^1$, along with the canonical labeling of the fibers over $0$ and $\infty$.}
\end{figure}

For simplicity, we mainly consider Schubert curves
$S = S(\alpha, \beta, \gamma)$, which are intersections of only three Schubert
varieties, though the results of this paper extend to the general case
without difficulty.

Our main geometric result, which we prove in Section \ref{sec:geometry}, describes the topology of $S(\alpha,\beta,\gamma)(\mathbb{R})$ in terms of shifted Littlewood-Richardson tableaux.

\begin{thm}\label{thm:intro-2}
There is a map $S \to \mathbb{P}^1$ that makes the real locus $S(\mathbb{R})$ a smooth covering of the circle $\mathbb{RP}^1$. The fibers over $0$ and $\infty$ are in canonical bijection with, respectively, $\LRyb$ and $\LRby$. Under this identification, the arcs of $S(\mathbb{R})$ covering $\mathbb{R}_-$ induce the \emph{jeu de taquin bijection}
\[\sh : \LRby \to \LRyb,\]
and the arcs covering $\mathbb{R}_+$ induce a different bijection $\esh$, called \emph{(shifted) evacuation-shuffling}. The monodromy operator $\omega$ is, therefore, given by $\omega = \sh \circ \esh.$
\end{thm}

See Figure \ref{fig:covering-space} for an illustration. This theorem is the type B analog of \cite[Corollary 4.9]{bib:Levinson} in type A. In particular, the (real) topology of $S(\mathbb{R})$ is implicitly described
by the combinatorial operations $\sh$ and $\esh$.

\begin{remark}\label{rmk:modulispace}
The work in \cite{bib:Levinson} also describes stable degenerations of $S(\alpha, \beta, \gamma)$, but since we are primarily concerned with the Schubert curve itself, we take a simpler approach using the third author's work \cite{bib:Pur10} on the geometry of the type B Wronski map $\Wr: \OGn \to \mathbb{P}^{n(n-1)/2}$, restricted to $S \subset \OG$. Analogous results to \cite[Theorems 1.1-1.6]{bib:Speyer} and \cite[Theorem 4.7]{bib:Levinson} do hold for stable degenerations in type B, yielding fiber spaces of orthogonal Schubert problems over $\overline{M}_{0,r}$. In the interest of brevity (and since the techniques are essentially unchanged), we omit these further constructions.
\end{remark}

 While shifted jeu de taquin is very well-studied, the operation $\esh$ is less well understood, particularly in type B. The basic definition is the same as in type A: $\esh$ is (informally) the conjugate of $\sh$ by the operation of rectification, and behaves very differently from $\sh$. The rectification step makes $\esh$ computationally and conceptually difficult to study, so our next objective is to give a ``local'' algorithm for $\esh$, which instead moves the box directly through the tableau via certain local moves. Here the details diverge from the type A story, as we delve into the peculiarities of shifted tableaux.
 
 A key property of $\esh$, manifest for both geometric and combinatorial reasons, is that it commutes with all sequences of jeu de taquin slides. This property is called \emph{coplacticity}, and in type A many coplactic operators are known, notably Kashiwara's crystal operators. For type B, such operators are less well-known, and we rely heavily on results in \cite{bib:GLP}, which establishes a crystal-like structure on shifted tableaux with operators that are coplactic for shifted jeu de taquin. (In fact, those operators were discovered as a result of the preliminary investigations of this paper.) We review these more fundamental coplactic operators in Section \ref{sec:coplactic}, and we initiate our study of $\esh$ by expressing it in terms of them (Theorem \ref{thm:crystal-algorithm}). We then unravel the description further to obtain the desired direct combinatorial algorithm for computing $\esh$, and hence $\omega$, based on local moves 
(Theorem \ref{thm:step-by-step-algorithm}).

In certain respects, our local algorithm resembles the algorithm given in \cite{bib:GillespieLevinson} for type A evacuation-shuffling, such as having a natural division into two `phases' that follow different rules. However, there are important structural differences, which we discuss in Section \ref{subsec:phase1-vs-phase2}, including a negative result (Proposition \ref{prop:phase2-indecomposable}) that prevents an analysis similar to that in \cite{bib:GillespieLevinson}.

Finally, we study the K-theory class $[\mathcal{O}_S] \in \K(\OGn)$, which encodes much of the topological information about the curve 
$S = S(\alpha,\beta,\gamma)$, in particular its degree and holomorphic 
Euler characteristic. It can be computed combinatorially in terms of
\textit{shifted genomic tableaux}, defined by Pechenik and Yong
\cite{bib:Pechenik}.
For Schubert curves in $\Grkn$, there is a strong connection 
between the combinatorial algorithm for $\omega$ and the K-theory class 
$[\mathcal{O}_S]$: the individual steps of the algorithm correspond to
the genomic tableaux that compute this class
\cite{bib:GillespieLevinson}.
Although the combinatorics of shifted genomic tableaux and shifted evacation-shuffling are considerably 
different, we find, surprisingly, that a similar result holds for 
type B Schubert curves.

\begin{thm}\label{thm:intro-ktheory}
There is a two-to-one correspondence between non-adjacent steps of the 
algorithm in Theorem~\ref{thm:step-by-step-algorithm} for computing 
$\omega$ on $S(\alpha, \beta, \gamma)$, and the set of all 
shifted ballot genomic tableaux $\K(\gamma^c/\alpha;\beta)$.
\end{thm}

Informally, each genomic tableau is obtained once by a move in reading order and once in reverse reading order. We state this result more precisely, and discuss some of its geometric consequences, as well as other connections to K-theory in 
Section~\ref{sec:K-theory}. For example, it follows that the map $f: S(\mathbb{R}) \to \mathbb{RP}^1$ cannot be topologically trivial unless the map $S \to \mathbb{P}^1$ is algebraically trivial, that is, $S \cong \bigsqcup_{\deg f} \mathbb{P}^1$. This property is known in type A.

\subsection{Outline}

In Section~\ref{sec:tableaux}, we review combinatorial background
material on shifted tableaux necessary for the rest of the paper.  This
includes the shifted jeu de taquin, dual equivalence and the shifted 
Littlewood-Richardson rule, coplactic structure on shifted tableaux.
We provide background on the orthogonal Grassmannian 
$\OGn$, including facts about Schubert
varieties with respect to the osculating flags $\mathcal{F}_t$
in Section~\ref{sec:geometry}.
We also prove our foundational geometric results about type B Schubert
curves, including Theorem~\ref{thm:intro-2}.
The combinatorial algorithm for computing the monodromy map 
$\omega$ is developed in
Sections~\ref{sec:crystal-algorithm} and~\ref{sec:local-algorithm}.
Section~\ref{sec:crystal-algorithm} provides an version of the algorithm 
in terms of the coplactic operators discussed in 
Section~\ref{sec:coplactic},
and Section~\ref{sec:local-algorithm} then reformulates this, without
the use of these operators.  Since there are steps that are not
immediately clear how to invert, we also discuss the inverse of these 
algorithms.  
Finally, in Section~\ref{sec:K-theory}, we apply these results toward
understanding topological properties of complex type B Schubert curves,
and discuss the connections 
to K-theory of $\OGn$.

%%%%%%%%%%%%%%%%
%%%%%%%%%%%%%%%%
%%%%%%%%%%%%%%%%

\subsection{Acknowledgements}

We thank Oliver Pechenik for helpful discussions on shifted genomic tableau and the K-theory of the orthogonal Grassmannian.  Thanks also to Nic Ford for pointing out an error in a previous version of one of our proofs.  Computations in Sage \cite{sage} were very helpful in conjecturing and verifying many of the results in this paper.

\section{Combinatorial notation and background}
\label{sec:tableaux}

A \newword{partition} is a weakly decreasing sequence $\lambda$ of nonnegative integers $\lambda_1 \geq  \cdots \geq \lambda_k \geq 0$.  Likewise, a \newword{strict partition} $\sigma$ is a strictly decreasing sequence of nonnegative integers.  The \textbf{Young diagram} of a partition $\lambda$ is the left-aligned partial grid of squares with $\lambda_i$ squares in the $i$-th row, while the \textbf{shifted Young diagram} of a strict partition $\sigma$ is defined similarly but aligned along a staircase as shown in Figure \ref{fig:shifted-reflect}.
 We say that $|\lambda|=\sum \lambda_i$ is the \textbf{size} of $\lambda$, and the entries $\lambda_i$ are its \textbf{parts}; likewise for $\sigma$.

We assume throughout that our partitions $\lambda$ fit in an $n \times (n+1)$ rectangle, that is, $\lambda$ has at most $n$ parts and $\lambda_i \leq n+1$ for each $i$.   We write $\lambda \subseteq \rect$. We similarly assume that strict partitions $\sigma$ fit in a height-$n$ staircase, that is, $\sigma_i \leq n+1-i$ for all $i$. We write $\sigma \subseteq \stair$.

Given a strict partition $\sigma$, we define an ordinary partition $\tilde \sigma$ by $\tilde{\sigma}_i = \sigma_i + \#\{j : j \leq i < j + \sigma_j\}$ (Figure \ref{fig:shifted-reflect}), obtained by ``unfolding'' $\sigma$ past the main diagonal. Note that $|\tilde\sigma| = 2|\sigma|$. If $\lambda = \tilde \sigma$ for some strict partition, we say $\lambda$ is {\bf symmetric}. If not, we define the {\bf symmetrization} of $\lambda$ as the smallest symmetric partition containing $\lambda$.

\begin{figure}
\begin{center}
  \includegraphics{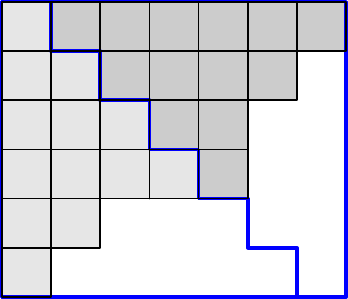}
  \end{center}
\caption{\label{fig:shifted-reflect} A strict partition $\sigma=(6,4,2,1)$ shown above the diagonal. Its associated symmetric partition $\tilde\sigma=(7,6,5,5,2,1)$ is the union of the two shaded regions shown.}
\end{figure}

A \textbf{standard Young tableau} of partition shape $\lambda$ or shifted shape $\sigma$ is a filling of the boxes of its diagram with the numbers $1,\ldots,n$, where $n$ is the size of the partition, such that the entries are increasing across rows and down columns.  We write $\mathrm{SYT}(\lambda)$ or $\mathrm{SYT}(\sigma)$ to denote the set of all standard Young tableaux of shape $\lambda$ (resp.\ $\sigma$).

\subsection{Canonical form and representatives}

Let $w = w_1w_2 \dots w_n$ be a string in symbols $\{1',1,2',2,3',3,\ldots\}$. 
\begin{definition}
  Let $w$ be a string in symbols $\{1',1,2',2,3',3,\ldots\}$.  The \textit{first $i$ or $i'$} of $w$ is the leftmost entry which is either equal to $i$ or $i'$. We denote this entry by $\first(i,w)$ or $\first(i',w)$, whichever is more convenient; we emphasize that both refer to the same entry. When $w$ is clear from context we suppress it and write $\first(i)$ or $\first(i')$.
The \textbf{canonical form} of $w$ is the string formed by replacing $\first(i,w)$ (if it exists) with $i$ for all $i\in \{1,2,3,\ldots\}$.  We say two strings $w$ and $v$ are \textbf{equivalent} if they have the same canonical form; note that this is an equivalence relation.
\end{definition}

\begin{definition}
 A \textbf{word} is an equivalence class $\hat{w}$ of the strings $v$ equivalent to $w$.  If $w$ is in canonical form, we say that $w$ is the \textbf{canonical representative} of the word $\hat{w}$.  We often call the other words in $\hat{w}$ \textbf{representatives} of $\hat{w}$ or of $w$.  The {\bf weight} of $w$ is the vector $\mathrm{wt}(w) = (n_1, n_2, \ldots ),$ where $n_i$ is the total number of $(i)$s and $(i')$s in $w$.
\end{definition}

\begin{example}
  The canonical form of the word $1'1'2'112'$ is $11'2112'$.  The set of all representatives of $11'2112'$ is $\{1'1'2'112',11'2'112',1'1'2112',11'2112'\}$.
\end{example}

All the enumerative results in this setting count (canonical forms of) words and tableaux. In later sections, however, we will find that some of the algorithms require using different representatives.

\subsection{Skew shapes, semistandard tableaux, and jeu de taquin}

A \textbf{(shifted) skew shape} is a difference $\sigma/\rho$ of two partition diagrams, formed by removing the squares of $\rho$ from the diagram of $\sigma$, if $\rho$ is contained in $\sigma$ (written $\rho\subseteq \sigma$).  For instance, in Figure \ref{fig:ssyt} the shape shown is $(6,5,2,1)/(3,2)$. If $\rho = \eset$, we may refer to the shape as a \newword{(shifted) straight shape} for emphasis.

\begin{figure}[b]
  \begin{center}
    \includegraphics{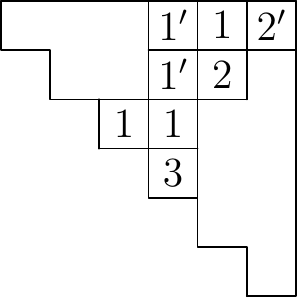}
  \end{center}
\caption{\label{fig:ssyt} A skew shifted semistandard Young tableau in canonical form.}
\end{figure}
  
  A \textbf{shifted semistandard Young tableau} (shifted SSYT) is a filling of the boxes with entries from the alphabet $\{1'<1<2'<2<3'<3<\cdots \}$ such that the entries are weakly increasing down columns and across rows, and such that primed entries can only repeat in columns, and unprimed only in rows. The \textbf{(row) reading word} of such a tableau is the word formed by concatenating the rows from bottom to top (in Figure \ref{fig:ssyt}, the reading word is $3111'21'12'$).  The {\bf weight} of $T$ is the vector $\mathrm{wt}(T) = (n_1, n_2, \ldots)$, where $n_i$ is the total number of $(i)$s and $(i')$s in $T$.

We use the notions of inner and outer \textbf{jeu de taquin} (JDT) slides defined by Sagan and Worley \cite{Sagan,Worley}.  Inner and outer slides are defined as usual (see e.g. \cite{bib:GLP} for a more detailed description) but with two exceptions to the sliding rules: if an outer slide moves an $i$ down into the diagonal and then another $i$ to the right on top of it, that $i$ becomes primed (and vice versa for the corresponding inner slide), as shown below.
\begin{center}
 \includegraphics{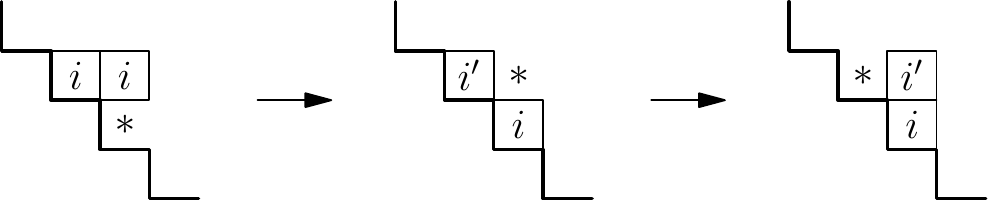}
\end{center}
Similarly, if an outer slide moves an $i$ down into the diagonal, then moves an $i'$ to the right on top of it, the $i$ becomes primed.
\begin{center}
 \includegraphics{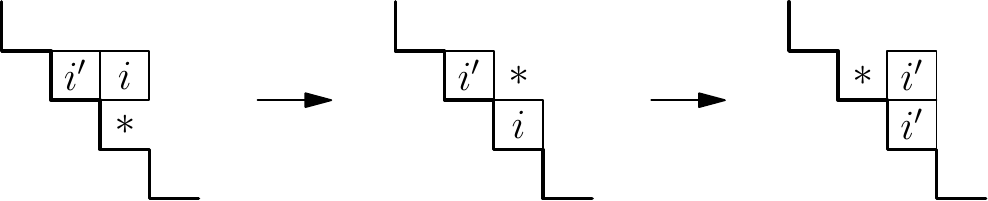}
\end{center}

Note that a tableau in canonical form remains in canonical form after applying any JDT slide.

We write $\rectify(T)$ or $\rectify(w)$ to denote the jeu de taquin \textbf{rectification} of any shifted semistandard tableau $T$ with reading word $w$.  We say that $T$ is {\bf Littlewood-Richardson}, and $w$ is {\bf ballot}, if for every $i$, the $i$-th row of $\rectify(T)$ consists entirely of $(i)$s.

\begin{definition}
  The \textbf{standardization} of a word $w$ (or tableau $T$) is the word $\mathrm{std}(w)$ (or tableau $\mathrm{std}(T)$) formed by replacing the letters in order with $1,2,\ldots,n$ from least to greatest, breaking ties by reading order for unprimed entries and by reverse reading order for primed entries. The resulting total order on the boxes is called \textbf{standardization order}.  We use the symbols $\ltst$ and $\leqst$ to compare letters in this ordering.
\end{definition}

\begin{example}
  The standardization of the word $1121'22'1'11$ is the word $348297156$.
\end{example}

The following is well-known (see \cite{bib:GLP} for a proof).

\begin{proposition}\label{prop:standardization}
  A filling $T$ of a shifted skew shape is semistandard (and not necessarily canonical) if and only if $\mathrm{std}(T)$ is a standard shifted tableau.
\end{proposition}

\subsection{Dual equivalence and Littlewood-Richardson tableaux}

Two \textit{standard} skew shifted tableaux of the same shape are said to be (shifted) \textbf{dual equivalent} if their shapes transform the same way under any sequence of jeu de taquin slides. We extend this notion to \textit{semistandard} (shifted) tableaux by defining two tableaux to be \textbf{dual equivalent} if and only if their standardizations are, as in \cite{SaganBook}. The word `dual' refers to the recording tableau under the shifted Schensted correspondence \cite{Sagan, Worley}. See \cite{bib:Haiman} for more in-depth discussions of dual equivalence.

A \textbf{dual equivalence class} is an equivalence class of Young tableaux under dual equivalence.  Given two dual equivalence classes $D,D'$ of skew shapes $\lambda/\mu$ and $\rho/\lambda$ that extend one another (note $\lambda$ is in common to the shapes' constructions), we can perform \textbf{tableau switching} as follows.  Choose any standard representatives $T,T'$ of $D$ and $D'$, and perform successive inner jeu de taquin slides on $T'$ in the order specified by the labels on the squares of $T$ from largest to smallest (recall that $T$ and $T'$ are both standard Young tableaux, determining a total ordering on their squares).  Let $\tilde{T'}$ be the resulting tableau and let $\tilde{T}$ be the tableau formed by the empty squares after these slides, labeled in reverse order.  Then the dual equivalence classes $\tilde{D'}$ and $\tilde{D}$ of $\tilde{T'}$ and $\tilde{T}$ are independent of the choices of representatives $T$ and $T'$, and we say $(\tilde{D'},\tilde{D})$ is the \textbf{switch} of $(D,D')$.

\begin{remark}
  The tableau switching above can equivalently be defined by performing successive \textit{outer} jeu de taquin slides on $T'$ according to the ordering on the squares of $T$ (see \cite{bib:Haiman}).
\end{remark}

By the above facts, we may speak of the \newword{rectification shape of a dual equivalence class} $\rsh(D)$.  This is the shape of any rectification of any representative of the class $D$.  It is well-known (see \cite{bib:Haiman} or \cite{bib:GillespieLevinson}) that the dual equivalence classes of a given skew shape and rectification shape (in both the shifted and unshifted settings) are counted by a Littlewood-Richardson coefficient.

\begin{lemdef} \label{lem:dual-LRcoeff}
Let $\lambda/\mu$ be a (shifted or unshifted) skew shape and let 
\[\DE_\mu^\lambda(\beta) = \{\text{dual equivalence classes } D \text{ with } \sh(D) = \lambda/\mu \text{ and } \rsh(D) = \beta \}.\]
Then in the unshifted setting, $|\DE_\mu^\lambda(\beta)| = c_{\mu \beta}^\lambda$, and in the shifted setting, $|\DE_\mu^\lambda(\beta)| = f_{\mu \beta}^\lambda.$
\end{lemdef}

Recall (see, e.g., \cite{bib:Fulton}) that the Littlewood-Richardson coefficient $c_{\mu\beta}^\lambda$ is also the coefficient of the Schur function $s_\lambda$ in the Schur expansion of the product $s_\mu s_\beta$, and is the structure coefficient of the Schubert class $[X_\lambda]$ in the product $[X_\mu] \cdot [X_\beta]$ in the cohomology ring $H^\ast(\Grkn)$.  In the shifted setting, the Littlewood-Richardson coefficient $f_{\mu\beta}^\lambda$ is the coefficient of the Schur $P$-function $P_\lambda$ in the product $P_\mu P_\beta$, or equivalently the coefficient of the Schur $Q$-function $Q_{\beta}$ in the skew Schur $Q$-function $Q_{\lambda/\mu}$.  (For more details on Schur $P$ and $Q$-functions, see \cite{Stembridge} or \cite{bib:GLP}.)  Finally, the Littlewood-Richardson coefficients $f_{\mu\beta}^\lambda$ are also the structure constants in the cohomology ring of the orthogonal Grassmannian, $H^\ast(\OGn)$. It is also convenient to define the generalized Littlewood-Richardson coefficient $c_{\mu^{\bullet}}^\lambda$ (resp. $f_{\mu^{\bullet}}^\lambda$) as the coefficient of $[X_\lambda]$ in $[X_{\mu^{(1)}}] \cdots [X_{\mu^{(r)}}]$ (resp., the coefficient of $[\Omega_\lambda]$ in $[\Omega_{\mu^{(1)}}] \cdots [\Omega_{\mu^{(r)}}]$).

\subsubsection{Connection to Littlewood-Richardson tableaux}

For a straight shape $\beta$, the \textbf{highest-weight} tableau of shape $\beta$ is the tableau having its $i$-th row filled entirely with the letter $i$ (in both the shifted and unshifted setting).  For skew shapes, a dual equivalence class $D$ of rectification shape $\beta$ has a unique \newword{highest-weight representative}, that is, the unique tableau $T$ dual equivalent to $D$ that rectifies to the highest-weight tableau of shape $\beta$.  Note also that if $S,T$ are highest-weight skew tableaux with $T$ extending $S$, then the tableaux formed by switching them, $(T',S')$, are also highest-weight, since switching does not change their rectification. 

We can therefore work entirely with Littlewood-Richardson tableaux in place of their dual equivalence classes for the purposes of tableaux switching and enumeration of Littlewood-Richardson coefficients. 

\begin{definition}
  For a (shifted or unshifted) skew shape $\lambda/\mu$ we define
\[\LR_\mu^\lambda(\beta) = \{\text{Littlewood-Richardson tableaux } T \text{ with } \sh(T) = \lambda/\mu \text{ and } \rsh(T) = \beta \}.\]
\end{definition}

Note that we have a natural bijection $$\DE_\mu^\lambda(\beta)\leftrightarrow \LR_\mu^\lambda(\beta).$$
A tableau is Littlewood-Richardson if and only if its reading word is \textbf{ballot}, a local condition defined in \cite{Stembridge} in the shifted case.  Here we \textit{define} a reading word to be ballot if its tableau is Littlewood-Richardson, and we discuss various ballotness criterion in Section \ref{sec:coplactic} below.

\subsection{Chains of tableaux}

A \textbf{chain} of (shifted) skew shapes is a sequence of shapes of the form $$\lambda^{(2)}/\lambda^{(1)},\,\, \lambda^{(3)}/\lambda^{(2)},\,\,\ldots,\,\, \lambda^{(r+1)}/\lambda^{(r)}$$
for some nested sequence of (shifted) partitions $$\lambda^{(1)}\subseteq \lambda^{(2)}\subseteq \cdots \subseteq \lambda^{(r+1)}.$$
We say that each skew shape $\lambda^{(i+1)}/\lambda^{(i)}$ \textbf{extends} the previous shape $\lambda^{(i)}/\lambda^{(i-1)}$ in the chain.  We define a \textit{chain of dual equivalence classes} or \textit{chain of Littlewood-Richardson tableaux} to be an assignment of a class or tableau to each shape in a chain of shapes.  We write 
$$\DE(\alpha^{(1)},\alpha^{(2)},\ldots,\alpha^{(n)})\hspace{0.5cm}\text{and}\hspace{0.5cm} \LR(\alpha^{(1)},\alpha^{(2)},\ldots,\alpha^{(n)}),$$ respectively, to denote the set of all chains of dual equivalence classes (respectively, Littlewood-Richardson tableaux) whose rectification shapes (respectively, weights) are given by the tuples $\alpha^{(1)},\ldots,\alpha^{(n)}$ in order.  Note that the weights $\alpha^{(i)}$ generally do not uniquely determine the shapes $\lambda^{(i)}$.

\subsection{Coplactic operators on shifted tableaux}\label{sec:coplactic}

We now briefly review the combinatorial notation introduced in detail in \cite{bib:GLP}, regarding the weight raising and lowering operators $E_i',E_i,F_i',F_i$ on shifted SSYT's.  These operators have the property of being compatible with jeu de taquin:

\begin{definition}
An operation on canonical shifted tableaux (or on their reading words) is \textbf{coplactic} if it commutes with all shifted jeu de taquin slides.
\end{definition}

When we study the geometry of the Schubert curve $S(\alpha, \beta, \gamma)$, we will encounter certain coplactic operations on tableaux. In Section \ref{sec:crystal-algorithm}, we characterize those operations in terms of the natural operators $E_i',E_i,F_i',F_i$.

Throughout this section, we consider words consisting only of the letters $\{1',1,2',2\}$, and define only the operators $E'_1,E_1,F_1',F_1$.  For general words $w$, $E'_i$ and $F'_i$ are defined on the subword containing the letters $\{i',i,i+1',i+1\}$, treating $i$ as $1$ and $i+1$ as $2$.  To further simplify our notation in this subsection we use the following shorthands.

\begin{definition}
Define $E' = E'_1$, $F' = F'_1$, $E=E_1$, $F=F_1$.
\end{definition}

\subsubsection{Primed operators}

\begin{definition} \label{def:primed-operators}
We define $E'(w)$ to be the unique word such that
\[\mathrm{std}(E'(w)) = \mathrm{std}(w) \hspace{0.5cm}\text{ and }\hspace{0.5cm} \mathrm{wt}(E'(w)) = \mathrm{wt}(w) + (1,-1).\]
if such a word exists; otherwise, $E'(w) = \varnothing$. We define $F'(w)$ analogously using $-(1,-1)$.
\end{definition}

\begin{proposition} \label{prop:primed-properties}
  The maps $E'$ and $F'$ have the following properties.  \cite{bib:GLP}
 \begin{enumerate}
   \item \label{prop:primed-partial-inverses} They are partial inverses of each other, that is, $E'(w) = v$ if and only if $w = F'(v)$.
   \item \label{prop:primed-semistandard}
The maps $E'$ and $F'$ are well-defined on skew shifted semistandard tableaux.
   \item  \label{prop:primed-coplactic}
The operations are coplactic, that is, they commute with all jeu de taquin slides.
   \item \label{prop:explicit-definitions-primed}
To compute $F'(w)$, consider all representatives of $w$. If all representatives have the property that the last $1$ is left of the last $2'$ then $F'(w) = \varnothing$.  If there exists a representative such that the last $1$ is right of the last $2'$ then $F'(w)$ is obtained, using this representative, by changing the last $1$ to a $2'$.

The word $E'(w)$ is defined similarly with the roles of $1$ and $2'$ reversed: if the last $2'$ is right of the last $1$ in some representative, change it to a $1$ (in that representative). Otherwise $E'(w) = \varnothing$.
  \end{enumerate}
\end{proposition}

\subsubsection{Lattice walks}

The first step in defining the more involved operators $F(w)$ and $E(w)$ is to associate, to each word $w$, a lattice walk in the first quadrant of the plane.  This walk is a sequence of points in $\mathbb{N} \times \mathbb{N}$, starting with $(0,0)$.  We specify the walk by assigning a step to each $w_i$, $i=1, \dots, n$. This step will be one of the four principal direction vectors:
\[
\east{~~} \ =\  (1,0)
\qquad \west{~~} \ =\  (-1,0)
\qquad \north{} \ =\  (0,1)
\qquad \south{} \ =\  (0,-1)\,.
\]
The $i$th point $(x_i,y_i)$ is the sum of the steps assigned to $w_1, \dots, w_i$. We define the walk inductively: suppose $i >0$, and we have assigned steps to $w_1, \dots, w_{i-1}$.  We assign the step to $w_i$ according to
Figure \ref{fig:directions}, with two cases based on whether or not the step from $(x_{i-1},y_{i-1})$ starts on one of the $x$ or $y$ axes.  We will generally write the label each step of the walk by the letter $w_i$, so as to represent both the word and its walk on the same diagram.

\begin{figure}
	    \begin{center}
	    	\begin{tabular}{|c|cccc|}
	    	\hline
	    	$x_iy_i=0$ & \east{1'} & \east{1} & \north{2'} & \north{2}   \\[.8ex]
	    	\hline
	    	 $x_iy_i\neq0$ &\east{1'} & \south{1} & \west{2'} & \north{2}   \\[.8ex]
	    	\hline
	        \end{tabular}
	    \end{center}
 \caption{\label{fig:directions} The directions assigned to each of the letters $w_i=1'$, $1$, $2'$, or $2$ according to whether the location of the walk just before $w_i$ starts on the axes ($x_iy_i=0$) or not.}
\end{figure}

\begin{example}
\label{ex:walk}
Here is the walk for $w = 1221'1'111'1'2'2222'2'11'1$.

\begin{center}
\begin{picture}(5,4.2)(0,0)
\multiput(0,0)(0,0.2){22}{\line(0,1){0.1}}
\multiput(0,0)(0.2,0){27}{\line(1,0){0.1}}
\put(0,0){\circle*{0.13}}
\put(4,2){\circle{0.13}}
\stepeast{1}{%
\stepnorth{2}{%
\stepnorth{2}{%
\stepeast{1'}{%
\stepeast{1'}{%
\stepsouth{1}{%
\stepsouth{1}{%
\stepeast{1'}{%
\stepeast{1'}{%
\stepnorth{2'}{%
\stepnorth{2}{%
\stepnorth{2}{%
\stepnorth{2}{%
\stepwest{2'}{%
\stepwest{2'}{%
\stepsouth{1}{%
\stepeast{1'}{%
\stepsouth{1}{%
}}}}}}}}}}}}}}}}}}
\end{picture}
\end{center}
\end{example}

%%%%%%%%%%%%%%%%
\subsubsection{Critical substrings and definition of $E$ and $F$}\label{sec:EFdefinition}

\begin{definition}
If $w$ is a word and $u = w_k w_{k+1} \dots w_l$ is a substring of some representative of $w$, we say $u$ is a {\bf substring} of the word $w$. We say $(x,y) = (x_{k-1},y_{k-1})$ is the \defn{location} of $u$ in the walk of $w$.
\end{definition}

\begin{definition} We say that $u$ is an \defn{$F$-critical substring} if certain conditions on $u$ and its location are met. There are five types of $F$-critical substring. Each row of the first table in Figure \ref{fig:criticals} describes one type, and a transformation that can be performed on that type.
\end{definition} 

\begin{figure}
\begin{center}
\begin{tabular}{|c|c|c|c|c|}
\hline
\multirow{2}{*}{Type} 
& \multicolumn{3}{c|}{Conditions} & \multirow{2}{*}{Transformation}  \\
     & \multicolumn{1}{c}{Substring}&  \multicolumn{1}{c}{Steps} & Location &  
\\\hline
\hline
 \multirow{2}{*}{1F} & 
\multirow{2}{*}{$u = 1(1')^*2'$} &
\east{1} ~ \east{1'} ~ \north{2'} & 
$y=0$ & 
\multirow{2}{*}{$u \to 2'(1')^*2$} \\[.5ex]\cline{3-4}
& &  
\south{1} ~ \east{1'} ~ \north{2'} & 
$y=1$, $x \geq 1$ & 
\\[.5ex]\hline
\multirow{2}{*}{2F} &
\multirow{2}{*}{$u = 1(2)^*1'$} &
\east{1} ~ \north{2} ~ \east{1'} & 
$x = 0$ & 
\multirow{2}{*}{$u \to 2'(2)^*1$}  \\[.5ex]\cline{3-4}
&& 
\south{1} ~ \north{2} ~ \east{1'} &
$x = 1$, $y \geq 1$ & 
\\[.5ex]\hline
 3F & $u = 1$ & 
\east{1} & 
$y = 0$ & 
$u \to 2$ 
\\\hline
 4F & 
$u = 1'$ & 
\east{1'} & 
$x  = 0$ & 
$u \to 2'$ 
\\\hline
\multirow{2}{*}{5F} & $u = 1$ 
& \south{1}
& \multirow{2}{*}{$x=1$, $y \geq 1$} 
&
\multirow{2}{*}{undefined} \\[.5ex]\cline{2-3}
& $u=2'$ & \west{2'} &&
\\\hline
\end{tabular}
\end{center}
\vspace{0.5cm}

\begin{center}
\begin{tabular}{|c|c|c|c|c|}
\hline
\multirow{2}{*}{Type} 
& \multicolumn{3}{c|}{Conditions} & \multirow{2}{*}{Transformation}  \\
     & \multicolumn{1}{c}{Substring}&  \multicolumn{1}{c}{Steps} & Location &  
\\\hline
\hline
 \multirow{2}{*}{1E} & 
\multirow{2}{*}{$u = 2'(2)^*1$} &
\north{2'} ~ \north{2} ~ \east{1} & 
$x=0$ & 
\multirow{2}{*}{$u \to 1(2)^*1'$} \\[.5ex]\cline{3-4}
& &  
\west{2'} ~ \north{2} ~ \east{1} & 
$x=1$, $y \geq 1$ & 
\\[.5ex]\hline
\multirow{2}{*}{2E} &
\multirow{2}{*}{$u = 2'(1')^*2$} &
\north{2'} ~ \east{1'} ~ \north{2} & 
$y = 0$ & 
\multirow{2}{*}{$u \to 1(1')^*2'$}  \\[.5ex]\cline{3-4}
&& 
\west{2'} ~ \east{1'} ~ \north{2} &
$y = 1$, $x \geq 1$ & 
\\[.5ex]\hline
 3E & $u = 2'$ & 
\north{2'} & 
$x = 0$ & 
$u \to 1'$ 
\\[.5ex]\hline
 4E & 
$u = 2$ & 
\north{2} & 
$y  = 0$ & 
$u \to 1$ 
\\[.5ex]\hline
\multirow{2}{*}{5E} & $u = 1$ 
& \south{1}
& \multirow{2}{*}{$y=1$, $x \geq 1$} 
&
\multirow{2}{*}{undefined} \\[.5ex]\cline{2-3}
& $u=2'$ & \west{2'} &&
\\\hline
\end{tabular}
\end{center}

\caption{\label{fig:criticals} Above, the table of $F$-critical substrings and their transformations.  Below, the table of $E$-critical substrings and their transformations.  Here $a(b)^*c$ means any string of the form $abb \dots bc$, including $ac$, $abc$, $abbc$, etc.}
\end{figure}

We define the {\bf final} $F$-critical substring $u$ of $w$ to be the $F$-critical substring with the highest (rightmost) possible ending index. 

\begin{definition}\label{def:F}
We define the word $F(w)$ as follows. We fix a representative $v$ containing the \textbf{final} critical substring $u$, and transform $u$ (in $v$) according to its type (and then canonicalize the result if necessary). If the type is 5F, or if $w$ has no $F$-critical substrings, then $F(w)$ is undefined and we write $F(w) = \varnothing$.

We define $E$ similarly using the corresponding notion of $E$-critical substrings, and transformation rules for each, given by the second table in Figure \ref{fig:criticals}. 
\end{definition}

The relevant properties of $E, F$ are as follows.

\begin{proposition}\label{prop:unprimed-properties}
   The operators $E$ and $F$ satisfy the following properties.
   \begin{enumerate}
     \item They are partial inverses of each other, that is, $E(w)=v$ if and only if $w=F(v)$.
     \item They are well-defined on skew shifted semistandard tableaux (by applying the maps to the reading word).
     \item The maps $E$ and $F$ are coplactic, that is, they commute with all jeu de taquin slides.
     \item For any fixed $i$, the operators $F_i,F_i',E_i,E_i'$ commute with each other when defined.
   \end{enumerate}
\end{proposition}
See \cite{bib:GLP} for examples and further discussion regarding these operators.

\subsubsection{Ballotness criteria}

The lattice walks and coplactic operators give rise to criteria for a tableau to be Littlewood-Richardson, or equivalently for its word to be ballot.

\begin{thm} \label{cor:ballot-walk-criterion} 
A word $w$ is ballot if and only if either of the following equivalent
criteria hold.
\begin{enumerate}
\item
For all $i$, the lattice walk of the subword $w_i$ consisting of the letters $i,i',i+1,i+1'$ has $y_n = 0$, that is, ends on the $x$-axis.
\item
$E_i(w)=E_i'(w)=\varnothing$ for all $i$.
\end{enumerate}
\end{thm}

We will also need Stembridge's original definition 
ballotness \cite{Stembridge}, which we recall here as a lemma.

\begin{lemma}[Stembridge] \label{lem:stembridge}
   Let $w=w_1\cdots w_n$ be a word. For each $i$ and for $1\le j\le 2n$, define $$m_i(j)=\begin{cases}\#\{k\ge n-j+1: w_k=i\} & \text{if }j\le n \\ m_i(n)+\#\{k\le j-n: w_{k}=i'\} & \text{if }j\ge n+1\end{cases}.$$
   Then $w$ is ballot if and only if the following two conditions hold for all $i$:
   \begin{enumerate}
     \item[\textbf{S1.}] If $j< n$ and $m_i(j)=m_{i+1}(j)$, then $w_{n-j}$ is not equal to $i+1$ or $(i+1)'$,
     \item[\textbf{S2.}] If $j\ge n$ and $m_i(j)=m_{i+1}(j)$, then $w_{j-n+1}$ is not equal to $i$ or $(i+1)'$.
   \end{enumerate}
\end{lemma}

Note that the quantity $m_i(j)$ can be computed by reading through the word once backwards and then once forwards.  We increment $m_i(j)$ if the $j$-th letter we read in this way is $i$ on the first pass, or $i'$ on the second pass.   It is also evident from the definition that $m_i(j)\ge m_{i+1}(j)$ for all $i$ and $j$ in a ballot word.

An immediate consequence of this is the following sufficient condition,
which we will use more often.

\begin{corollary}\label{cor:ballotness-preserving}
Let $w$ be a ballot word in letters $1',1,2',2$.  
Any word obtained from $w$ by moving any $1'$ to a position earlier in the word is ballot.
The word obtained by moving any $1$ to the end of $w$ is also ballot.
\end{corollary}

%%%%%%%%%%%%%%
%% GEOMETRY %%
%%%%%%%%%%%%%%

\section{Geometry of type B Schubert curves}\label{sec:geometry}

We now develop the main geometric results, using the Wronski map studied extensively in \cite{bib:Pur10} and \cite{bib:Pur10B}.

\subsection{Ordinary and orthogonal Grassmannians}\label{sec:sym}

If $W$ is a vector space, we write $\Gr(k,W)$ for the Grassmannian of $k$-dimensional subspaces of $W$. Given a partition $\lambda$ and a flag $\mathscr{F}$, we write the corresponding Schubert variety as
\[X_\lambda(\mathscr{F}) := \{U \in \Gr(k,W) : \dim U \cap \mathscr{F}_{n-k+i-\lambda_i} \geq i \text{ for all } i\},\]
which has codimension $|\lambda|$.

If $W$ has dimension $2n+1$ and is equipped with a nondegenerate symmetric bilinear form $\langle-,-\rangle$, a subspace $U \subset W$ is {\bf isotropic} if $\langle u, u' \rangle = 0$ for all $u, u' \in W$. We write $\OG(n,W) \subset \Gr(n,W)$ for the orthogonal Grassmannian of maximal isotropic subspaces. A flag $\mathscr{F}$ is {\bf orthogonal} if $\langle u, u' \rangle = 0$ whenever $u \in \mathscr{F}_i$ and $u' \in \mathscr{F}_{2n+1-i}$ for some $i$. (The lower half of such a flag consists of isotropic subspaces, while the upper half consists of their duals under $\langle -, - \rangle$.) Given a strict partition $\sigma$ and an orthogonal flag $\mathscr{F}$, we define the orthogonal Schubert variety:
\begin{align*}
\Omega_\sigma(\mathscr{F}) &:= X_{\tilde\sigma}(\mathscr{F}) \cap \OG(n,W) \\
&= \{U \in \OG(n,W) : \dim U \cap \mathscr{F}_{n+1+i-\tilde\sigma_i} \geq i \text{ for all } i\}.
\end{align*}
This has codimension $|\sigma|$ in $\OG(n,W)$.

In what follows, we let $V \cong \mathbb{C}^2$ be a two-dimensional vector space. We will describe several $GL_2$-equivariant constructions on $\mathbb{P}^1 = \mathbb{P}(V)$ and related spaces. We note that these constructions are coordinate-free on $\mathbb{P}^1$ and can be extended to $\overline{M}_{0,r}$ (see Remark \ref{rmk:modulispace}).

\subsubsection{A coordinate-free bilinear form on $\mathrm{Sym}^d(\mathbb{C}^2)$} \label{sec:bilinear}

There is a unique (up to scaling) $\mathrm{GL}(V)$-equivariant bilinear form
\begin{align} \label{eqn:bilinear-form}
\langle-,-\rangle &: \mathrm{Sym}^d(V) \otimes \mathrm{Sym}^d(V) \to \det(V)^{\otimes d}, \\
\nonumber \langle v_1 \cdots v_d, w_1 \cdots w_d \rangle &= \frac{1}{d!}\sum_{\pi \in S_d} (v_1 \wedge w_{\pi(1)})\tensor \cdots \tensor (v_d \wedge w_{\pi(d)}).
\end{align}
This form is symmetric if $d$ is even and alternating if $d$ is odd. It follows from $\mathrm{GL}(V)$-equivariance that the form is nondegenerate for $d > 1$. Uniqueness follows from the Pieri rule in $\mathrm{GL}$-representation theory; the $\mathrm{GL}(V)$ representation $\det(V)^{\otimes d}=\bigwedge^2(V)^{\otimes d}$ is the irreducible representation associated to the partition $(d,d)$, and each copy of $\mathrm{Sym}^d(V)$ is an irreducible representation with partition $(d)$.  There is therefore only one copy of $\det(V)^{\otimes d}$ in  $\mathrm{Sym}^d(V)\otimes \mathrm{Sym}^d(V)$.  (See, for instance, \cite{bib:Fulton}, Chapter 8 for details on the Pieri rule.)

In coordinates, we may think of $\mathrm{Sym}^d(\mathbb{C}^2)$ as the vector space $\mathbb{C}[z]_{\leq d}$ of degree-at-most-$d$ polynomials. Then the form is given by
\[\langle z^a, z^b \rangle = \begin{cases} 0 & \text{if } a+b \ne d, \\ \frac{(-1)^b}{\binom{d}{b}} & \text{if } a+b = d.\end{cases}\]
Note that this is a scalar multiple (by a factor of $1/d!$) of the form described in \cite{bib:Pur10B}.

\begin{remark}
From now on, we write $\Gr := \Gr(k,\mathrm{Sym}^d(\mathbb{C}^2))$ and $\OG := \OG(n,\mathrm{Sym}^{2n}(\mathbb{C}^2)).$
\end{remark}

\subsubsection{Osculating flags and Schubert varieties}\label{sec:isotropic}

We now define the \emph{osculating flag} $\mathscr{F}(p) \subset \mathrm{Sym}^d(\mathbb{C}^2)$ for $p \in \mathbb{P}^1$. In terms of representations, $p$ gives a line $\langle v \rangle \subset \mathbb{C}^2$, and the codimension-$i$ part of $\mathscr{F}(p)$ is 
\[\mathscr{F}(p)^i = \{u \in \mathrm{Sym}^d(\mathbb{C}^2) : u = v^i w \text{ for some } w \in \mathrm{Sym}^{d-i}(\mathbb{C}^2)\}.\]
In terms of polynomials, $\mathscr{F}(p)^i \subset \mathbb{C}[z]_{\leq d}$ consists of the polynomials vanishing at $p$ to order at least $i$, that is, $(z-p)^i$ divides $f(z)$. (If $p = \infty$, we interpret this as meaning $\deg(f) \leq d-i$.) 

\begin{proposition}
The flags $\mathscr{F}(p)$ are orthogonal with respect to the form $\langle-,-\rangle$.
\end{proposition}
\begin{proof}
Let $u \in \mathscr{F}(p)^a$ and $\tilde u \in \mathscr{F}(p)^b$. If $p$ corresponds to $\langle v \rangle \subset \mathbb{C}^2$, we have factorizations $u = v^a w$ and $u' = v^b \tilde w$, where $w, \tilde w$ have degrees $d-a$ and $d-b$.

Then, if $a+b > d$, we see directly (regardless of what $w, \tilde w$ are) that in the sum for $\langle u,\tilde u\rangle$, every choice of $\pi \in S_d$ gives at least one factor $(v \wedge v) = 0$. Therefore $\langle u, \tilde u\rangle = 0$.
\end{proof}

We will only consider Schubert conditions in $\Gr$ and $\OG$ using osculating flags. Thus, for $p \in \mathbb{P}^1$, we put
\begin{equation}
X_\lambda(p) := X_\lambda(\mathscr{F}(p)), \qquad \Omega_\sigma(p) := \Omega_\sigma(\mathscr{F}(p)),
\end{equation}
for any partition $\lambda$ or strict partition $\sigma$.

These Schubert varieties have unusually strong transversality properties. The key facts about them are as follows. Let $\lambda^{(1)}, \ldots, \lambda^{(r)}$ be partitions and $\sum |\lambda^{(i)}| = \dim(\Gr)$. Let $p_1, \ldots, p_r \in \mathbb{P}^1$ be distinct points. Consider the intersection
\[X = X_{\lambda^{(1)}}(p_1) \cap \cdots \cap X_{\lambda^{(r)}}(p_r).\]

\begin{thm}[\cite{bib:EH86},\cite{bib:Pur09}] \label{thm:eisenbud-harris}
The intersection $X$ is dimensionally transverse; that is, $\dim(X) = 0$ and the multiplicity of $X$ is given by the corresponding product of Schubert classes, $c_{\lambda_\bullet}^\rect$.

The analogous statement holds in $\OG$, replacing partitions $\lambda$ by strict partitions $\sigma$ such that $\sum |\sigma^{(i)}| = \dim(\OG)$, and replacing Schubert varieties $X_\lambda(p)$ by $\Omega_\sigma(p)$ and $c_{\lambda_\bullet}^{\rect}$ by $f_{\sigma_\bullet}^{\stair}$.
\end{thm}

When all the osculation points are real, the statement is even stronger:

\begin{thm}[\cite{bib:MTV09},\cite{bib:Pur09}] \label{thm:MTV-purbhoo}
Suppose $p_i \in \mathbb{RP}^1$ for all $i$. Then $X$ is reduced and consists entirely of real points.
The same statement holds in $\OG$, with modifications as above.
\end{thm}

\subsection{The Wronski map}

The \emph{Wronski map} 
\[\Wr: \Gr(k,\mathrm{Sym}^{n-1}(\mathbb{C}^2)) \to \mathbb{P}^{k(n-k)}= \mathbb{P}(\mathrm{Sym}^{k(n-k)}\mathbb{C}^2)\]
is defined as follows. Thinking of elements $f_1, \ldots, f_k \in \mathrm{Sym}^{n-1}(\mathbb{C}^2)$ as polynomials in $z$, we consider the determinant
\[\Wr(f_1, \ldots, f_k; z) := \det \begin{bmatrix} f_1 & \cdots & f_k \\ f_1' & \cdots & f_k' \\ \vdots &&\vdots \\ f_1^{(k-1)} & \cdots & f_k^{(k-1)}\end{bmatrix} \in \mathrm{Sym}^{k(n-k)}(\mathbb{C}^2).\]
This determinant depends only on the subspace $U = \langle f_1, \ldots, f_k\rangle \in \Gr(k,\mathrm{Sym}^{n-1}(\mathbb{C}^2))$, so we put $\Wr(U;z) := \Wr(f_1, \ldots, f_k)$. We often factor the Wronskian as $\prod (z-p_i)^{d_i}$ and think of it as a multiset of points ${\bf m} = \{(p_\bullet,d_\bullet)\}$.  

\begin{remark}
Note that technically the points $p_i$ are on $\mathbb{P}^1$ and the Wronskian factorization $\prod (z-p_i)^{d_i}$ should be written as a two-variable homogeneous polynomial $\prod (p_{i_1}z_1-p_{i_2}z_2)^{d_i}$ where the points involved have homogeneous coordinates $(p_{i_1}:p_{i_2})$.  In terms of the local coordinate $z$, then, if one of the points $p_i$ is at $\infty=(0:1)$, then instead of having a factor of $(z-p_i)^{d_i}$, we simply lower the degree of the Wronskian polynomial by $d_i$ and only multiply together the remaining factors.
\end{remark}

We note that $\Wr$ has a coordinate-free description similar to equation \eqref{eqn:bilinear-form}, induced by the unique $\GL_2$-equivariant linear map
\[\bigwedge^k(\mathrm{Sym}^{n-1}(\mathbb{C}^2)) \to \mathrm{Sym}^{k(n-k)}(\mathbb{C}^2) \otimes \det(\mathbb{C}^2)^{\otimes \binom{k}{2}},\]
projectivized and restricted to $\Gr$ in the Pl\"{u}cker embedding. We will not need this; for us, the key fact is that $\Wr$ is flat and finite, of degree $|\mathrm{SYT}(\rect)|$, and its fibers have the following property.

\begin{lemma}[{\cite[Theorem 2.5]{bib:Pur10}}] \label{lem:wronski-divisibility-type-A}
Let $U \in \Gr(k,\mathrm{Sym}^{n-1}(\mathbb{C}^2))$ and $p \in \mathbb{P}^1$. Then $\Wr(U;z)$ is divisible by $(z-p)^{d}$ if and only if there is some partition $\lambda \vdash d$ so that $U \in X_\lambda(\mathscr{F}(p))$.
\end{lemma}

Consequently, the fiber of the Wronski map at a point, say $\Wr(U;z) = \prod (z-p_i)^{d_i}$, is a union of intersections of Schubert varieties:
\begin{equation}\label{eqn:wronski-fiber-type-A}
\Wr^{-1} \bigg( \prod_i (z-p_i)^{d_i} \bigg)= \bigsqcup_{\lambda^{(i)} \vdash d_i} \bigcap_i X_{\lambda^{(i)}}(p_i).
\end{equation}
Here, the union is over tuples $\lambda^{(1)}, \ldots, \lambda^{(r)}$ with $\lambda^{(i)} \vdash d_i$ for each $i$.

In type B, the statements are analogous. The restriction of $\Wr$ to $\OG$ is again flat and finite, of degree $|\mathrm{SYT}(\stair)|$, and the fibers can be described based on square factors of $\Wr(U;z)$.

\begin{lemma}[{\cite[Theorem 5]{bib:Pur09}}] \label{lem:wronski-divisibility-type-B}
Let $U \in \OG(n,\mathrm{Sym}^{2n}(\mathbb{C}^2))$ and $p \in \mathbb{P}^1$. Then $Wr(U;z)$ is a perfect square. Moreover, it is divisible by $(z-p)^{2d}$ if and only if there is some strict partition $\sigma \vdash d$ so that $U \in \Omega_\sigma(\mathscr{F}(p))$.
\end{lemma}

Thus, the fibers of the Wronski map on $\OG$ have an analogous description in terms of intersections of orthogonal Schubert varieties:
\begin{equation}\label{eqn:wronski-fiber-type-B}
\Wr^{-1} \bigg( \prod_i (z-p_i)^{2d_i} \bigg) \cap \OG= \bigsqcup_{\sigma^{(i)} \vdash d_i} \bigcap_i \Omega_{\sigma^{(i)}}(p_i).
\end{equation}

\subsubsection{Tableau labels, type A}

By Theorems \ref{thm:eisenbud-harris} and \ref{thm:MTV-purbhoo}, if the points $p_i$ are all real, the set
\[X_{\lambda^{(1)}}(p_1) \cap \cdots \cap X_{\lambda^{(r)}}(p_r)\]
consists of reduced real points, enumerated by the appropriate Littlewood-Richardson coefficient. In fact, more is true: it was shown in \cite{bib:Pur10} that one may label these points individually by Young tableaux, viewing them as points in a fiber of the Wronski map, as in Equation \eqref{eqn:wronski-fiber-type-A}.

Let ${\bf d} = (d_1, \ldots, d_r)$ be nonnegative integers with $\sum d_i = k(n-k)$. Let $P_{{\bf d}} \subset \mathbb{P}^{k(n-k)}$ be the locus where the Wronskian factors as a product $\prod (z-p_i)^{d_i}$ for some $p_i$. Let $\overline{P_{{\bf d}}}$ be the closure of $P_{{\bf d}}$. By the theorems above, the restriction of the Wronski map to the real points of $P_{{\bf d}}$ is a covering map. The approach in \cite{bib:Pur10} is then to vary the points $p_i$ and label them based on asymptotic behavior. The result is as follows, in slightly different language:

\begin{thm}[Tableau labels, type A \cite{bib:Pur10}] \label{thm:tableau-labels-type-A}
Suppose $p_1 < p_2 < \ldots < p_r < \infty$. Let ${\bf m} = \{(p_i,d_i)\}$ be the corresponding multiset of points on $\mathbb{RP}^1$. There is a bijection between points of $\Wr^{-1}({\bf m})$ lying in $X_{\lambda^{(1)}}(p_1) \cap \cdots \cap X_{\lambda^{(r)}}(p_r)$ and sequences of dual equivalence classes $(D_1, \ldots, D_r)$ of type $(\lambda^{(1)}, \ldots, \lambda^{(r)})$.

Moreover, this bijection changes only when a point $p_i$ crosses $\infty$.
\end{thm}

We briefly describe these labels in the case where all the multiplicities are 1, that is, the Wronskian factors as $\Wr(z) = \prod (z-p_i)$. In this case, $\Wr^{-1}({\bf m})$ is in bijection with the set of standard tableaux of rectangular shape, $\mathrm{SYT}(\rect)$. For complete details, see \cite[Section 4]{bib:Pur10}. By translating all points upwards along $\mathbb{R}$, we assume $0 < p_1 < \cdots < p_r < \infty$.

Let $x \in \Wr^{-1}({\bf m})$ be a point in the fiber. Viewing $x$ as the row span of a $k \times n$ matrix, we write $p_\lambda(x)$ for the Pl\"{u}cker coordinate corresponding to the tuple of columns
\[\lambda + (k, k-1, \ldots, 1) = (\lambda_1 + k, \cdots, \lambda_k + 1).\] We consider the degeneration where we send $p_i \to 0$ for all $i$, with $p_i \approx z^{k(n-k)+1-i}$. This gives a path in $\mathbb{P}^{k(n-k)}$, staying in the locus $P_{\bf 1,\ldots,1}$, where $Wr$ is a covering map. We lift this path to $x$ and examine the rates of convergence of the Pl\"{u}cker coordinates $p_\lambda(x)$ as $z \to 0$. 

It turns out that these rates are determined by a uniquely defined standard tableau $T = T(x)$, as follows: for each $\lambda$, let $\mathrm{val}(T|_\lambda)$ be the sum of the entries of $T$ in the squares $\lambda$. Then
\[p_\lambda(x;z) \approx z^{\mathrm{val}(T|_\lambda)} \text{ as } z \to 0.\]

\begin{example}
In $\Gr(2,4)$, consider the intersection of four divisorial Schubert classes. The path described above lifts to the Schubert cell complementary to $X_\ybox(\infty)$, with coordinates $\begin{bmatrix}0 & 1 & p_{\tyng{1}} & p_{\tyng{2}} \\ 1 & 0 & p_{\tyng{1,1}} & p_{\tyng{2,1}}\end{bmatrix}$. Note that, up to sign, $p_{\tyng{2,2}}$ is the determinant of the last two columns and $p_{\varnothing} = 1$, the determinant of the first two columns. Then, as $z \to 0$, the two points in $\Wr^{-1}({\bf m})$ are of the form
\[\young(12,34) : \begin{bmatrix}0 & 1 & O(z) & O(z^3) \\ 1 & 0 & O(z^4) & O(z^6)\end{bmatrix}, \qquad 
\young(13,24) : \begin{bmatrix}0 & 1 & O(z) & O(z^4) \\ 1 & 0 & O(z^3) & O(z^6)\end{bmatrix},
\]
and some cancellation gives $p_{\tyng{2,2}} = O(z^{10})$ in both cases.
\end{example}
The assignment $x \mapsto T$ is bijective. By construction, it changes only when the lowest or highest point $p_i$ crosses to $\infty$. Specifically, suppose $p_1$ decreases and crosses $-\infty$. Then the normalization described above (translating all points upwards to $\mathbb{R}_+$) changes, as does the tableau label. It turns out that the new tableau $T'$ is obtained from $T$ using Sch\"{u}tzenberger's \emph{tableau promotion} (see \cite{bib:StanleyEC2}).

In the case where the multiplicities ${\bf m} = \{(p_i, d_i)\}$ are arbitrary, a point $x \in \Wr^{-1}({\bf m})$ lies in some intersection of Schubert varieties $X_{\lambda^{(1)}}(p_1) \cap \cdots \cap X_{\lambda^{(r)}}(p_r)$, for some partitions $\lambda^{(i)} \vdash d_i$. To label $x$, we first perturb the base point ${\bf m}$ slightly, separating $p_i$ into $d_i$ separate points. We lift this perturbation to $x$, then obtain a standard tableau $T$ by the construction above. Different choices of lift will give different standard tableaux, but the data of the sequence of dual equivalence classes $(D_1, \ldots, D_r)$ remains well-defined (\cite[Theorem 6.4]{bib:Pur10}), and the statement about Sch\"{u}tzenberger promotion generalizes to Theorem \ref{thm:collisions-typeA} below. Concretely, if $T \in SYT(\rect)$ is such a tableau, then $D_i$ is the dual equivalence class of the subtableau containing the entries $r_i + 1, \ldots, r_i + |\lambda_i|$, where $r_i = \sum_{j < i} |\lambda_j|$. The class $D_i$ will have type $\lambda^{(i)}$.

\subsubsection{Collisions and monodromy}

When a marked point crosses $\infty$, the normalization step (moving all points to $\mathbb{R}_+$) changes, so the tableau label changes. It turns out that the new tableau label is obtained by tableau switching.

\begin{thm}[{\cite[Theorem 3.5]{bib:Pur10}}]
Consider a path $\gamma : [0,1] \to \mathbb{P}^{k(n-k)}$ in which $p_1$ crosses $-\infty$ and the others remain constant. Let $x \in \Wr^{-1}(\gamma(0))$ and let $x' \in \Wr^{-1}(\gamma(1))$ be the result of lifting $\gamma$ to $x$.

If $x$ corresponds to the chain of dual equivalence classes $(D_1, \ldots, D_r)$, then $x'$ corresponds to $(D'_2, D'_3, \ldots, D'_r, D'_1)$, obtained by successively tableau-switching $D_1$ past $D_2, D_3, \ldots, D_r$.
\end{thm}

Finally, we give a partial description of how the tableau labels change when the points $p_i$ collide. In general, the map becomes a branched cover. 

\begin{thm} \label{thm:collisions-typeA}
Let $\gamma : [0,1] \to \mathbb{P}^{k(n-k)}$ be the path in which $p_i$ collides with $p_{i+1}$ and the other points stay constant. Let $x \in \Wr^{-1}(\gamma(0))$ and let $x' \in \Wr^{-1}(\gamma(1))$ be the result of lifting $\gamma$ to $x$.

Suppose $x$ corresponds to the chain of dual equivalence classes $(D_1, \ldots, D_r)$. Let $\bar D = D_i \cup D_{i+1}$, the class obtained by concatenating tableaux. Then $x'$ is labeled by $(D_1, \ldots, D_{i-1},  \bar D, D_{i+2}, \ldots, D_r)$.
\end{thm}

This statement follows directly from the limiting procedure in Theorem \ref{thm:tableau-labels-type-A} for constructing the tableau labels. Note that, in general, the same class $\bar D$ may arise from different pairs $(D_i,D_{i+1})$, even if the shapes $\lambda^{(i)}, \lambda^{(i+1)}$ are fixed.

\subsubsection{Tableau labels, type B}

In type B, the story is analogous. Let $p_1 < p_2 < \ldots < p_r < \infty$, and let ${\bf m} = \{(p_i,d_i)\}$ be a multiset, where $\sum d_i = \dim(\OG)$. For each $i$, let $\sigma^{(i)} \vdash d_i$ be a strict partition.

\begin{thm}[Tableau labels, type B \cite{bib:Pur10B}] \label{thm:tableau-labels-type-B}
There is a bijection between points of $\Wr^{-1}({\bf m})$ lying in $\Omega_{\sigma^{(1)}}(p_1) \cap \cdots \cap \Omega_{\sigma^{(r)}}(p_r)$ and chains of \emph{shifted} dual equivalence classes $(D_1, \ldots, D_r)$ of type $(\sigma^{(1)}, \ldots, \sigma^{(r)})$.

This bijection changes only when a point $p_i$ crosses $\infty$ or collides with another point:
\begin{itemize}
\item If $p_1$ crosses $-\infty$, the label $(D_1, \ldots, D_r)$ changes to $(D'_2, \ldots, D'_r, D'_1)$, obtained by shifted-tableau-switching $D_1$ past $D_2, \ldots, D_r$ in order.
\item If $p_i$ collides with $p_j$, the label $(D_1, \ldots, D_r)$ changes to $(D_1, \ldots, D_{i-1},  \bar D, D_{i+2}, \ldots, D_r)$, where $\bar D = D_i \cup D_{i+1}$ is the class obtained by concatenation.
\end{itemize}
\end{thm}

The approach and method are identical; the main connection is the labels in type B are compatible with those in type A in the following sense. In the case where the multiplicities are all $1$, the Wronskian factors as $\Wr(z) = \prod (z-p_i)^2$, so a point $x \in \Wr^{-1}({\bf m})$ lies in some intersection of codimension-2 Schubert varieties in $\Gr(n,W)$,
\[X_{\tyng{2} \text{ or } \tyng{1,1}}(p_1) \cap \cdots \cap X_{\tyng{2} \text{ or } \tyng{1,1}}(p_N).\]
Thus $x$ is labeled by a chain of dual equivalence classes $(D_1, \ldots, D_N)$ of type $(\tyng{2} \text{ or } \tyng{1,1}, \ldots, \tyng{2} \text{ or } \tyng{1,1})$. Note that each $D_i$ is uniquely determined by the data of the skew shape it occupies and, if the two boxes are non-adjacent, by whether the filling should be ``horizontal-type'' (${\tiny \young(:2,1)}$) or ``vertical-type'' (${\tiny \young(:1,2)}$). By \cite[Theorem 1]{bib:Pur10B}, it turns out that $x \in \OG$ if and only if:
\begin{itemize}
\item[(i)] All the Schubert conditions are type $\tyng{2}$, and
\item[(ii)] The tableau $T$ is \emph{symmetric}, that is, each successive pair of boxes occupies the locations $(i,j)$ and $(j,i+1)$ for some $i,j$.
\end{itemize}
Note that condition (i) is necessary since $X_{\tyng{2}} \cap \OG = \Omega_{\tyng{1}}$, whereas $X_{\tyng{1,1}} \cap \OG = X_{\tyng{3,1}} \cap \OG = \Omega_{\tyng{2}}$ since $\tyng{3,1}$ is the symmetrization of $\tyng{1,1}$. Thus, by dimension-counting, the intersection in $\OG$ will be empty (by Theorem \ref{thm:eisenbud-harris}) unless condition (i) holds.

Given condition (ii), it is natural to record only the ``upper half'' of the entries, thereby obtaining a shifted standard tableau:
\[\young(124,356) \quad \leadsto \ \ \young(\hfil24,\hfil\hfil6)\quad = \quad \young(12,:3)\ .\]
When the multiplicities are larger, the approach is analogous to that used in type $A$: we perturb a multiple point into a collection of distinct points, corresponding to square factors $(z-p_i)^2$ of the Wronskian, and we lift to $\OG$. We thereby obtain a standard tableau, which is symmetric if and only if $x \in \OG$. As in type A, the resulting tableaux are well-defined only up to shifted dual equivalence.

\begin{thm}[\cite{bib:Pur10B}] \label{thm:compatibility-BA}
Let $x \in \Wr^{-1}(\mathbf{m}) \cap \OG$. The symmetrical and shifted labels of $x$, according to Theorems \ref{thm:tableau-labels-type-A} and \ref{thm:tableau-labels-type-B}, agree under the above bijection. \end{thm}

\subsubsection{Schubert curves and evacuation shuffling}

Suppose $\sigma^{(1)}, \ldots, \sigma^{(r)}$ are partitions with $\sum |\sigma^{(i)}| = \dim(\OG) - 1$, and $p_1, \ldots, p_r \in \mathbb{P}^1$ are distinct points. We define the {\bf Schubert curve}
\[S(\sigma^\bullet) = \bigcap_{i=1}^r \Omega_{\sigma^{(i)}}(\mathscr{F}(p_i)).\]
Given $x \in S$, we see by Lemma \ref{lem:wronski-divisibility-type-B} and degree-counting that $\Wr(x;z)$ must have the form
\[\Wr(x;z) = \prod_{i=1}^r (z-p_i)^{2|\sigma^{(i)}|} \cdot (z-t)^2,\]
where the $t$ factor varies depending on $x$. The set of polynomials of this form, for all $t$, is naturally identified with $\mathbb{P}^1$, yielding a map $\Wr: S(\sigma^\bullet) \to \mathbb{P}^1.$
We describe the geometry of this map. Note that the analogous statements in type A are already known \cite[Corollary 2.9]{bib:Levinson} via the $\overline{M}_{0,n}$ construction.

\begin{lemma} \label{lem:geometry-schubert-curves-type-B}
The Schubert curve $S(\sigma^\bullet)$ is reduced, has smooth real points, and $\Wr: S(\sigma^\bullet) \to \mathbb{P}^1$ is flat. Over $\mathbb{RP}^1$, the map is smooth, and moreover $\Wr^{-1}(\mathbb{RP}^1)$ consists entirely of real points. In particular, $S(\mathbb{R}) \to \mathbb{RP}^1$ is a smooth covering map of degree $f_{\sigma_\bullet,\ybox}^{\stair}$.
\end{lemma}

\begin{proof}
First note that the fibers for $p \in \mathbb{P}^1 - \{p_1, \ldots, p_r\}$ are of the form $S(\sigma^\bullet) \cap \Omega_{\ybox}(\mathscr{F}(p))$. For the other fibers, if $p = p_i$ for some $i$, the fiber is the disjoint union
\[\Wr^{-1}(p_i) = \bigsqcup_{\sigma \in \sigma^{(i)}_+} \Omega_\sigma(\mathscr{F}(p_i)) \cap \bigcap_{j \ne i} \Omega_{\sigma^{(j)}}(\mathscr{F}(p_j)),\]
where $\sigma^{(i)}_+$ is the set of strict partitions obtained by adding a box to $\sigma^{(i)}$.

By Theorem \ref{thm:eisenbud-harris}, all the fibers are dimensionally transverse intersections, so each fiber's multiplicity is given by the corresponding product of Schubert classes. By Theorem \ref{thm:MTV-purbhoo}, all the fibers over $\mathbb{RP}^1 \setminus \{p_i\}$ are reduced and real.
By the Pieri rule, the set-theoretic multiplicity does not change at the fibers $p=p_i$, and these fibers are also reduced and real, by Theorem \ref{thm:MTV-purbhoo}.

We deduce that $S \to \mathbb{P}^1$ is flat, $S$ is generically reduced and its real points are all smooth. By flatness, it follows that $S$ is reduced.
\end{proof}

Notably, the real points of $S$ are smooth even in the case where $t = p_i$ for some $i$. The tableau label will change at this collision; we describe the resulting monodromy using the \textbf{evacuation shuffle} operator $\esh$ defined below. 

\begin{thm} \label{thm:esh-type-B}
Consider the path $\gamma : [0,1] \to [p_i - \epsilon, p_i + \epsilon] \subset \mathbb{RP}^1$. Let $x \in \Wr^{-1}(\gamma(0))$ and let $x' \in \Wr^{-1}(\gamma(1))$ be obtained by lifting $\gamma$ to $x$.

Let $x$ be labeled by $(D_1, \ldots, \ybox, D_i, \ldots D_r)$. Then, for $x'$, the pair $(\ybox, D_i)$ is replaced by $(D'_i, \ybox) = \esh(\ybox,D_i)$, and the other classes are unchanged.
\end{thm}
The evacuation shuffle operation $\esh : (D,E) \mapsto (E',D')$ is uniquely determined by two properties:
\begin{itemize}
\item[(i)] If $D$ is a straight shape, $\esh(D,E)$ is the same as tableau-switching $(D,E) \sim (\tilde E,\tilde D)$, and
\item[(ii)] The operation $\esh$ is coplactic, that is, if $(D,E,F) \sim (\tilde F,\tilde D, \tilde E)$ is obtained by successively switching $D$ and $E$ past $F$, then $(\esh(D,E),F) \sim (\tilde F, \esh(\tilde D,\tilde E))$.
\end{itemize}
In particular, $\esh(\ybox,D_i)$ may be computed by first rectifying $(\ybox,D_i)$ to a straight shape, then tableau-switching $\ybox$ past $D_i$, then un-rectifying. We show that this procedure must compute the monodromy of the path $\gamma$.
\begin{proof}[Proof of Theorem \ref{thm:esh-type-B}]
First, the collisions statement in Theorem \ref{thm:tableau-labels-type-B} shows that the labels $D_1, \ldots, D_{i-1}, D_{i+1}, \ldots, D_r$ do not change from $x$ to $x'$. Next, consider the path $\gamma'$, where we first rotate all the $p_i$ through $\mathbb{R}_-$,
\[ (p_1, \ldots, p_{i-1}, t, p_i, \ldots p_r) \xrightarrow{\text{ rotate }} (t, p_i, \ldots, p_r, p_1, \ldots, p_{i-1}),\]
then switch $t$ and $p_i$, then rotate back. On $\mathbb{P}^N$, there is a homotopy from $\gamma$ to $\gamma'$. By Lemma \ref{lem:geometry-schubert-curves-type-B}, the fibers of the Wronski map do not collide when $t$ collides with $p_i$, so we may lift the homotopy to $\OG$ and replace $\gamma$ by $\gamma'$.

Combinatorially, $\gamma'$ changes the tableau label by first switching the pair $(\ybox, D_i)$ inwards to a rectified shape, then allowing the collision, then tableau-switching back. It therefore suffices to show that, for rectified shapes, the collision step is the same as tableau-switching $(\ybox,D_i)$.

Thus we reduce to the case $i=1$, and we show $(\ybox, D_1) \sim (D'_1, \ybox)$.

At the collision, the label on $x$ changes from $(\ybox, D_i)$ to $\ybox \sqcup D_i$, and the label on $x'$ changes from $(D'_i, \ybox)$ to $D'_i \sqcup \ybox$, and these classes agree:
\[\ybox \sqcup D_i = D'_i \sqcup \ybox = D_{\sigma^{(i)}_+},\]
the unique dual equivalence class of straight shape $\sigma_+^{(i)}$, for some strict partition extending $\sigma^{(i)}$ by a box. By the Pieri rule, there is only one factorization of $D_{\sigma^{(i)}_+}$ into pairs of the given forms. Tableau switching $(\ybox, D_i)$ gives such a pair, so it must yield $(D'_i,\ybox)$.
\end{proof}

\subsection{The monodromy operator $\omega$}\label{sec:esh}

We now describe the \textbf{monodromy operator} $\omega$, that is, the permutation on any fiber of the Wronski map, say $\Wr^{-1}(0)$ (assuming without loss of generality that $p_i\neq 0$ for all $i$), induced by the smooth covering $\Wr:S(\mathbb{R})\to \mathbb{RP}^1$.

By repeatedly applying Theorem \ref{thm:esh-type-B} to move a point $x$ past each point $p_i$ to make a full loop above $\mathbb{RP}^1$, we see that can canonically label the points of this fiber by the chains of dual equivalence classes in the set
\[\DE(\lambda^{(1)}, \ebox, \lambda^{(2)}, \ldots, \lambda^{(r)}),\]
such that the monodromy operator $\omega$ is given by the composition of shuffles and evacu-shuffles
\[\omega = \sh^{(2)} \circ \cdots \circ \sh^{(r-1)} \circ \esh^{(r-1)} \circ \cdots \circ \esh^{(2)},\]
where $\esh^{(i)}$ and $\sh^{(i)}$ act on the $i$-th and $(i+1)$-th tableaux in the chain.
In the sections that follow, our local description of $\esh$ will apply to each of the above $\esh^{(i)}$ operations.  Therefore, our main results, in the case of three marked points $p_1,p_2,p_3$, generalize without difficulty to the general case.  Thus, for simplicity, we restrict for the remainder of the paper to the case of three partitions $\alpha, \beta, \gamma$, i.e. we study the operator
\[\omega = \sh^{(2)} \circ \esh^{(2)}\]
on the sets $\DEyb$ and  $\DEby$, or equivalently
\[\LRyb \text{ and } \LRby.\]

For the remainder of the paper we will primarily work with $\LRyb$ and $\LRby$ rather than the corresponding dual equivalence chains.  Furthermore, since we mostly work only with $\sh^{(2)}$ and $\esh^{(2)}$, we often simply abbreviate them as $\sh$ and $\esh$.

 Since the straight shape $\alpha$ and anti straight shape $\gamma^c$ each have only one Littlewood-Richardson tableau, an element of $\LRyb$ can be thought of as a pair $(\ybox,T)$, with $T$ a Littlewood-Richardson tableau of rectification shape (and content) $\beta$, and $\ybox$ an inner co-corner of $T$, such that the shape of $\ybox \sqcup T$ is $\gamma^c/\alpha$. We represent elements of $\LRby$ similarly, with $\ybox$ as an outer co-corner. We will occasionally refer to the element as $T$ if the position of the $\ybox$ is understood.

Combinatorially, $\omega$ can be thought of as a commutator of well-known operations on shifted Young tableaux.  Computing $\esh(\ybox,T)$ is equivalent to the following steps:
\begin{itemize}
  \item \textbf{Rectification.} Treat the $\ybox$ as having value $0$ and being part of a semistandard tableau $\widetilde{T}=\ybox \sqcup T$.  Rectify, i.e. shuffle $(S,\widetilde{T})$ to $(\widetilde{T}',S')$, where $S$ is an arbitrary straight-shape tableau.
  \item \textbf{Promotion (see \cite{bib:StanleyEC2}).} Delete the $0$ of $\widetilde{T}'$ and rectify the remaining tableau.  Label the resulting empty outer corner with the number $\ell(\beta)+1$.
  \item \textbf{Un-rectification.} Un-rectify the new tableau by shuffling once more with $S'$.  Replace the $\ell(\beta)+1$ by $\ybox$.
\end{itemize}
Note that the promotion step corresponds to shuffling the $\ybox$ past the rest of the rectified tableau. 

For computational purposes, this description is inefficient and opaque -- it is difficult to predict $\esh(\ybox, T)$ even though (it turns out) very few entries are changed! Thus, our next goal is to describe this algorithm `locally', in a way that does not involve rectifying the tableau.

\section{Shifted Evacuation-shuffling via coplactic operators}\label{sec:crystal-algorithm}

We now give a first local algorithm to compute evacuation-shuffling on skew shifted tableaux,
\[\mathrm{esh} : (\ybox, T) \mapsto (T', \ybox).\]
as a certain composition of the coplactic operators (see Section \ref{sec:coplactic}) $E_i, E'_i, F_i, F'_i$ for various $i$ (Theorem \ref{thm:crystal-algorithm}). 

We begin with an observation about the `promotion' step of the rectify-unrectify process described above.

\begin{lemma}\label{lem:straight-promotion}
  Suppose $\ybox\sqcup T$ is a straight shape tableau with $T$ a Littlewood-Richardson tableau and $\ybox$ in the (unique) inner corner of $T$.  Then shuffling $\ybox$ past $T$ consists of two phases:
\begin{itemize}
  \item \textbf{Phase 1:} If the entry $y$ east of $\ybox$ is a primed letter, slide $\ybox$ past $y$ and then south past the entry below it.  Repeat until there is not a primed entry east of $\ybox$, and then go to Phase 2.
  \item \textbf{Phase 2:} Slide the $\ybox$ horizontally to the end of its row.
\end{itemize}
\end{lemma}

\begin{proof}
  Since $T$ is highest weight, it can be formed by a single outer slide applied to a highest weight straight shape tableau.  It is clear by inspection that such an outer slide is the reverse of the two-phase process described above.  (See Figure \ref{fig:rectified-esh}.) 
\end{proof}

\begin{figure}
\begin{center}
\includegraphics{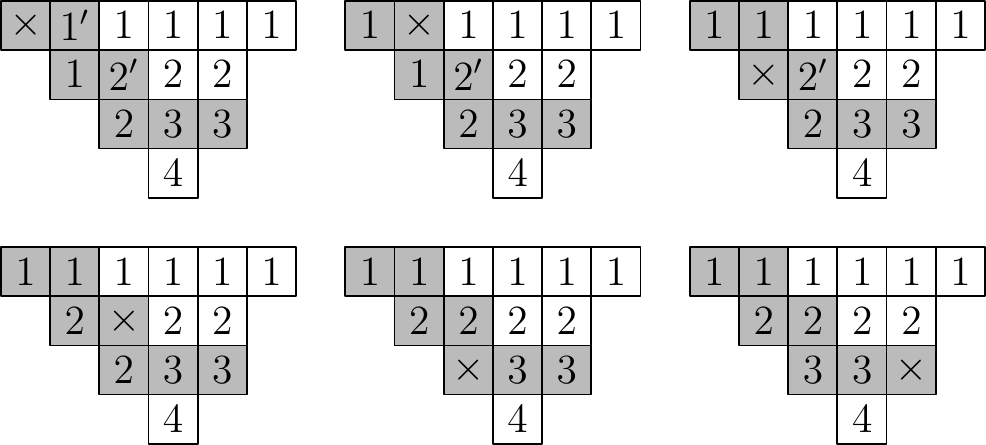}
\end{center}
\caption{\label{fig:rectified-esh} Shuffling the $\ybox$ past a highest weight tableau; the path of the $\ybox$ is shaded.}
\end{figure}

\subsection{The coplactic algorithm}
By expressing each phase described in 
Lemma~\ref{lem:straight-promotion} in
terms of the operators $E_i, E_i', F_i, F_i'$, we obtain a local algorithm 
for $\esh$.

\begin{thm}[Coplactic algorithm for $\esh$]\label{thm:crystal-algorithm}
  Let $(\ybox,T)\in \LRyb$.  Then $\esh(\ybox,T)$ can be computed as follows. Start with step $i=1$ and in Phase 1, and replace $\ybox$ by $0$.
  \begin{itemize}
    \item \textbf{Phase 1:} 
    If $E_{i-1}'(T)\neq E_{i-1}(T)$, apply $E_{i-1}'\circ E_{i-1}^{\mathrm{wt}(T)_i-2}$. Increment $i$ and repeat Phase 1.
    
    Otherwise, replace the (unique) $i{-}1$ by $\ybox$, add $1$ to all values less than $i{-}1$, and go to Phase 2.
    
    \item \textbf{Phase 2:} Set $\ybox=i'$ and apply the composition of operators $$F_r \circ \cdots \circ F_{i+1} \circ F_{i}$$ to the resulting tableau, where $r=\ell(\beta)$.  Replace the only $r+1$ by $\ybox$. 
  \end{itemize} 
\end{thm}

\begin{proof}
It suffices to show that this local algorithm is correct in the rectified setting (i.e. it agrees with Lemma~\ref{lem:straight-promotion}), and all operations are coplactic.  

In the rectified setting, it is straightforward to check that when we are in phase 1, $E_{i-1}' \circ E_{i-1}^{\mathrm{wt}(T)_i-2}$ has the same effect (up to relabelling entries) as sliding $\ybox$ through the $i$-entries of $T$.  The transition from Phase 1 to Phase 2 can be detected by whether or not $E_i'$ and $E_i$ agree, a coplactic condition; and in Phase 2, $F_r \circ \cdots \circ F_{i+1} \circ F_{i}$ has the same effect as sliding the $\ybox$ horizontally to the end of the row.  Thus the algorithm is correct on rectified tableaux.

The operations $E_i,F_i,E_i',F_i'$ are all coplactic, so it remains to check that all ``relabelling'' steps (adding $1$ to entries, replacing $\ybox$ with a numerical value, etc.) are also coplactic.  We claim that all of these replacements preserve the standardization of the tableau.  This clear everywhere except for the replacement $\ybox = i'$ at the start of Phase 2.  For this step, we need to show that at the end of Phase 1, before we do any of the relabelling, we can replace the unique $i-1$ with an $i'$ and this preserves standardization.  By definition, this is equivalent to saying that $F'_{i-1}$ is defined at this stage.  But since $F'_{i-1}$ is coplactic, it is sufficient to verify that this is true in the rectified case, and this is again straightforward.
\end{proof}

Similarly we can express $\esh^{-1}$ in terms of coplactic operators.
The main detail requiring explanation here is how to detect and perform the 
Phase 1/Phase 2 transition while running the algorithm in reverse.

\begin{thm}\label{thm:reverse-coplactic}
  Let $(T,\ybox)\in \LRby$.  Then $\mathrm{esh}^{-1}(T,\ybox)$ can be computed as follows.  Replace $\ybox$ with $r+1$ where $r$ is the largest entry appearing in $T$, and let $T'$ be the union of $T$ with this entry $r+1$. 
  \begin{itemize}
     \item \textbf{Reverse Phase 2:} Apply $E_{r},E_{r-1},E_{r-2}$, and so on to $T'$ until reaching an index $i$ for which $E_{i-1}$ is not defined (or $i=1$).  Replace the first $i$ or $i'$ in standardization order by $\ybox$, and call the resulting tableau $T''$.  Proceed to Reverse Phase 1.
     \item \textbf{Reverse Phase 1:} Substract $1$ from all entries less 
than $i$ in $T''$.  Replace $\ybox$ with $i-1$ and apply the composition
\[
     F_0^{\mathrm{wt}(T)_0-2} \circ F'_0 \circ F_1^{\mathrm{wt}(T)_1-2}\circ F'_1 \circ \cdots \circ F_{i-2}^{\mathrm{wt}(T)_{i-2}-2} \circ F'_{i-2}
\]
to the resulting tableau.  Replace the only $0$ by $\ybox$.
  \end{itemize}
\end{thm}

The proof uses essentially the same idea as the proof of 
Theorem~\ref{thm:crystal-algorithm}, and we omit it.

\subsection{Phase 1 vs. Phase 2 in types A and B} \label{subsec:phase1-vs-phase2}
The fact that the algorithm in Theorem~\ref{thm:crystal-algorithm} has two phases with two 
completely different descriptions is, at first glance, a bit surprising --- in contrast to $\sh(\ybox, T)$, which does not require a two-phase algorithm. Lemma~\ref{lem:straight-promotion} provides some motivation for this fact, but doesn't really address the question of why this division exists. Here, we discuss some facts about evacuation-shuffling that help explain this structure.

A similar phenomenon occurs in type A: evacuation-shuffling divides into two phases with seemingly different rules. In establishing the algorithm via local moves, however, a key simplifying step is to understand Phase 2 as a `dualized' Phase 1 \cite[Lemmas 4.14-15]{bib:GillespieLevinson}. In type B, we show below that no comparable duality exists, i.e. Phase 2 is \emph{not} simply a disguised version of Phase 1.

Suppose $T$ is a shifted or unshifted tableau, and let $\ybox$ be an inner corner of $T$.  
Write $T = T_{<s} \sqcup T_{\geq s}$, where $T_{< s}$ and $T_{\geq s}$ consist of all (primed or unprimed) entries of $T$ less than $s$ and greater than or equal to $s$, respectively. A key property of the jeu de taquin bijection is that we can compute $\sh(T,\ybox)$ in stages, by sliding $\ybox$ past $T_{\geq s}$, and then past $T_{< s}$. In general evacuation shuffling does not have this property, but there is a special case in which it does.

Let $\mu$, $\nu$ and $\lambda$ denote the rectification shapes of $T_{< s}$, $T_{\geq  s}$ and $\ybox \sqcup T$, treating $\ybox$ as $0$. Let $c^\lambda_{\tinybox, \mu, \nu}$ or $f^\lambda_{\tinybox, \mu, \nu}$ denote the appropriate generalized Littlewood-Richardson coefficient. We call $T= T_{< s} \sqcup T_{\geq s}$ a \newword{simple decomposition} if $c^\lambda_{\tinybox, \mu, \nu}$ or $f^\lambda_{\tinybox, \mu, \nu} = 1$. Note that this definition depends not only on $T$, but also on the choice of inner corner $\ybox$. The decomposition is \newword{non-trivial} if $T_{< s}$ and $T_{\geq s}$ are both nonempty.  

\begin{thm}
\label{thm:eshdecomp}
If $T= T_{< s} \sqcup T_{\geq s}$ is a simple decomposition, then
\[
\esh(\ybox,T) = (T'_{< s} \sqcup T'_{\geq s}, \ybox)
\]
where 
$(T'_{< s}, T'_{\geq s}, \ybox) 
= \esh^{(2)} \circ \esh^{(1)}(\ybox, T_{< s}, T_{\geq s})$.
\end{thm}

\begin{proof}
Write $\esh(\ybox,T) = (\hat T',\ybox)$.  We must show that
$\hat T'_{< s} = T'_{< s}$ and $\hat T'_{\geq  s} = T'_{\geq  s}$.
Since evacuation shuffling is coplactic, and well-defined on dual
equivalence classes, it suffices to prove this in the case where
$\ybox \sqcup T$ is rectified, and $T_{< s}$ and $T_{\geq  s}$
are both Littlewood-Richardson tableaux. (Technically, here we mean
$T_{\geq s}$ is a Littlewood-Richardson tableau after decrementing
all its entries by $s-1$.)  
Note that $(T'_{< s}, T'_{\geq  s}, \ybox)$ and 
$(\hat T'_{< s}, \hat T'_{\geq  s}, \ybox)$ are both
chains of Littlewood-Richardson tableaux contributing to the
coefficient $f^\lambda_{\mu,\nu,\tinybox}$ 
($c^\lambda_{\mu,\nu, \tinybox}$ for unshifted tableaux).  
But since $f^\lambda_{\mu,\nu,\tinybox} = 1$, we must have
$(T'_{< s}, T'_{\geq  s}, \ybox)= (\hat T'_{< s}, \hat T'_{\geq  s}, \ybox)$,
as required.
\end{proof}

\begin{remark}
Theorem~\ref{thm:eshdecomp} provides a sufficient, but not a 
necessary condition to conclude $\esh(\ybox, T)
= (T'_{< s} \sqcup T'_{\geq s}, \ybox)$.  
In the case where $\ybox \sqcup T$
is rectified, the conclusion can be reformulated as
$\esh(\ybox, T_{\geq  s}) = \sh(\ybox, T_{\geq s}$,
or equivalently $\omega(\ybox, T_{\geq  s}) = (\ybox, T_{\geq  s})$.
Examples of tableaux with this property are discussed in
\cite[Proposition 7.5]{bib:GillespieLevinson}.
\end{remark}

For our purposes, the main example is the Phase 1/Phase 2 decomposition
in computing $\esh(\ybox, T)$.
If the transition from Phase 1 to Phase 2 occurs
when $i=s$, then $T = T_{< s} \sqcup T_{\geq s}$ is a simple 
decomposition.
This can be seen directly using the Littlewood-Richardson rule:
$(\ybox, T_{< s}, T_{\geq  s})$ is the unique chain of 
Littlewood-Richardson tableaux contributing to $f^\lambda_{\tinybox, \mu, \nu}$.
This explains why the Phase 1/Phase 2 decomposition, suggested intuitively by 
Figure~\ref{fig:rectified-esh} and Lemma~\ref{lem:straight-promotion}
in the rectified case,
actually leads to a natural division of $T$ into two parts for 
computing $\esh(\ybox, T)$ in general.

To address the question of why the algorithm is so different on 
each part, first note that this is not the only simple
decomposition of $T$.
In fact, by the same argument as above, 
$T = T_{< i} \sqcup T_{\geq  i}$ is a simple decomposition for all
$i < s$.
This tells us that Phase 1 can be computed by evacu-shuffling
$\ybox$ past the $1$- and $1'$-entries in $T$, then past the 
$2$- and $2'$-entries, and so forth, up until the transition point.  
Thus, Phase 1 is completely determined
by the Pieri case, i.e. by how $\esh$ behaves on tableaux in which 
all entries are equal to $1$ or $1'$. 
We record this fact for future reference.

\begin{lemma}\label{lem:partial-esh}
  Phase 1 of $\esh$ can be computed by applying $\esh$ to move the $\ybox$ past just the $1'/1$-ribbon, then separately past the $2'/2$-ribbon, and so on.
\end{lemma}

Note that this fact is also made clear by the coplactic algorithm in Theorem \ref{thm:crystal-algorithm}: step $i$ of Phase $1$ only involves the entries $i'/i$ and a single $i{-}1$, a placeholder for $\ybox$. In Phase 2, however, the $F_i$ step involves the $i'/i$ and $i{+}1'/i{+}1$ strips, both of which may be large.

However, for $r \geq i \geq s$, $T = T_{< i} \sqcup T_{\geq  i}$ is not
a simple decomposition, and thus Phase 2 cannot be further decomposed
in the same way.  However, this analysis is specifically about Littlewood-Richarson tableaux.  It is reasonable to wonder whether we could somehow decompose Phase 2, if we replace $T$ by some other tableau in its dual equivalence class. For unshifted tableaux, this is precisely what happens. The \emph{$s$-decomposition} defined in \cite{bib:GillespieLevinson} is equivalent to choosing a different representative for the dual equivalence class, yielding a simple decomposition.  From this point of view, Phase 2 in the unshifted algorithm behaves like a dual version of Phase 1. However, the following theorem shows that Phase 2 for shifted tableaux is not simply decomposable, in general. Thus for shifted tableaux, Phase 2 seems to be fundamentally different from Phase 1.

\begin{proposition} \label{prop:phase2-indecomposable}
Let $T$ be a Littlewood-Richardson tableau of weight $\beta$, and 
let $\ybox$ be an inner corner of $T$.  Suppose $\esh(\ybox, T)$
is computed entirely in Phase 2, and $\beta_1 = \beta_2+1$. Then
there is no tableau in the dual equivalence class of $T$ 
with a non-trivial simple decomposition.
\end{proposition}

\begin{proof}
Suppose to the contrary that such a tableau, say $T^*$, exists.
Then $T^*$ is in the dual equivalence class of $T$, for
for some $t$,
$T^*_{< t}$, $T^*_{\geq  t}$, $T^*$, and  $\ybox \sqcup T^*$ have 
rectification shape $\mu, \nu, \beta$, and $\lambda$ respectively, 
and $f^\lambda_{\mu, \nu, \tinybox} = 1$.
Let $(T_1, T_2,\ybox')$
be the unique chain of Littlewood-Richardson tableaux such 
such that $T_1$ has weight $\mu$, $T_2$ has weight $\nu$, 
contributing to $f^\lambda_{\mu, \nu,\tinybox}$.

We first claim that $\ybox'$ must be in the first row of $\lambda$.
Since $T^*_{< t}$ and $T^*_{>t}$ are obtained from a decomposition
of $T^*$, we necessarily have $f^\beta_{\mu,\nu} \geq 1$.  Moreover, 
since 
\[
f^\lambda_{\mu, \nu \tinybox} = \sum_\alpha 
f^\lambda_{\alpha, \tinybox} f^\alpha_{\mu\nu},
\]
we see that
$\alpha = \beta$ is the only term contributing
to this sum, and $f^\beta_{\mu\nu} =1$.
It follows that $(T_1, T_2)$ must be the unique chain of 
Littlewood-Richardson tableaux
contributing to $f^\beta_{\mu\nu}$, and hence $\ybox' = \lambda/\beta$.
Since $T^*$ is in the dual equivalence class of $T$, $\lambda$ is also
the rectification shape of $T' \sqcup \ybox$.  By 
Lemma~\ref{lem:straight-promotion}, $\ybox' = \lambda/\beta$ ends
up in row $s$ if the transition from Phase 1 to Phase 2 occurs at
$i=s$.  Since we assumed $\esh(\ybox,T)$
is computed entirely in Phase 2, the transition here is at $s=1$, 
which proves the claim.

Let $b$ denote the first $1$ in $T_2$ in reading order.  We next claim that
$b$ cannot be in the first row of $\lambda$.  Recall that
the entries in row $i$ of $T_2$ are primed or unprimed numbers which
are at most $i$.  Thus, we have the following inequalities:
\begin{align*}
\nu_1 &\geq \beta_1 - \mu_1, \\
\nu_2+\nu_1 &\geq \beta_1 +\beta_2 - \mu_1 - \mu_2.
\end{align*}
If $b$ is in the first row of $\lambda$, then the first of these is
an equality.  Together with $\beta_1 = \beta_2+1$ and $\mu_1 \geq \mu_2+1$,
we deduce that $\nu_1 \leq \nu_2$.  This is impossible, since $\nu$ is a 
strict partition, which proves our second claim.

We produce a new chain of Littlewood-Richardson tableaux 
$(T_1, T_2', \ybox'')$ as follows.
Beginning with $T_2 \sqcup \ybox'$, replace $\ybox'$ with a $1$, and 
replace $b$
with $\ybox''$; then slide $\ybox''$ out.  The effect of 
these replacements on the word of $T_2$ is to move the first $1$ to the 
end of the word, which maintains ballotness 
(by Corollary~\ref{cor:ballotness-preserving}), as does sliding.  
Therefore the resulting
tableau $T_2'$ is also a Littlewood-Richardson tableau of weight $\nu$.
Thus $(T_1, T'_2, \ybox'')$ is a chain of Littlewood-Richardson 
tableaux contributing to the coefficient $f^\lambda_{\mu, \nu, \tinybox}$.
Finally, note that $\ybox'$ is in the first row of $\lambda$, but
$\ybox''$ cannot be in the first row of $\lambda$, so
$(T_1, T_2, \ybox') \neq (T_1, T'_2, \ybox'')$.
Thus $f^\lambda_{\mu, \nu, \tinybox} \geq 2$, for a contradiction.
\end{proof}

%%%%%%%%%%%%%

\section{Shifted Evacuation shuffling via switches}\label{sec:local-algorithm}

In this section, we
reinterpret the algorithm of Theorem \ref{thm:crystal-algorithm} as a
sequence of \emph{switches} which move the $\ybox$ across the tableau.
This formulation describes $\esh$ as a process that more closely
resembles jeu de taquin, and we show
that the algorithm has a number of the same properties as jeu de taquin.
These properties will allow us to make connections with K-theory
in Section~\ref{sec:K-theory}.  
Sections~\ref{sec:switches}--\ref{sec:reverse-algorithm} introduce
the concepts and state the main theorems that are needed for these 
applications. The proofs are given in
Section~\ref{sec:proofs}.
 The reader who is not interested in the technical details of the proofs may wish to proceed directly from 
Section~\ref{sec:reverse-algorithm} to Section~\ref{sec:K-theory},
with some assurance from the authors that these results have also been
extensively verified by computer calculations.

\subsection{Switches, hops, and inverse hops}
\label{sec:switches}

We introduce some terminology that will be
useful for stating, analyzing and applying the algorithm.
Throughout, we will implicitly assume that we are operating on a tableau 
containing a single $\ybox$, whose reading word (not including $\ybox$) is
ballot.  In examples, we will only display the reading word of
the tableau under consideration.

\begin{definition}\label{defn:switch}
A \newword{switch} means that we interchange the positions of $\ybox$ 
and another entry $t$, subject to the condition that 
between $\ybox$ and $t$ in reading order, 
there are no other occurrences of the symbol $t$.  (The entry $t$ switching
with $\ybox$ with may come
either before or after $\ybox$ in reading order.)
\end{definition}

Notice that this definition may depend on the choice of representative for the tableau. For example, $11\ybox \leadsto \ybox11$ is not a switch, but $1'1\ybox \leadsto \ybox11'$ is a switch.  For a variety of reasons, we are forced to work with representatives when talking about switches,  and the question of which representative to use will be a recurring issue. To help keep things straight, we adopt the convention that a switch \emph{never} implicitly changes whether an entry is primed: if the  entry is an $i'$ before the switch, it will be an $i'$ after the switch. When we need to change representatives, this will be done explicitly and separately from switching.

\begin{definition}\label{defn:valid}
We say $\ybox$ \newword{can switch with} an entry $t$ in a tableau, or the switch is \newword{valid}, if, after switching $\ybox$ with $t$, the reading word of the resulting tableau (not including $\ybox$) is ballot.  
\end{definition}

\begin{definition}\label{defn:hop}
A valid switch is called a \newword{hop}, if $\ybox$ moves backward in standardization order ($\ltst$).  That is, if $t$ is unprimed, then $\ybox$ moves forward in reading order, and if $t$ is primed then $\ybox$ moves backward in reading order. If a valid switch is not a hop, then its inverse is a hop, and we refer to such switches as \newword{inverse hops}.  We also say that $\ybox$ \newword{hops across} a $t$ or \newword{inverse hops across} a $t$ respectively.
\end{definition}

\begin{remark}\label{rmk:exceptional}
In the case where $t$ is the first $i$ or $i'$, this definition depends on the choice of representative.  For example both $1\ybox1 \leadsto \ybox11$ and $1'\ybox1 \leadsto \ybox1'1$ are valid switches, but the former is a hop and the latter is an inverse hop. Even though, in this case, the result of switching is well-defined on equivalence classes, these should be regarded as two different valid switches.  
\end{remark}

\subsection{The switching algorithm}

We now describe $\esh$ as a sequence of switches.

\begin{thm}[Switching algorithm for $\esh$]\label{thm:step-by-step-algorithm}
  Let $(\ybox,T)\in \LRyb$.  Then $\mathrm{esh}(\ybox,T)$ can be computed as follows, 
starting at $i=1$, and stopping when we reach $i=\ell(\beta)+1$.  
  \begin{itemize}
     \item \textbf{Phase 1:} 
	Begin with $T$ in canonical form. 	
    
    If there is an $i'$ after $\ybox$ in reading order, hop $\ybox$ across an $i'$, then hop across an $i$.  Increment $i$ by $1$ and repeat.
    
    If no such $i'$ exists, go to Phase 2.
 
     \item \textbf{Phase 2:} 
If $\ybox$ most recently switched with an entry earlier in reading order, or $\ybox$ has not yet moved, enter Phase 2(a) below; otherwise, skip to Phase 2(b). 
 \begin{itemize}
   \item \textbf{Phase 2(a):} If $\ybox$ precedes all $i$ and $i'$ entries in reading order, skip to Phase 2(b).
   
   Change $\first(i)$ to $i'$, then perform as many valid inverse hops across $i'$ as possible.

If the $\ybox$ now precedes all $i$ and $i'$ entries in reading order, go to Phase 2(b).  

Otherwise, if the $i,i+1$-reading word has the form $\ldots i(i+1)^\ast \ybox \ldots$, hop $\ybox$ across the $i$. 

Increment $i$ by $1$ and repeat Phase 2.
 
  \item \textbf{Phase 2(b):}  Perform as many valid inverse hops across $i$ as possible.
 
If the $i,i+1$-reading word has the form $\ldots\ybox (i')^\ast(i+1')\ldots$, hop $\ybox$ across the $i+1'$.  

Increment $i$ by $1$ and repeat Phase 2. 
  
 \end{itemize}
  \end{itemize}
\end{thm}

We refer to a switch occurring in the algorithm above as a \newword{step} of
the algorithm.

\begin{example}\label{ex:algorithm}
  Let $(\ybox,T)$ have reading word 
\[
  3233'42'232'2'121\ybox 1'12111111.
\]
Then the algorithm runs as follows.
\begin{center}
\begin{tabular}{llll}
\textbf{Step}
& $i$
& \textbf{Word}
& \textbf{Choice of representative}
\\ \hline
Start      & $1$ & $3233'42'232'2'121\ybox 1'12111111$ & Begin with canonical form.\\
Phase 1    & $1$ & $3233'42'232'2'1211'\ybox 12111111$ \\
Phase 1    & $1$ & $3233'42'232'2'12\ybox 1'112111111$ \\
Phase 2(a) & $2$ & $32'33'42'232'\ybox 122'1'112111111$ & First $2$ changes to $2'$. \\
Phase 2(a) & $2$ & $32'33'42'23\ybox 2'122'1'112111111$ \\
2(a) (hop)   & $2$ & $32'33'42'\ybox 322'122'1'112111111$ \\
Phase 2(a) & $3$ & $3'2'3\ybox 42'3'322'1'22'1'112111111$ & First $3$ changes to $3'$. \\
Phase 2(a) & $3$ & $\ybox\  2'33'42'3'322'122'1'112111111$ \\
Phase 2(b) & $3$ & $32'\ybox 3'42'3'322'122'1'112111111$ \\
Phase 2(b) & $4$ & $32'43'\ybox 2'3'322'122'1'112111111$
\end{tabular}
\end{center}
\end{example}

We will show that every step in the switching algorithm is a valid switch.
This is not obvious, and is asserted by part (1) of the following theorem.  

\begin{thm}\label{thm:phase12-ballotness}
  Let $T$, including the $\ybox$, be a tableau that appears after some step of the computation of $\esh(\ybox,T_1)$ for some pair $(\ybox,T_1)\in \LRyb$.  Then:
  \begin{enumerate}
    \item[(1)] Omitting the $\ybox$, the reading word of $T$ is ballot.
    \item[(2)] Omitting the $\ybox$, the tableau is semistandard.
  \end{enumerate}
\end{thm}
Note that the same properties hold for the tableaux arising as intermediate steps during ordinary jeu de taquin.   In the next section, we will discuss two more properties of the switching algorithm that are analogues of jeu de taquin properties (Theorem~\ref{thm:pause-esh}).  Unfortunately, the individual steps of the switching algorithm are not coplactic operations in any meaningful sense.  This will make the proofs considerably more technical than the proof of Theorem~\ref{thm:crystal-algorithm}.  We will prove Theorems \ref{thm:step-by-step-algorithm} and \ref{thm:phase12-ballotness} in Section \ref{sec:proofs} below.  

For now, here is a brief outline of what we prove, to relate Theorem~\ref{thm:step-by-step-algorithm} to Theorem~\ref{thm:crystal-algorithm}.  
In Phase 1, the operator
$E_{i-1}' \circ E_{i-1}^{\mathrm{wt}(T)_i-2}$ decrements all but three of the entries in the $i-1',i-1,i',i$-subtableau; at the end of Phase 1, we increment all these entries, and the net result is that they are unchanged.  The other three entries correspond to the $\ybox, i,  i'$ identified in the switching description of Phase 1, and the operator changes these to $i',i,\ybox$ respectively.  We show that the transition from Phase 1 to Phase 2 occurs for the same $i$
in both algorithms.  Call this transition point $i=s$, and let $T'$ denote the tableau at the start of Phase 2 in the coplactic algorithm (after setting $\ybox = s'$).
At each Phase 2 step of the switching algorithm, the tableau is
essentially just $F_{i-1} \circ \dots \circ F_s(T')$, but with $\ybox$ 
replacing one of the entries (either an $i$ or an $i'$).  
The $\ybox$ moves as follows:
it first switches through the $i/i'$ subword in standardization order, until it reaches the final $F_i$-critical substring.
Then, the four possible ways Phase 2 can loop (either in Phase 2(a) or Phase 2(b) and with or without a hop) correspond to the four possible types for this critical substring, and in each case the movement of $\ybox$ and incrementing
of $i$ corresponds to applying the operator $F_i$.

\begin{remark}
The best practical way to compute $\esh(\ybox, T)$ is to use a hybrid
of the two algorithms.  The switching algorithm is easier and more efficient
for Phase 1, and the coplactic algorithm is better for Phase 2. For K-theoretic purposes, however, the switching algorithm connects to genomic tableaux (see Section \ref{sec:K-theory}).
\end{remark}

\subsection{Index decomposition}

In the switching algorithm, $\ybox$ first switches with $1'$ entries, then $1$ entries, then $2'$, then $2$, and so on.  We therefore wish to define the \emph{index} of a switch as the letter the $\ybox$ switches with, with one exceptional case for the situation described in Remark \ref{rmk:exceptional}. We first make the following definition in anticipation of this boundary case.

\begin{definition}\label{defn:exceptional}
A valid switch is \newword{exceptional} if $\ybox$ switches with 
$i' = \first(i')$, $\ybox$ moves backward in 
reading order, and there are no entries equal to $i$ between $i'$ and 
$\ybox$ in reading order.
\end{definition}
For example, the switch $1'\ybox1 \to \ybox1'1$ is exceptional, but $1\ybox1 \leadsto \ybox11$ and $1'1\ybox \leadsto \ybox11'$ are valid switches that are not exceptional.  Note in particular that exceptional switches are inverse hops, and inverses of exceptional switches are hops. 

\begin{definition}\label{defn:switchindex}
The \newword{index} of a switch between $\ybox$ and $t$ is defined to
be $t$, \emph{unless} $t=i'$ and the switch is exceptional or the
inverse of an exceptional switch, in which
case the index is defined to be $i$.
\end{definition}

We emphasize that although the exceptional switch has index $i$ instead of $i'$, it is still considered a switch with an $i'$ and the $i'$ involved in the switch remains an $i'$ after the switches.

This leads to the following definition, which plays a significant role in analyzing orbits of $\omega$, and thereby real connected components of Schubert curves (see Section \ref{sec:K-theory}).

\begin{definition}\label{def:index-decomposition}
For $t=1,1',2,2',\dots$, we define $\partesh_t$ to be the operation of perfoming
all steps of index $t$ in the switching algorithm.  This gives the 
\newword{index decomposition}.
\[
  \esh = \partesh_{\ell(\beta)} \circ \partesh_{\ell(\beta)'} \circ \dots \circ \partesh_2 \circ \partesh_{2'} \circ \partesh_1 \circ \partesh_{1'}.
\]
\end{definition}

Note that an exceptional switch of $\ybox$ with $i'$ must be the last switch involving $i'$, and is defined to have index $i$ (Definition \ref{defn:switchindex}). Grouping exceptional switches with $\partesh_i$ rather than $\partesh_{i'}$ turns out to be necessary for the index decomposition to consist of well-defined bijections.

We observe that for all symbols $t$, $\partesh_t$ does one of three things:
\begin{itemize}
\item perform a single hop, 
\item perform an exceptional switch followed by a sequence of zero 
or more inverse hops, or
\item perform a sequence of zero or more inverse hops, with no 
exceptional switches.
\end{itemize}

The most important property of the index decomposition is that it 
identifies points
at which we can ``pause and resume'' the algorithm.  Here,
\newword{pause} means that we stop partway through the algorithm, forget what 
phase of the algorithm we were in, and forget what representative we are
currently working with (e.g. by putting the tableau back into canonical
form).
In general we cannot pause the algorithm at an arbitrary point, and be able 
to resume, as the following example shows.

\begin{example}
Let $w = 22'1'12'\ybox11$ and $v = 22'1'1'2'\ybox11$, differing only in the fourth letter. For both words, $\esh$ begins in Phase 2(a), so we have shown the representative in which the first $1$ is primed.  If, for both words, we stop the
algorithm when we reach the end of Phase 2(a), we find that $w$ and $v$ have
transformed into equivalent words.
\[w \leadsto 22'\ybox12'1'11, \quad v \leadsto 22'\ybox1'2'1'11.\]
Thus $\esh$ cannot be paused and resumed from this point.

Note, however, that the last switch performed on $v$ is exceptional, so we have gone one step beyond computing $\partesh_{1'}(v)$.  If instead we
pause the algorithm after performing all steps of index $1'$ on $v$ and $w$,
the words we obtain are not equivalent.
\end{example}

To state this pause and resume property more precisely, we make the 
following definition.

\begin{definition}
Define $Z_{i}=Z_{i}(\alpha,\beta,\gamma)$ to be the set of all ballot semistandard tableaux of shape $\gamma^c/\alpha$ and content $\beta$, with $\ybox$ being an outer co-corner of the letters $1',1,\ldots,i$ and an inner co-corner of the remaining letters. Let $Z_{i'}$ denote the corresponding set in which the $\ybox$ is an outer co-corner of the letters $1', 1, \ldots, i{-}1, i'$, and an inner co-corner of the remaining letters, in canonical form (the first $i$ or $i'$ treated as unprimed).
\end{definition}

In particular $Z_0 =  \LR(\alpha, \ybox, \beta, \gamma)$, and 
$Z_{\ell(\beta)} = \LR(\alpha, \beta, \ybox, \gamma)$.

\begin{thm}\label{thm:pause-esh}
Let $t \in \{1',1,2',2, \dots\}$.
\begin{enumerate}
\item[(3)] 
For $T \in \LR(\alpha, \ybox, \beta, \gamma)$, we have
\[
  \partesh_t \circ \dots \circ \partesh_1 \circ \partesh_{1'}(T) \in Z_t(\alpha, \beta,\gamma).
\]
\item[(4)]
The map
\[
  \partesh_t \circ \dots \circ \partesh_1 \circ \partesh_{1'}: Z_0 \to Z_t
\]
is a bijection.
\end{enumerate}
\end{thm}

Informally, statement (3) says that $Z_t$ is the set of all possible 
tableaux we could reach if we pause after performing all steps of index $t$.
Statement (4) says that the algorithm is uniquely reversible, and
hence resumable, after pausing at this point.  
In particular, this means we can regard $\partesh_{i'}$ and $\partesh_i$ as
well-defined bijections from one paused state to the next:
\[
  \partesh_{i'} : Z_{i-1} \to Z_{i'}\,, \qquad
  \partesh_{i} : Z_{i'} \to Z_i.
\]
We will prove (3) for Phase 1 and Phase 2 separately, along 
with (1) and (2) from Theorem~\ref{thm:phase12-ballotness}.
Statement (4) follows from the following theorem,
which tells us explicitly how to compute these
bijections starting from the paused state.

\begin{thm}\label{thm:resume-esh}
For $t=1',1,2',2,\dots$, let $T \in Z_{i-1}$ if $t=i'$, 
and $T \in Z_{i'}$ if $t=i$.
The following recipe computes $\partesh_t(T)$, reproducing the sequence 
of switches of index $t$ in the switching algorithm.
\begin{itemize}
\item 
\textbf{Case 1:}
Set $\first(t,T)$ to be primed.
If it is possible to perform an exceptional switch of index $t$, first perform
this switch; then perform as many inverse hops of index $t$ as possible, 
and stop.  (Note that this case can only occur if $t$ is unprimed.)

\item 
\textbf{Case 2:}
Otherwise, set $\first(t,T)$ to be unprimed.
If it is possible to perform a hop of index $t$, do so and stop.

\item 
\textbf{Case 3:}
Otherwise, set $\first(t,T) = t$. Perform as many inverse hops of index $t$ as 
possible.
\end{itemize}
\end{thm}

The properties described in Theorem~\ref{thm:pause-esh} are also properties
of ordinary jeu de taquin.  We can decompose the jeu de taquin
map $\sh : \LRby \to \LRyb$ similarly.  Let $\partsh_i : Z_i \to Z_{i'}$ and
$\partsh_{i'} : Z_{i'} \to Z_{i-1}$ be the partial jeu de taquin bijections,
defined by sliding the box through the $i$-entries and $i'$-entries
respectively.  Then we have the decomposition
\[
  \sh = \partsh_{1'} \circ \partsh_1 \circ \partsh_{2'} \circ \partsh_2 \circ \dots \circ
   \partsh_{\ell(\beta)'} \circ \partsh_{\ell(\beta)}.
\]

Combining this decomposition with the index decomposition of $\esh$, we have
the following diagram in which each map is a bijection: 
\begin{equation}\label{eqn:factorization}
\xymatrix{
Z_{0} \ar@/_10pt/[r]_-{\partesh_{1'}} &
Z_{1'} \ar@/_10pt/[r]_-{\partesh_1} \ar@/_10pt/[l]_-{\partsh_{1'}}&
Z_{1} \ar@/_10pt/[r]_-{\partesh_{2'}} \ar@/_10pt/[l]_-{\partsh_{1}} &
Z_{2'} \ar@/_10pt/[r]_-{\partesh_{2}} \ar@/_10pt/[l]_-{\partsh_{2'}} &
\cdots
\ar@/_10pt/[r]_-{\partesh_{\ell(\beta)}} \ar@/_10pt/[l]_-{\partsh_{2}} &
Z_{\ell(\beta)} \ar@/_10pt/[l]_-{\partsh_{\ell(\beta)}}.
} 
\end{equation}
This diagram will play an important role in analyzing
the orbit structure of $\omega$, in Section~\ref{sec:K-theory}.

\subsection{Reversing the algorithm}
\label{sec:reverse-algorithm}

Since there are a few of subtle details in reversing the switching algorithm, we provide a complete statement (without proof) of the algorithm for $\esh^{-1}$ in terms of switches.  The main issue here is identifying the choice of representative to use at each reverse step, which does not always match the description in the forward algorithm.

\begin{thm}\label{thm:reverse-algorithm}
  Let $(\ybox,T)\in \LRyb$.  Then $\mathrm{esh}^{-1}(\ybox,T)$ can be computed as follows, starting with $i = \ell(\beta)$, and stopping when we reach $i=0$.

Begin with $T$ in canonical form. If the word formed by replacing $\ybox$ by $\ell(\beta)+1$ is ballot, go to Reverse Phase 1, and otherwise start in Reverse Phase 2(b) below.
  \begin{itemize}
     \item \textbf{Reverse Phase 2:}  If $\ybox$ precedes all $i$ and $i'$ entries in reading order, change $\first(i)$ to $i'$.
     
If $\ybox$ most recently switched with an entry earlier in reading order, or $\ybox$ has not yet moved, enter Reverse Phase 2(b) below; otherwise, skip to Reverse Phase 2(a).
 \begin{itemize}
   \item \textbf{Reverse Phase 2(b):} Perform as many valid hops across $i$ as possible.

If $\ybox$ now precedes all $i$ and $i'$ entries in reading order, go to Reverse Phase 2(a).

Then, if replacing $\ybox$ by $i$ results in a ballot $i-1,i$-subword, decrease $i$ by $1$ and go to Reverse Phase 1.

Otherwise, if the $i-1,i$-reading word has the form $\cdots (i')(i-1')^\ast \ybox \cdots$, inverse hop $\ybox$ across the $i'$.

Decrease $i$ by $1$ and repeat Reverse Phase 2.
 
  \item \textbf{Reverse Phase 2(a):}  Perform as many valid hops across $i'$ as possible.
  
     Then, if replacing $\ybox$ by $i'$ results in a ballot $i-1,i$-subword, decrease $i$ by $1$ and go to Reverse Phase 1.  
     
     Otherwise, if the $i-1,i$-reading word has the form $\cdots\ybox (i)^\ast i-1\cdots$, inverse hop $\ybox$ across the $i-1$. 

Decrease $i$ by $1$ and repeat Reverse Phase 2.
 \end{itemize}
     \item \textbf{Reverse Phase 1:} Change $\first(i)$ to $i'$.

Inverse hop $\ybox$ across an $i$, then inverse hop across an $i'$. Decrease $i$ by $1$ and repeat until $i=0$.
 \end{itemize}
\end{thm}

We also state the inverse algorithm for the operations
$\partesh_i^{-1} : Z_i \to Z_{i'}$, and $\partesh_{i'}^{-1} : Z_{i'} \to Z_{i-1}$.

\begin{thm}
\label{thm:reverse-resume-esh}
For $t=1',1,2',2, \dots$, and $T \in Z_t$, the following recipe computes
$\partesh_t^{-1}(T)$, reproducing the inverse of the steps in 
Theorem~\ref{thm:resume-esh}.
\begin{itemize}
\item \textbf{Case A:}
Set $\first(t,T)$ to be primed.  If it possible to perform
an inverse hop of index $t$, then do so, giving priority to the non-exceptional switch if there is a choice; then stop.

\item \textbf{Case B:}
Otherwise, keep $\first(t,T)$ primed.
If it is possible to perform
an inverse exceptional switch of index $t$, then do so and stop.

\item \textbf{Case C:}
Otherwise, set $\first(t,T)$ to be unprimed.  
Perform as many hops of index $t$ as possible, and stop.  (If $t$ is unprimed, the last step here could be an inverse exceptional switch.)
\end{itemize}
\end{thm}

Note that the three cases here do not quite correspond to the
three cases of Theorem~\ref{thm:resume-esh}, and the order of precedence is different. Case A of Theorem \ref{thm:reverse-resume-esh} (a single inverse hop) inverts Case 2 of Theorem \ref{thm:resume-esh} (a single hop). The sequence of hops of Case C inverts the sequence of inverse hops of Case 3 and also of Case 1, when when the latter produces more than just the exceptional step. Case B also inverts Case 1, when the latter produces only the exceptional step.

\begin{remark} \label{rmk:reverse-esh-lowest-weight}
It is also possible to compute $\esh^{-1}(T, \ybox)$ by replacing $T$ by the unique lowest-weight (anti-ballot) tableau in its dual equivalence class, then performing a `reflected' form of the algorithms given in this paper. Since we work entirely with highest-weight tableaux, we omit the precise statements.
\end{remark}

\subsection{Proof of the switching algorithm and its properties}\label{sec:proofs}

This section is devoted to proving Theorems \ref{thm:step-by-step-algorithm}, \ref{thm:phase12-ballotness}, \ref{thm:pause-esh}, and \ref{thm:resume-esh}.

\subsubsection{Phase 1 proofs}\label{sec:phase1-proof}

We now show that Phase 1 can be described as in Theorem \ref{thm:step-by-step-algorithm}.

\begin{thm}\label{thm:phase1-algorithm}
  Let $(\ybox,T)\in \LRyb$.  Then Phase 1 of the algorithm of Theorem \ref{thm:crystal-algorithm} for computing $\esh(\ybox,T)$ agrees with Phase 1 of the algorithm of Theorem \ref{thm:step-by-step-algorithm}.
\end{thm}

\begin{proof}
As noted in Lemma~\ref{lem:partial-esh}, it suffices to analyze Phase 1 in the
``Pieri case'', where all entries of $T$ are $1$ or $1'$.   Assume that
$T$ has this form. Let $m$ denote the number of entries in $T$,
and let $k$ denote the number of these entries equal to $1'$,
with $T$ in canonical form.
Let $T^+$ be the tableau obtained by replacing each 
$1/1'$ in $T$ by a corresponding $2/2'$, and replacing $\ybox$ by a $1'$.

We first show that the two algorithms agree on when to enter Phase 2.  
For this, we need to show that there exists a $1'$ after $\ybox$ in $T$ 
in reading order if and only if $E_1(T^+) \neq E_1'(T^+)$.  We use
\cite[Corollary 5.40]{bib:GLP}, which states that $E_1(T^+)=E_1'(T^+)$ if and
only if the rectification shape of $T^*$ has only one row, and
\cite[Theorem 4.3]{bib:GLP}, which tells us that $T^+$ has only one row 
if and only if all
steps in the walk point up or right.  Since $T^+$ has no $1$s,
its walk cannot have a down step, and the only way to have a left
step is to have a $2'$ after the $1'$, which corresponds to a $1'$ in $T$
after the $\ybox$.  Thus the two conditions for when to enter Phase 2
agree.

Now, assume that we do not immediately switch to Phase 2.  Then
Theorem~\ref{thm:crystal-algorithm} says that 
$\esh(\ybox, T)$ is obtained by computing $E_1' \circ E_1^{m-2}(T^+)$, and
replacing the only $2$ by $\ybox$.  We must show that this has
the same effect as switching $\ybox$ with the next $1'$ in $T$,
and then switching with the preceding $1$ in $T$.  Denote the corresponding
entries of $T^+$ by $t$ and $u$: $t$ is the first $2'$ after the unique $1'$ in
$T^+$ and $u$ is the first $2$ before $t$.  (See Figure \ref{fig:E-walks}.)
  
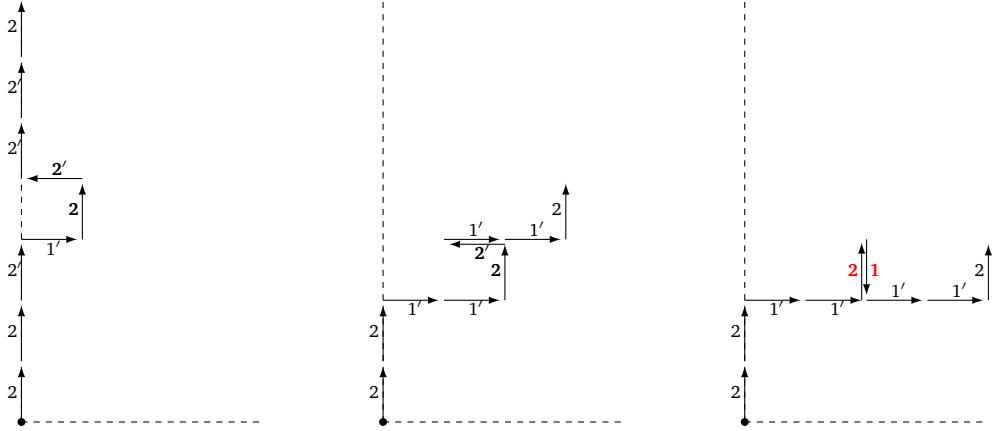
\begin{figure}
\begin{center}
\begin{picture}(2,8.2)(0,0)
\multiput(0,0)(0,0.2){35}{\line(0,1){0.1}}
\multiput(0,0)(0.2,0){20}{\line(1,0){0.1}}
\put(0,0){\circle*{0.13}}
%\put(4,2){\circle{0.13}}
\stepnorth{2}{%
\stepnorth{2}{%
\stepnorth{2'}{%
\stepeast{1'}{%
\stepnorth{\mathbf{2}}{%
\stepwest{\mathbf{2'}}{%
\stepnorth{2'}{%
\stepnorth{2'}{%
\stepnorth{2}{%
}}}}}}}}}
\end{picture} \hspace{3cm}
%%%%%%%%%%%%%%%
\begin{picture}(2,7.2)(0,0)
\multiput(0,0)(0,0.2){35}{\line(0,1){0.1}}
\multiput(0,0)(0.2,0){20}{\line(1,0){0.1}}
\put(0,0){\circle*{0.13}}
%\put(4,2){\circle{0.13}}
\stepnorth{2}{%
\stepnorth{2}{%
\stepeast{1'}{%
\stepeast{1'}{%
\stepnorth{\mathbf{2}}{%
\stepwestshiftS{\mathbf{2'}}{%
\stepeastA{1'}{%
\stepeastA{1'}{%
\stepnorth{2}{%
}}}}}}}}}
\end{picture} \hspace{3cm}
%%%%%%%%%%%%%%%%%
\begin{picture}(2,7.2)(0,0)
\multiput(0,0)(0,0.2){35}{\line(0,1){0.1}}
\multiput(0,0)(0.2,0){20}{\line(1,0){0.1}}
\put(0,0){\circle*{0.13}}
%\put(4,2){\circle{0.13}}
\stepnorth{2}{%
\stepnorth{2}{%
\stepeast{1'}{%
\stepeast{1'}{%
\stepnorthshiftW{\mathbf{\color{red} 2}}{%
\stepsouth{\mathbf{\color{red} 1}}{%
\stepeastA{1'}{%
\stepeastA{1'}{%
\stepnorth{2}{%
}}}}}}}}}
\end{picture} 
\end{center}
\caption{\label{fig:E-walks} From left to right, the walks of $T^+$, 
$E_1^{k-1}(T^+)$, and $E'_1 \circ E_1^{k-1}(T^+)$, where $T$ is a tableau with reading word $111'\ybox 11'1'1'1$.
The special entries, 
$t$ and $u$, are red and in boldface.}
\end{figure}

Suppose there are $a$ entries equal to $2'$ after $t=2'$.  Then the last such is the final $E$-critical substring by rule 3E.  So this changes to a $1'$ upon applying $E_1$.  By the same reasoning, applying $E_1$ again $a-1$ more times changes the remaining $2'$ entries after $t$ to $1'$, and the result is $E_1^a(T^+)$.  

Note that the tail of the walk after $t$ remains weakly above $t$ in $E_1^a(T^+)$ (and hence above the $x$-axis), since all $1'$ entries point east and all $2$ entries north.  Since the only type of critical substring starting at a $2$ rather than a $2'$ is type 4E and must start at the $x$-axis, if there are any $2'$ entries remaining on the $y$-axis then the final critical substring is not type 4E, and starts at a $2'$.

There are now $b=k-a-1$ entries equal to $2'$ before $t$.  If $b>0$, then since the first $2/2'$ is unprimed, there are at least two upward arrows before $t$.  So $t$ starts at $y\ge 2$ and therefore cannot be the start of an $E_1$-critical substring (since its only possible type is 2E).  Therefore, the final $E$-critical substring is begins at the highest $2'$ on the $y$-axis; it is either of type 3E, or 1E (ending at $\first(1)$, which in our chosen representative is $1'$). In either case, the transformation rule changes this $2'$ to a $1'$.  Repeating this argument, we find that $E_1^{k-1}(T^+)$ simply changes each $2'$ other than $t$ to $1'$.  (See the second diagram in Figure \ref{fig:E-walks}.)

Since $E_1'$ commutes with $E_1$, we may apply $E_1'$ at this point in the process rather than after all applications of $E_1$.  Note that applying $E_1'$ to $E_1^{k-1}(T^+)$ changes $t$ into a $1$, and there are then no $2'$ entries remaining in $E_1'(E_1^{k-1}(T^+))$.

To apply the remaining $E_1$ operators, we see a similar process on the $2$s.
Since there are no more $2'$s, the final $E_1$-critical substring begins with a $2$ starting on the $x$-axis, and is either of type 4E, or is of type 1E/2E, beginning with $\first(2)$.  In either case, the transformation changes the first $2$ to a $1$, excluding $u$, which can never be the final critical substring since it is followed by $(1')^*1$.  Thus, applying $E_1^{m-k-2}$ now changes all remaining $2$s, except for $u$, to $1$s.  Therefore, the total effect of the operator
$E_1' \circ E_1^{m-2} = E_1^{m-k-1}\circ E_1'\circ E_1^{k-1}$ on $T^+$
is to change every $2'$ to $1'$ except for $t$, which changes to $1$, and to change every $2$ to $1$ except for $u$, as desired.
\end{proof}

We now establish Theorem \ref{thm:phase12-ballotness} on ballotness and 
semistandardness of tableaux arising during Phase 1, as well as property (3) of Theorem \ref{thm:pause-esh} on the relative position of $\ybox$ with respect to the other entries at a step in the algorithm.  The argument relies on the following lemma.

\begin{lemma}\label{lem:oversight}
Let $T$ be a skew shifted semistandard tableaux in entries 
$1,1',2,2'$, and a single $\ybox$, which are ordered as either
$1' < 1 < \ybox < 2' < 2$, or
$1' < 1 < 2' < \ybox < 2$.
Let $\mu$ be the rectification shape $T$ (defined to be 
the rectification shape of its standardization).
If $\mu_1 = \mathrm{wt}(T)_1$, then
the reading word of $T$, omitting the $\ybox$, is ballot.
\end{lemma}

\begin{proof}
Since ballotness depends only on the word of $T$, we
may assume that the entries of $T$ are in distinct rows and columns,
so that deleting $\ybox$ from $T$ results in a tableaux $T'$.
Let $\nu$ denote the rectification shape of $T'$.
Since $T'$ is obtained by deleting entries from $T$,
$\nu$ is contained in $\mu$. (See Lemma 6.3.9 of \cite{Worley}.)
But since $\nu_1 \geq \mathrm{wt}(T')_1 = \mathrm{wt}(T)_1 = \mu_1$,
we must have $\nu_1 = \mathrm{wt}(T')_1$.  Since $\nu$ has only two rows,
this means $\nu = \mathrm{wt}(T')$, i.e.
$T'$ is a Littlewood-Richardson tableau.
\end{proof}

\begin{thm}\label{thm:phase1-ballotness}
Properties (1)--(3) stated in Theorems \ref{thm:phase12-ballotness} and \ref{thm:pause-esh} hold after each Phase 1 switch.
\end{thm}

\begin{proof}
Let $T'$ and $T''$ denote the tableaux that occur in Phase 1 immediately before $\ybox$ switches with $i'$ and $i$, respectively.

From the coplactic algorithm it is clear the $T'$ satisfies
(2) and (3).
We show that $T'$ satisfies (1).  We need to show that for all $j$,
the subword consisting of $j-1',j-1,j',j$ is ballot.  For $j \neq i$, this 
follows from coplacticity, since this is true in the rectified case.
For $j=i$, consider the subtableau $S$ of $T_{i'}$ defined by entries
$i-1' < i-1 < \ybox < i' < i$, replacing these 
by $1'< 1< \ybox< 2'< 2$.  Let $\mu$ be the rectification shape
of $S$.  In the case where $\ybox \sqcup T$, we find that 
$\mu_1 = \mathrm{wt}_1(S)$, and by coplacticity, this must be true in general.
Therefore, by Lemma~\ref{lem:oversight}, we deduce that the subword 
for $j=i$ is ballot.

We now show that $T''$ satisfies (1). Again, we need to show ballotness 
of the $j-1',j-1,j',j$-subword for all $j$.  For $j<i$ and $j>i+1$ this 
subword is the same as that of $T'$, so this case is done.  Note that
the word of $T_{i}$ is obtained from the word of $T'$, by moving a 
single $i'$-entry earlier.   By Corollary~\ref{cor:ballotness-preserving},
the $j=i+1$ case also follows from the ballotness of $T'$.  
Finally, for $j=i$, consider the subtableau $R$ of $T_i$
defined by entries $i-1' < i-1 < i' < \ybox < i$, replacing these 
by $1'<1< 2'< \ybox<2$.  Let $\nu$ be the rectification shape of $R$.
Note that $R$ is the tableau obtained by applying a Phase 1 step 
to $S$.  Let $S^+$ be the tableau obtained by replacing each
$2'/2$ entry by a corresponding $3'/3$, and setting $\ybox=2$.
The proof of Theorem~\ref{thm:phase1-algorithm} 
shows that if $k$ is the number of $2'$ entries in $S$ (in canonical form), 
then $E_2^{k-1}(S^+)$ is the tableau obtained from $R$ by replacing
each $2$ by a $3$, and setting $\ybox = 3'$.  This immediately 
implies that $T''$ satisfies (2) and (3).  Moreover $E_2^{k-1}(S^+)$
has the same standardization as $R$.  Therefore $\nu$ 
is the rectification shape of $E_2^{k-1}(S^+)$, which 
(since $E_2$ is coplactic, and $S^+$ has the same standardization 
as $S$) is the rectification shape of $S$, i.e $\nu=\mu$.  
Therefore, by Lemma~\ref{lem:oversight}, 
we again conclude that subword for $j=i$ is ballot.
\end{proof}

\subsubsection{Removability Lemmas}

In order to prove Phase 2 of the switching algorithm, we first require several lemmas about removability of entries as defined below. We first require a local notion of ballotness.

\begin{definition}
  A word $w$ is \textbf{$i$-ballot} if $E_i(w)=E'_i(w)=0$, or equivalently, if its $i',i,(i+1)',i+1$-subword is ballot.
\end{definition}

We now define removability of entries.

\begin{definition}
  An entry of a word $w$ is \textbf{removable} if deleting this entry from $w$ results in a ballot word.  We also say an entry is \textbf{$i$-removable} if deleting the entry results in an $i$-ballot word.
\end{definition}

\begin{example}
  If $w=211'11'2'22'11$, then the last two $1$ entries are not removable, but the other $1$ and $1'$ entries are removable. 
\end{example}

The aim of this section is to study \emph{pairs} of removable entries, consecutive in standardization order ($\ltst$).  This allows us to consider certain switches of the algorithm as a pair of removable entries in a tableau, rather than as a moving $\ybox$.  The main result is that consecutive removable pairs always lie in a fixed interval in standardization order.  Furthermore, individual entries with value $i'$ or $i$, and $(i+1)'$ or $i+1$, that are $i$-removable \emph{almost} form an interval, except for two cases in which there is an isolated additional removable entry.

We begin with a technical lemma about ballot walks.  We recall the following proposition from \cite{bib:GLP}.

\begin{proposition}[Bounded error] \label{lem:bounded-error}
Let $w$ be a word in $\{i',i,(i+1)',(i+1)\}$ and consider the walk for $w$ beginning at an arbitrary starting point $(x_0,y_0)$, not necessarily the origin. 

If we shift the start by either $\west{}$ or $\north{}$, then the endpoint also shifts by \emph{either} $\west{}$ or $\north{}$.

Similarly, if we shift the start by $\east{}$ or $\south{}$, the endpoint also shifts by $\emph{either}$ $\east{}$ or $\south{}$.
\end{proposition}

We will need the following variant of this proposition specifically for ballot words in the letters $i',i,(i+1)', i+1$.

\begin{lemma}[Tail errors]\label{lem:tail-errors}
  Suppose the final $F_i$-critical substring of a ballot word $w$ in $\{i',i,(i+1)',(i+1)\}$ ends at $w_j$.  Consider any tail of the walk starting after $w_j$.  If the start point of this tail is shifted one step left, then the entire tail shifts one step left.  If no $(i+1)'$ in the tail starts on the $x$-axis and the start point of the tail is shifted one step up, then the entire tail shifts one step up.
\end{lemma}

\begin{proof}
  For simplicity, assume $i=1$.  In the walk, note that any $1$ that occurs after $w_j$ cannot start on the $x$-axis, for otherwise it would form a later type 3F critical substring.  Such a $1$ also cannot start on the $y$-axis, since the walk eventually reaches the $x$-axis again and therefore either a 2F or 5F critical substring would occur.   Furthermore, a $1$ or $2'$ after $w_j$ cannot start on the line $x=1$ since this would be a later 5F critical substring.  
  
  If we shift the start point of the tail one step left, then every $1'$ or $2$ does not change direction upon shifting their start point left, and by the above analysis any $2'$ or $1$ also does not change direction.  It follows that the entire tail shifts one step left.
  
  Now, suppose we shift the start point of the tail one step up and no $2'$ in this tail starts on the $x$-axis.   Then by the above reasoning the entire tail shifts one step up as well.
\end{proof}

We now establish the fact that $i$-removable entries $i$ or $i'$ in a ballot word, with one exception, all lie on one side of a ``break point'' in standardization order.  Recall that we use $\ltst$ to denote standardization order; we use this notation heavily in the next several lemmas.  

\begin{lemma}[Lower Break Point]\label{lem:break-point}
  Let $w$ be $i$-ballot.   If $F_i(w)$ is defined, let $w_j$ be the first (resp.\ last) letter of the final $F_i$-critical substring if the substring is type 1F or 3F (resp.\ 2F or 4F).  If the substring is type 2F, also let $w_a=i$ be the first letter of the critical substring.  
  
  Then an $i$ or $i'$ entry $w_k$ is $i$-removable in $w$ if and only if either:
  \begin{itemize}
   \item $w_k\leqst w_j$, or 
   \item $w_k=w_a$ in the 2F case and its arrow in the $i,i+1$-walk starts on the $y$-axis.  
   \end{itemize}
   If $F_i(w)$ is not defined, then no $i$ or $i'$ entry is $i$-removable.  
   
\end{lemma}

\begin{proof}
  For simplicity we assume that $i=1$ and that $w$ only contains the letters $1', 1, 2', 2$, and we refer to removability rather than $1$-removability.
  
  \textbf{Case 1:} Suppose the final $F_1$-critical substring is type 1F or 3F, and let $w_j=1$ be its first letter.  Let $w_k$ be an arbitrary $1$ or $1'$ entry.
  
  First suppose $w_j\ltst w_k$.  Then $w_k=1$ and $j<k$. Note that a $2'$ occuring after $w_k$ cannot start on the $x$-axis, since the start point of $w_k$ is above the $x$-axis and so there is some $1$ (possibly $w_k$) that points down to the $x$-axis just before this $2'$, creating either a 1F or 3F critical substring later in the word.  By Lemma \ref{lem:tail-errors}, upon removing $w_k=1$, the  tail of the walk starting at $w_{k+1}$ shifts up one step.  Therefore it no longer ends on the $x$-axis, and $w_k$ is not removable.
  
  Now suppose $w_k\leqst w_j$.  If $w_k=1$ then $k\le j$.  Thus removing $w_k$ moves the start point of the 1F or 3F critical substring at $w_j$ either one step up or one step left by Lemma \ref{lem:bounded-error}.  A simple analysis of the possible locations for a 1F critical substring shows that this moves the endpoint after this critical substring one step left.  It follows from Lemma \ref{lem:tail-errors} that the entire tail moves one step left and the walk still ends on the $x$-axis.  So $w_k$ is removable in this case.   If instead $w_k=1'$, removing it moves the tail of the walk after it one step left.  Therefore, again by Lemma \ref{lem:tail-errors}, $w_k$ is removable in this case as well.
  
 \textbf{Case 2:} Suppose the final $F$-critical substring is type 2F or 4F, and let $w_j=1'$ be its last letter.  In the 2F case also let $w_a=1$ be the first entry.  Let $w_k$ be an arbitrary $1$ or $1'$ entry.
  
  First suppose $w_k\leqst w_j$.  Then $w_k=1'$ and $k\ge j$.   Removing $w_k$ shifts the tail of the walk after $w_k$ to the left one step, and again by Lemma \ref{lem:tail-errors} we find that $w_k$ is removable in this case.
  
  Next suppose the critical substring is type 2F with $w_a$ starting on the $y$-axis, and we remove $w_a$.  Then $w_a$ points right, so this also shifts the endpoint of the critical substring one step left, and we are done as above.
  
  Now, suppose $w_j\ltst w_k$.  If $k<j$ and $w_k\neq w_a$, then removing $w_k$ moves the start point of the final critical substring either up or left one step.  By analyzing the possible locations for the 4F and 2F critical substrings, it follows that the endpoint of this $F_1$-critical substring moves one step up.  Furthermore, if $w_k=w_a$ and $w_a$ starts off the $y$-axis then the endpoint shifts up as well, upon removing $w_k$.
  
  If the type 4F or 2F critical substring ends above the $x$-axis, then there cannot be a $2'$ after it that starts on the $x$-axis by the same argument as above.   In the unique case when it does end on the $x$-axis, namely a 2F critical substring $11'$ starting at $(1,1)$, then the only way a $2'$ can start on the $x$-axis afterwards is if $2'$ appears after some number of $1'$s following the 2F critical substring.  But then this creates a longer type 1F critical substring, a contradiction.   The tail therefore moves up by Lemma \ref{lem:tail-errors} and therefore the word is no longer ballot. 
  
  Finally, if $k\ge j$, then since $w_j\ltst w_k$ we have $w_k=1$, and removing $w_k$ moves the endpoint up by the same argument as in the 1F/3F case.  This completes the proof.
\end{proof}

We can now define the \emph{lower $i$-break point} of a ballot word for which $F_i$ is defined.

\begin{definition}\label{def:lower-break-point}
Let $w$ be an $i$-ballot word for which $F_i(w)$ is defined, and define $w_a$ and $w_j$ as in Lemma \ref{lem:break-point}.  The \textbf{lower $i$-break point} of $w$ is defined to be $w_a$ if the $F_i$-critical substring is type 2F and $w_a$ is adjacent to $w_j$ in standardization order, and it is defined to be $w_j$ otherwise.
\end{definition}

Note that the lower break point is $w_a$ only when $w_a$ is the first $i$ or $i'$ in reading order, and is therefore removable since it starts on the $y$-axis in the $i,i+1$-walk.  

We now similarly establish an upper break point for removability of $i+1$ and $(i+1)'$ entries in an $i',i,(i+1)',i+1$-subword that is ``nearly ballot'', namely the application of $F_i$ to a ballot word.

\begin{lemma}[Upper Break Point]\label{lem:upper-break-point}
  Let $w$ be an $i$-ballot word for which $F_i(w)$ is defined.    Let $w_j$ be the last (resp.\ first) letter of the final $F_i$-critical substring if the substring is type 1F or 3F (resp.\ 2F or 4F).  If the substring is type 1F, also let $w_a=i$ be the first letter of the critical substring.  Let $v=F_i(w)$ and let $v_j$ and $v_a$ be the transformed letters in $v$.
  
  Then an $i+1$ or $(i+1)'$ entry $v_k$ is $i$-removable in $v$ if and only if either:
  \begin{itemize}
  \item $v_j\leqst v_k$, or
  \item $v_k=v_a$ in the 1F case and its arrow in the $i,i+1$-walk starts on the $x$-axis.
  \end{itemize}
\end{lemma}

The proof is similar in nature to that of Lemma \ref{lem:break-point} (Lower Break Point) and we omit it.

\begin{definition}\label{def:upper-break-point}
  Let $w$ be an $i$-ballot word for which $v=F_i(w)$ is defined, and define $v_j$ and $v_a$ as in Lemma \ref{lem:upper-break-point}.  Then the \textbf{upper $i$-break point} of $v$ is defined to be $v_a$ if the final $F_i$-critical substring is type 1F and $v_a$ is adjacent to $v_j$ in standardization order, and it is defined to be $v_j$ otherwise. 
\end{definition}

Note that the upper break point is $v_a$ only when $v_a$ is the first $i+1$ or $(i+1)'$ in reading order, and is therefore removable since it starts on the $x$-axis in the $i,i+1$-walk.

We now establish the main lemmas on removable pairs.

\begin{lemma}[Removable Pairs]\label{lem:removable-pairs}
  Let $w$ be a word that is both $i$-ballot and $i+1$-ballot.  Suppose $v=F_i(w)$ is defined and is $j$-ballot for all $j\neq i$.  Let $v_u$ be the upper $i$-break point of $v$, and let $w_{\ell}$ be the lower $i+1$-break point of $w$.  Then $v_u\leqst v_{\ell}$.  Moreover, a pair of entries $v_j,v_{k}$ with values $i+1$ or $(i+1)'$ that are consecutive in standardization order are both removable in $v$ iff $$v_u\leqst v_j\leqst v_{k}\leqst v_{\ell}.$$ 
\end{lemma}

\begin{remark}
 If $w$ is ballot, then $w$ and $v=F_i(w)$ satisfy the conditions of the Removable Pairs lemma. To see this, rectify $w$ and $v$ (note that the conditions are coplactic).
\end{remark}

\begin{proof}

 Since $v$ is $i-1$-ballot, an entry $i+1$ is removable in $v$ if and only if it is both $i$-removable and $i+1$-removable.  Thus, for simplicity of notation we may assume that $i=1$, and we consider $1$-removability and $2$-removability of entries.  As a shorthand, we will also simply write `$i,i+1$-subword' to denote the subword consisting of $i',i,(i+1)',(i+1)$ entries, for $i=1,2$.

  We next show that the upper $1$-break point $v_u$ is removable in $v$.  If the final $F_1$-critical substring in $w$ is type 3F or 4F, then $w_u$ is the critical substring, and so applying $F_1$ and removing $v_u$ has no effect on the $2,3$-subword.  In the $1,2$-subword, the effect is to remove a $1$ on the $x$-axis (type 3F) or a $1'$ on the $y$-axis (type 4F), which shifts the tail of the $1,2$-walk left one step, and so it remains ballot by Lemma \ref{lem:tail-errors}.  
  
  Now suppose the final $F_1$-critical substring is type 1F, of the form $1(1')^\ast 2'$, which changes to $2'(1')^\ast 2$.  Then applying $F_1$ and removing $v_u=2$ has the same net effect as removing the first $1$ in the final $F_1$-critical substring and moving the next $2'$ earlier in the word past the sequence of $1'$ letters.  The $2,3$-subword is therefore ballot by Corollary~\ref{cor:ballotness-preserving}.
  
  For the $1,2$-subword in the 1F case, note that if we remove the $1$ at the start of the $F_1$-critical substring, then the $1'$ entries after it do not start on the $y$-axis.  Thus moving a $2'$ to the left past these $1'$ entries does not change the endpoint of the walk after this sequence.  It follows that the effect on the endpoint of the walk is the same as simply removing this $1$ from $w$, which is a removable $1$ by Lemma \ref{lem:break-point}.  Thus the $1,2$-subword is still ballot.
  
  Finally suppose the final $F_1$-critical substring is type 2F, of the form $w_u\cdots w_t=1(2)^\ast 1'$, which transforms to $2'(2)^\ast 1$.  Then after applying $F_1$ to $w$ to form $v$ and removing $v_u=2'$, the net effect on the word is to remove $w_u=1$ and change $w_t=1'$ to $v_t=1$.  Note that the $2,3$-subword is unchanged by this process.  For the $1,2$-subword, removing the $1$ from $1(2)^\ast 1'$ shifts the start point of the $(2)^\ast$ in the $1,2$-walk either up or left by one, and when $w_t$ changes to $1$ the endpoint after this sequence shifts left by one.  By Lemma \ref{lem:tail-errors} the word is still ballot.
  
  We have shown that $v_u$ is removable, and hence by Lemma \ref{lem:break-point}, either the lower $2$-break point $v_{\ell}$ occurs weakly after $v_u$ in standardization order, or $v_u$ is the first $2$ of a final type 2F critical substring for $F_2$ starting on the $y$-axis.  We will show that the latter possibility cannot occur.   Indeed, assume the latter; then $v_u=2$ is the result of either a type 3F or type 1F change on the $1,2$-subword to form $v=F_1(w)$, and we can consider each case separately.
  
  If $v_u$ is the result of a type 3F change, then $w_u=1$, and since $v_u$ is the start of a type 2F critical substring in the $2,3$-subword, the next $2$ or $2'$ to the right of $w_u$ is a $2'$.  But this implies that there is a later type 1F critical substring in $w$ among the $1,2$-subword, a contradiction.
  
  If $v_u=2$ is the last entry of a transformed type 1F substring, say $w_t\cdots w_u=1(1')^\ast 2'$, then $v_t=2'$.  Note that in the $2,3$-walk, we may assume $v_u$ starts at $x=0$ and $y>1$, since if it starts at $(0,0)$ then $v_t$ is adjacent to $v_u$ in standardization order and so $v_u=v_{\ell}$.  Thus $v_t$ ends above the $x$-axis.  It follows that there cannot be a $3'$ after $v_t$ starting on the $x$-axis in the $2,3$-walk, since we would necessarily need a $2$ arrow to get down to the $x$-axis, and this would form a later type 1F critical substring with the $3'$.  
  
  Now, consider the process of forming the $2,3$-subword of $w$ from that of $v$.  We remove $v_t=2'$ and change $v_u$ from $2$ to $2'$.  Removing $v_t=2'$ moves the start point of the tail of the walk after it one step to the left, and so the start point of $v_u$ moves either one step left or one step up by Proposition \ref{lem:bounded-error}.  Since it is on the $y$-axis, it moves one step up, as does the rest of the 2F critical substring.  By Lemma \ref{lem:tail-errors}, since there is no $3'$ starting on the $x$-axis after $v_u$, the endpoint of the entire walk moves one step up.  Hence $w$ is not ballot, a contradiction.
  
  It follows that the lower $2$-break point, $v_{\ell}$, occurs weakly after $v_u$ in standardization order.  For the second part of the statement, let $v_j$ occur between $v_u$ and $v_{\ell}$ in standardization order.  Then by Lemmas \ref{lem:break-point} and \ref{lem:upper-break-point}, $v_j$ is both $1$-removable and $2$-removable, so it is removable.  No other $2$ or $2'$ is removable except possibly for the first $2'$ in the final $E_2$-critical substring in the case that the final $F_2$-critical substring in $w$ was type 1F, as in Lemma \ref{lem:upper-break-point}.  Since this extra entry is not adjacent to the others in standardization order (or it is itself already the upper $1$-break point), the result now follows.
\end{proof}

While Lemma \ref{lem:removable-pairs} is essential for understanding the intermediate steps after applying the coplactic operators $F_i$ in Phase 2 (as in Theorem \ref{thm:crystal-algorithm}), we also need to understand the removable pairs before the first application of an $F_i$ operator, that is, at the index in which Phase 2 begins.  We therefore require the following additional lemma.

\begin{lemma}[Initial Removable Pairs]\label{lem:removable-pairs-special}
  Let $T$ be a tableau appearing just after Phase 1 ends in the computation of $\esh(\ybox,T_0)$ for some $T_0$.  Let $T'$ be the tableau formed by replacing $\ybox$ by $s'$ in $T$ where $s$ is the first Phase 2 index.  Then $\ybox=s'$ is the first $s'$ in standardization order.  Moreover, if $v$ and $w$ are respectively the reading words of $T'$ and $E_s(T')$, then in $T'$, a pair of entries $v_j,v_k$ equal to $s$ or $s'$ that are consecutive in standardization order are both removable if and only if $$v_j\leqst v_k\leq v_{\ell}$$ where $\ell$ is the index for which $w_\ell$ is the lower $s$-break point of $T'$.  
\end{lemma}

\begin{proof}
  Note that the desired entries are removable for the $s,s+1$-subword, and are the only pairs of adjacent entries in standardization that are both removable, by the definition of the lower $s$-break point.  Thus it suffices to show that these entries are removable for the the $s-1,s$-subword as well.

  By Theorem \ref{thm:phase1-ballotness}, the entry $\ybox=s'$ is removable in $T'$.  Switching the $\ybox$ with the next $s'$ (which may be the first $s$) in standardization order, which is equivalent to removing the next $s'$ instead, is also ballot by the argument in Step (b) of Theorem \ref{thm:phase1-ballotness}.  
  
  Finally, we wish to show that if $\ybox=s$ is removable and we switch it with the next $s$ in standardization order, the resulting tableau is still ballot for the $s-1,s$-subword.  Indeed, this is equivalent to moving an $s$ earlier in the word past some $s-1$ or $(s-1)'$ entries, which can only increase the differences $m_{s-1}(j)-m_s(j)$ for $j\le n$, so Stembridge's condition (1) (in Lemma \ref{lem:stembridge}) is still satisfied.  It also does not change any of the values involved in condition (2), so the word is still ballot.
\end{proof}

\subsubsection{Phase 2 proofs}

We now prove the validity of Phase 2 of the switching algorithm.  We first require a lemma about ballotness at intermediate steps of the coplactic algorithm.

\begin{lemma}\label{lem:j-ballot}
  Let $(\ybox,T)\in \LRyb$.  In the computation of $\esh(\ybox,T)$, let $S$ be the tableau formed by replacing $\ybox$ by $s'$ after Phase 1 ends, where $s-1$ is the last Phase 1 index.  For $s-1\le i\le \ell(\beta)$, define $S_i=F_i\circ \cdots \circ F_{s+1}\circ F_s(S).$  Then $S_i$ is $j$-ballot for all $j\neq i$.
\end{lemma}

\begin{proof}
  Since the condition of being $j$-ballot is equivalent to being killed by the operators $E_j$ and $E_j'$, it is a coplactic condition.  It therefore suffices to check that this property is satisfied in the case that $\alpha$ is the empty partition.  This is straightforward to verify in this case.
\end{proof}

\begin{thm}\label{thm:phase2-algorithm}
  Let $(\ybox,T)\in \LRyb$.  Then Phase 2 of the algorithm in Theorem \ref{thm:crystal-algorithm} to compute $\mathrm{esh}(\ybox,T)$ agrees with Phase 2 of Theorem \ref{thm:step-by-step-algorithm}.
\end{thm}

\begin{proof}
  Define $S$ and $S_i$ as in Lemma \ref{lem:j-ballot}.  Notice that Lemma \ref{lem:j-ballot} shows that setting $w$ to be the reading word of $S_{i-1}$ and $v$ the reading word of $S_i=F_i(S_{i-1})$ satisfies the conditions of the Removable Pairs lemma (Lemma \ref{lem:removable-pairs}).  We use this fact implicitly throughout, freely applying Lemma \ref{lem:removable-pairs} without reference to Lemma \ref{lem:j-ballot}.
  
  Define $T_i$ to be the tableau formed by replacing the $\ybox$ with $(i+1)'$ (resp.\ $i+1$) after the index $i$ steps (resp.\ the index $i+1'$ steps) of the algorithm in Phase 2(a) (resp.\ Phase 2(b)).  We wish to show that $S_i=T_i$ for all $i$.  We proceed by induction on $i$, by simultaneously showing the following three statements:
  \begin{enumerate}
      \item $S_i=T_i$ for all $i\ge s-1$.
      \item Let $t_i$ be the $i+1$ or $i+1'$ entry in $T_i$ that replaced $\ybox$.  Then if $i\ge s$, the entry $t_i$ is the $F_i$-transformation of the last (resp.\ first) entry of the $F_i$-critical substring in $S_{i-1}$ if it is type 2F or 4F (resp.\ 1F or 3F).
      \item The entry $t_i$ is removable in $T_i$.
  \end{enumerate}

  We first show the base cases, $i=s-1$ and $i=s$.  The case $i=s-1$ simply refers to the tableau $S$ with no $F$ operators applied, so the first two claims hold trivially, and the third by Lemma \ref{lem:removable-pairs-special} (Initial Removable Pairs).  
  
  Now consider the case $i=s$.  Let $w_a\cdots w_b$ be the final $F_s$-critical substring of the reading word $w$ of $S_{i-1}$.  If it is type 4F, $w_a=w_b$ is the last removable $s'$ in standardization order in $S$ and changes to $(s+1)'$, and by the definition of the switching algorithm this entry is also $t_s$.  Thus the first two claims hold, and the third holds by the definition of a valid switch.   
  
  If $w_a\cdots w_b=s(s+1)^\ast s'$ is type 2F, first suppose the lower $s$-break point is $w_b$ (as opposed to the special case in which it can be $w_a$).  Then $F_s$ changes the substring to $(s+1)'(s+1)^\ast s$ to form $S_s$.  To form $T_s$, by Removable Pairs (Lemma \ref{lem:removable-pairs}) and the definition of the switching algorithm, the $\ybox$ inverse hops past each $s'$ until switching with $w_b$, forming $s(s+1)^\ast \ybox$, and then switches with $w_a=s$ via the special Phase 2(a)-hop step to form the substring $\ybox (s+1)^\ast s$.  Replacing $\ybox$ with $(s+1)'$ forms $T_s$ and therefore matches with $S_s$.  The first two claims therefore hold, and for the third, note that $\ybox$ was removable just before the Phase 2(a)-hop step.  The effect of the hop step on the reading word, ignoring $\ybox$, is to move an $s$ past some number of $s+1$ entries in reading order.  By the Stembridge ballotness condition, the word is still $s$-ballot, and since the hop did not affect the $s+1,s+2$-subword, the new position of $\ybox$ is $s+1$-removable as well.  Finally, since $T_s=S_s$, we must have $s-1$-ballotness at this step. It follows that $t_s$ is removable.
  
  If instead $w_a$ is the lower break point, then $w_a=s$ and $w_b=s'$ are consecutive in standardization order, and so these must be the first $s$ and $s'$ entries in reading order.  It follows that, switching $\ybox$ with each $s'$ and then the first $s$, as dictated by Phase 2(a), followed by interpreting $\ybox$ as $(s+1)'$, is equivalent to applying $F_s$ as well, and the argument follows as above.
  
  Finally, if $w_a\cdots w_b=s(s')^\ast (s+1)'$ is type 1F or 3F, then by Removable Pairs, the $\ybox$ completes Phase 2(a) and enters Phase 2(b), inverse hopping past each $s$ until switching with $w_a$.  In the 3F case we are done by a similar argument to 4F.  In the 1F case, $\ybox$ then further hops past $w_b=(s+1)'$, and since $F_s$ transforms the string to $(s+1)'(s')^\ast (s+1)$, interpreting $\ybox$ as $s+1$ yields the same tableau and we have $T_s=S_s$ as desired.  The entry $t_s$ is as described in the second claim, and removability follows from a similar argument as in the type 2F case.
  
  For the induction step, let $i\ge s$ and assume the claims hold for $S_{i-1}:=F_{i-1}\circ \cdots \circ F_s(S)$.  Then since the position of $\ybox$ that forms $T_{i-1}$ is removable, by the Removable Pairs lemma, the $F_i$-critical substring occurs later in standardization order.  An identical argument using the same cases as above now shows that the claims hold for $S_i$, and the induction is complete.
\end{proof}

Note that, in the above proof, the switch from Phase 2(a) to Phase 2(b) coincides with the first 1F or 3F critical substring.  In particular, it follows that the coplactic algorithm (Theorem \ref{thm:crystal-algorithm}) has a natural Phase 2(a)/2(b) dichotomy in Phase 2.  In Phase 2(a) the computations of $F_i$ involve only type 2F or 4F critical substrings, and in Phase 2(b) they involve only 1F or 3F critical substrings.  We state this precisely below.

\begin{corollary} \label{prop:2a2b-coplactic}
Let $S$ be a tableau appearing just after Phase 1 of the coplactic algorithm for computing $\esh(\ybox, T_0)$ for some $T_0$, and let $s$ be the first Phase 2 index.  Then there exists $k$ so that the type of the final $F_i$-critical substring of $$F_{i-1}\circ \cdots \circ F_{s+1}\circ F_s (S)$$ is 2F or 4F iff $i < k$, and 1F or 3F iff $i \ge k$.
\end{corollary}

  We now prove Theorem \ref{thm:phase12-ballotness} and \ref{thm:pause-esh}(3) for Phase 2.  

\begin{thm}\label{thm:phase2-ballotness}
Properties (1)--(3) from Theorems \ref{thm:phase12-ballotness} and \ref{thm:pause-esh} hold after each  Phase 2 switch.
\end{thm}

\begin{proof} 
   For (1), by Lemma \ref{lem:removable-pairs}, any Phase 2 step in the $i$-ribbon where $i>s$ (where $s$ is the transition point from Phase 1 to Phase 2) has the property that the $\ybox$ is removable, so the reading word is ballot.  By Lemma \ref{lem:removable-pairs-special}, the same is true for the moves in Phase 2 across the $s$-ribbon.  Thus (1) holds.
   
   We now show (2).  If the $\ybox$ has just completed the moves past $i'$ or $i$, and then the possible hop, in Phase 2(a) or 2(b), then replacing $\ybox$ by $(i+1)'$ or $(i+1)$ respectively gives the result of applying $F_i$ to a tableau.  This is semistandard, so removing the $\ybox$ leaves a semistandard tableau as well, hence (2) is satisfied.  Then, any subsequent switch with an $i+1'$ or $i+1$ in Phase 2 can also be thought of as the same tableau where we replace the $\ybox$ by $i+1$ or $(i+1)'$ as appropriate.  Finally, for the Phase 2 steps just after Phase 1 ends, note that if we replace $\ybox$ by $s'$ after the last Phase 1 move, then it is a semistandard tableau that we may apply $F_s$ to, and the same argument holds.
 
 We now show (3).  Suppose we are in Phase 2(a), the $\ybox$ has completed all inverse hop steps of index $i'$, and next step is not exceptional (or we are in Phase 2(b) and the $\ybox$ has completed all inverse hop steps of index $i$).  Then replacing $\ybox$ with $i'$ (respectively $i$) gives the tableau before the application of $F_i$, which is semistandard.  Moreover, if there were an $i'$ directly below (resp.\ $i$ directly to the right of) the $\ybox$, then switching $\ybox$ with this entry would not change the column (resp.\ row) reading word, and hence not affect ballotness, contradicting the definition of the Phase 2(a) and Phase 2(b) stopping points. Hence $\ybox$ is an outer co-corner of the tableau formed by the letters less than or equal to $i'$ (resp.\ $i$) and an inner co-corner of the remaining letters.  A similar argument applies if the next step is exceptional.
 
 Finally, we consider the tableaux after the special Phase 2 hop steps.  Suppose the $\ybox$ has just completed a Phase 2(a)-hop step of index $i$.  If this step was null, that is, if the box does not switch with any $i$, there cannot be an $i$ directly left of or below $\ybox$ (otherwise switching would not change the change the row or column reading word and the hop would occur), so (3) holds.   So assume a switch occured, from a substring $i(i+1)^\ast\ybox$ to $\ybox (i+1)^\ast i$.  Then the application of $F_i$ was type 2F where we interpret $\ybox$ as $i'$ in both settings.  Moreover, prior to the hop, the $\ybox$ may be replaced with $i'$ to obtain the (semistandard) tableau before the $F_i$ application; thus there is no $i$ directly to the right of the first $i$ in this string.  It follows that  the new position of $\ybox$ is an inner co-corner of the letters larger than $i$ and an outer co-corner of the letters less than or equal to $i$.  If instead we have just completed a Phase 2(b)-hop step of index $i+1'$, then the application of $F_i$ was type 1F, and a similar analysis shows that the new position of $\ybox$ satisfies (3).
\end{proof}

\subsubsection{Proof of Theorem \ref{thm:resume-esh}: Pausing and resuming}
\label{sec:pause-and-resume}

We now prove Theorem \ref{thm:resume-esh}.  To do so, we first show that the hop steps in Phase 2 can alternatively be computed by checking for valid switches rather than the existence of particular substrings.

\begin{proposition}\label{prop:teleport-valid}
For any index $i$, the Phase 2(a) hop of index $i$ occurs in the switching algorithm if and only if this switch is valid.  Similarly, the Phase 2(b) hop of index $i+1'$ occurs if and only if this switch is valid.

\end{proposition}

\begin{proof}
 \textbf{Phase 2(a):} Let $T$ be a tableau that arises in Phase 2(a) just after the inverse hops past $i'$, for which the next step is still in Phase 2(a).  If the hop of $\ybox$ past an $i$ then occurs, this switch is valid by Theorem \ref{thm:phase2-ballotness}.  
 
 Conversely, suppose the hop switch of $\ybox$ with the previous $i$ in reading order (call this entry $t$) is valid.  We will show that there are no $i'$, $i$, or $i+1'$ entries between $t$ and $\ybox$ in $T$, so that the substring of the $i,i+1$-reading word from $t$ to $\ybox$ has the form $i(i+1)^\ast\ybox$, and therefore a hop does occur.  By the definition of $t$, there is no $i$ entry between $t$ and $\ybox$.  
 
 Assume for contradiction that there is an $i'$ between $t$ and $\ybox$, and let $s=i'$ be the rightmost $i'$ preceding $\ybox$ (where $\first(i)$ is treated as $i'$).  Then $s$ is between $t$ and $\ybox$, and switching $\ybox$ with $s$ violates ballotness since $T$ is at the end of a Phase 2(a) move.  Thus one of the two Stembridge ballotness conditions (S1 or S2) stated in Lemma \ref{lem:stembridge} is violated by this switch.  Note that the effect of the switch on the $i,i+1$ reading word is to move $s=i'$ later in the word past some number of $i+1$ or $i+1'$ entries, which cannot violate condition S1 since the total numbers of $i$ and $i+1$ entries after each $i+1'$ or $i+1$ entry is unchanged.  Thus it violates S2, and in particular there must be an $i+1'$ entry between $s$ and $\ybox$, so that when we switch $s$ with $\ybox$ we would have $m_{i}(j)<m_{i+1}(j)$ where $j=n+k$ is the index for which there are $k$ letters before $\ybox$ in the reading word of $T$.  But then we must have $m_i(j)=m_{i+1}(j)$ in $T$, so switching $\ybox$ with $t=i$ would also violate condition S2, a contradiction.  It follows that there is no $i'$ between $t$ and $\ybox$.  
 
 If $s$ does not exist, then in fact the next step of $T$ would be an exceptional Phase 2(a) step rather than a hop, a contradiction.  Thus $s=i'$ is strictly left of $t$ in reading order.  Now, if there were an $(i+1)'$ between $t$ and $\ybox$, switching $\ybox$ with $s$ makes the word not ballot and so by an argument similar to that above, switching $\ybox$ with $t'$ would also make the word not ballot by violating condition 2 at the new position of $t'$.  Thus there is no $(i+1)'$ between $t'$ and $\ybox$, as desired.  
 
 \textbf{Phase 2(b):} Let $T$ be a tableau that arises in Phase 2(b) just after the inverse hops past $i$.  If the hop of $\ybox$ past the next $(i+1)'$ occurs, this switch is valid by Theorem \ref{thm:phase2-ballotness}.  
 
 Conversely, suppose the hop past the next $(i+1)'$ (call this entry $t$) is valid.   We wish to show that there is no $i$, $i+1'$, or $i+1$ between $\ybox$ and $t$ in reading order, so that the substring of the reading word is of the form $\ybox (i')^\ast (i+1')$.  There is no $i+1'$ between them by the definition of $t$.  
 
 Assume for contradiction that there is an $i$ between $\ybox$ and $t$, and let $s=i$ be the leftmost $i$ to the right of $\ybox$.  We know that switching $\ybox$ with $s$ yields a tableau which is not ballot for $i,i+1$ since $T$ is at the end of a Phase 2(b) move.  This can violate condition S1 of Lemma \ref{lem:stembridge}, if there is an $i+1$ between $\ybox$ and $s$ and the suffix after the $\ybox$ has the same number of $i+1$ letters as $i$, so that the switch would make $m_i(j)>m_{i+1}(j)$ for some $j$.  It can alternatively violate condition S1 if the new position of $s$, say $k$, satisfies $m_i(j+k)=m_{i+1}(j+k)$ after the switch.  In either case, switching $\ybox$ with $t$ would make $t$ violate the same condition (S1 or S2 respectively).  Thus there is no $i$ between $\ybox$ and $t$.  
 
 If there is no $i$ after $\ybox$ at all then there is no $i+1$ after it either by ballotness, so we may assume that $s=i$ exists and occurs after $t$ in reading order.  If switching $\ybox$ with $s$ violates condition S2 at the index of $s$, it violates it at $t$ too, so in fact switching with $s$ violates condition S1 instead.  If it violates S1 by making some suffix have more $i+1$ entries than $i$ entries, then switching $t$ with $\ybox$ would violate condition 1 at the new position of $t$, a contradiction.  Thus it violates S1 by having the suffix after either $t$ or after a later $(i+1)'$ between $t$ and $s$ have the same number of $i$ and $i+1$ entries.   But then switching $\ybox$ with $s$ makes the suffix after $t$ have the same number of $i$ entries as $i+1$ entries, and since switching with $t$ is ballot, there cannot be an $i+1$ between $\ybox$ and $t$, as desired.
\end{proof}

\begin{proof}[Proof of Theorem~\ref{thm:resume-esh}]
Let $\partesh^*_t$ denote the sequence of switches defined by the statement of
the Theorem.  We must show that $\partesh^*_t$ agrees with $\partesh_t$.

First, suppose $t=i'$ for some $i$.
Suppose $T$ arises partway through the switching algorithm, after completing
all steps of index $i-1$.
Then the steps performed by $\partesh_{i'}$, could be either the $i'$-step of Phase 1, the first part of Phase 2(a) (excluding any exceptional steps at the end, 
and the hop), or a Phase 2(b)-hop.
  
Since $t=i'$ there cannot be an exceptional step of index $i$, and we 
cannot be in Case 1 of $\partesh^*_t$.

If the next step is in Phase 1, then $\partesh_t$ performs a hop of index $i'$.
Since Case 2 of $\partesh^*_t$ is to do this if possible,
$\partesh^*_{i'}$ agrees with $\partesh_{i'}$.
  
  If the next step is Phase 2(a), note that by the Removable Pairs Lemma (\ref{lem:removable-pairs}) we cannot switch $\ybox$ with the next $i'$ in reading order, i.e. there is no hop of index $i$.  Thus $\partesh^*_{i'}$ is in Case 3, which performs all steps of index $i'$ in Phase 2(a). 
  
  Finally, if the next step is Phase 2(b)-hop, if the hop switch occurs then $\partesh^*_t$ makes the same switch and gives the correct output.  Otherwise, by Proposition \ref{prop:teleport-valid}, there is no valid hop of index $i'$, so $\partesh^*_{i'}$ is in Case 3.  We now only need to check that if the hop does not occur, $\partesh^*_{i'}$ does not perform any inverse hops.  By the nature of a null hop step in Phase 2(b), if we replace $\ybox$ with $i-1$ in $T$ then it is a type 3F critical substring for the operator $F_{i-1}$.  The $i-1,i$-walk is therefore on the $x$-axis just before $\ybox$, and replacing $\ybox$ with $i$ gives a non-ballot total walk since it is the result of $F_{i-1}$ applied to the tableau.  Therefore replacing $\ybox$ with $i'$ also gives a non-ballot total walk, since $i'$ and $i$ are both upwards arrows when starting on the $x$-axis.  By the Upper Break Point Lemma (\ref{lem:upper-break-point}), the previous $i'$ before $\ybox$ is not removable in either case, and therefore $\ybox$ cannot switch with the previous $i'$.
  
  Now suppose $t=i$.  In this case, the steps performed by $\partesh_i$ are could be either the $i$-step of Phase 1, the first part of Phase 2(b) (excluding the hop), an exceptional Phase 2(a) step followed by the first part of Phase 2(b), or a Phase 2(a)-hop step. 
  
If the next step is in Phase 1, then the exceptional step does not occur because the previous step was to switch $\ybox$ with a later $i'$, so $\partesh^*_i$ is not in Case 1.  Therefore, $\partesh^*_i$ is in Case 2, and agrees with $\partesh_i$.
  
  If the next step is Phase 2(b), $\ybox$ cannot make the exceptional move nor a hop of index $i$ by the Removable Pairs Lemma (Lemma \ref{lem:removable-pairs}) and so $\partesh^*_i$ is in Case 3, and agrees with the result of the index $i$ steps in Phase 2(b).  
Similarly, if the next step is an exceptional Phase 2(a) step, $\partesh^*_i$ is in Case 1, and again agrees with $\partesh_i$.
  
  Finally, suppose the next step is Phase 2(a)-hop.   If the hop occurs, this switch is valid and the algorithm agrees with $\partesh^*_i$ (Case 2).  If it does not occur, then by Proposition \ref{prop:teleport-valid}, switching with the previous $i$ is not valid, so $\partesh^*_i$ is in Case 3.  We now only need to check that if the hop does not occur, $\partesh^*_i$ does not perform any inverse hops.  Since $T$ is the result of the last of the $i'$ Phase 2(a) switches by assumption, if we replace $\ybox$ by $i'$ in $T$, we get a tableau with ballot $i,i+1$-subword whose final $F_i$-critical substring is type 4F at $\ybox=i'$.  Then by the Lower Break Point Lemma (\ref{lem:break-point}), the next $i$ after $\ybox=i'$ is not removable.  Since the walk of the reading word is on the $y$-axis at the critical substring, we may replace $\ybox$ with $i$ instead and obtain the same walk, and so the next $i$ after $\ybox$ is still not removable in this Case.  Hence we cannot have switched $\ybox$ with the next $i$ to obtain a ballot word, and the conclusion follows.  
\end{proof}

%%%%%%%%%%%%%%%%
%%% K-THEORY %%%
%%%%%%%%%%%%%%%%

\section{Connection to K-theory of the orthogonal Grassmannian}\label{sec:K-theory}

The structure sheaves $\mathcal{O}_\lambda$ of Schubert varieties in $\OGn$ form an additive basis for the K-theory ring $\K(\OGn)$, and they have a product formula
\[[\mathcal{O}_\mu] \cdot [\mathcal{O}_\nu] = \sum_{|\lambda| \geq |\mu| + |\nu|} (-1)^{|\lambda| - |\mu| - |\nu|}k_{\mu \nu}^\lambda [\mathcal{O}_\lambda],\]
for certain nonnegative integer coefficients $k_{\mu \nu}^\lambda$. These coefficients enumerate certain shifted tableaux called \emph{(shifted) ballot genomic tableaux\/,} defined by Pechenik and Yong in \cite{bib:Pechenik}.  We recall this definition here.

\begin{definition}[\cite{bib:Pechenik}]
For a skew shifted tableau with entries $i_j$ or $i'_j$ for $i\in \{1,2,\ldots\}$ where $j$ is any symbol, we call $i$ the \emph{gene family} of an entry $i_j$ or $i'_j$ and $j$ the \emph{gene}. Two such tableaux are \textit{equivalent} if they are the same up to a relabeling of the genes $j$.  An equivalence class $T$ of these tableaux is a \newword{semistandard genomic tableau} if:
\begin{itemize}
\item The tableau $T_{ss}$ obtained by forgetting the gene subscripts is semistandard and in canonical form.
\item Each gene $j$ consists only of letters $i_j$ and $i'_j$ for some $i$, and its letters are consecutive in standardization order.
\item For every primed entry $i'_j$, there is an unprimed letter $i_k$ preceding it in reading order from the same family $i$ but different gene $k\neq j$.
\item No two squares of the same gene are horizontally or vertically adjacent.
\end{itemize}
The \newword{K-theoretic content} of $T$ is $(c_1, \ldots, c_r)$ where $c_i$ is the number of genes in the $i$-th family. Finally, $T$ is \newword{ballot} if it is semistandard and has the following property:
\begin{itemize}
\item[$(*)$] Let $T'$ be any genomic tableau obtained by deleting, within each gene family of $T$, all but one of every gene. Let $T'_{ss}$ be the tableau obtained by deleting the corresponding entries of $T_{ss}$. Then the reading word of $T'_{ss}$ is ballot for all $T'$ obtained in this way.
\end{itemize}
\end{definition}
\noindent Write $\K(\lambda/\mu;\nu)$ for the set of ballot genomic tableaux of shape $\lambda/\mu$ and K-theoretic content $\nu$.
\begin{thm}[\cite{bib:Pechenik}]
We have $k_{\mu  \nu}^\lambda = |\K(\lambda/\mu;\nu)|$.
\end{thm}

We are most concerned with the case of partitions $\alpha, \beta, \gamma$ with $|\alpha| + |\beta| + |\gamma| = k(n-k)-1$. In this case there will only be one repeated gene, in one gene family. Let $\K(\gamma^c/\alpha;\beta)(i)$ be the set of increasing tableaux in which $i$ is the repeated gene family. For convenience, we state the following simpler characterization of this set:
\begin{lemma} \label{lem:genomic-criterion}
Let $T$ be a shifted semistandard tableau in canonical form of shape $\gamma^c/\alpha$ and content equal to $\beta$ plus a single extra $i$ or $i'$. Let $\{\ybox_1, \ybox_2\}$ be two squares of $T$, such that:
\begin{itemize}
\item[(i)] The squares $\ybox_1,\ybox_2$ are non-adjacent, consecutive in standardization order, and each contains $i$ or $i'$.
\item[(ii)] If $\ybox_1$ contains $i'$ and $\ybox_2$ contains $i$ (so that $\ybox_2$ is the first $i$ in reading order) then there is an $i$ between $\ybox_2$ and $\ybox_1$ in reading order.
\item[(iii)] For $k = 1,2,$ the word obtained by deleting $\ybox_k$ from the reading word of $T$ and canonicalizing is ballot.
\end{itemize}
There is a unique ballot genomic tableau $T' \in \K(\gamma^c/\alpha; \beta)(i)$ corresponding to the data $(T,\{\ybox_1,\ybox_2\})$. Conversely, each $T'$ corresponds to a unique $(T,\{\ybox_1,\ybox_2\})$.
\end{lemma}
\begin{proof}
The gene families of $T'$ are the entries of $T$. For $m \ne i$, the $m$-th gene family of $T'$ has all distinct genes, one for each entry $m$ in $T$. For the $i$-th gene family, the the squares $\ybox_1, \ybox_2$ are in the same gene $j$ and the others are in distinct genes.  Conditions (i) and (ii) are equivalent to $T'$ being genomic; ballotness of $T'$ is then equivalent to (iii).  
\end{proof}

\subsection{Generating genomic tableaux}
We now establish connections between local evacuation-shuffling and genomic tableaux. Consider a step of index $t$ in the switching algorithm for $\esh(\ybox, T)$. Call the position of $\ybox$ before and after the switch $\ybox_1$ and $\ybox_2$.  By Theorem \ref{thm:phase12-ballotness}, the tableau's reading word is ballot both before and after the switch. If we replace both $\ybox_1$ and $\ybox_2$ by $t$'s, then they are consecutive in standardization order by the nature of the algorithm. Putting the result in canonical form, we obtain a ballot genomic tableau, with $\{\ybox_1, \ybox_2\}$ as the unique repeated gene, as long as the two squares are non-adjacent. 
If the step is a hop (including the possibility of an exceptional step) we will say that $\ybox$ \newword{traversed} this genomic tableau \newword{in reverse standardization order}; otherwise, for an inverse hop, we will say that $\ybox$ traversed the tableau \newword{in standardization order}.

\begin{thm}\label{thm:k-theory-correspondence}
  The assignment above gives a two-to-one correspondence between non-adjacent steps of the switching algorithm (starting from all possible tableau in $\LRyb$) and the set $\K(\gamma^c/\alpha;\beta)$. Moreover, each genomic tableau is traversed once in standardization order and once in reverse standardization order.
\end{thm}

\begin{example}
  Suppose  $\alpha=(4,2)$, $\gamma^c=(5,3,1)$, and $\beta=(2)$.  The six genomic tableaux in $\K(\gamma^c/\alpha,\beta)$ are as follows, where the repeated gene is underlined:
  \begin{center}
  (a) {\small \begin{ytableau} 
    \none & \none & \underline{1} \\
    \none & \underline{1} \\
    1
  \end{ytableau}
  } \hspace{1cm}
  (b) {\small \begin{ytableau}
    \none & \none & 1 \\
    \none & \underline{1} \\
    \underline{1}
  \end{ytableau}
  } \hspace{1cm}
  (c) {\small \begin{ytableau}
    \none & \none & 1' \\
    \none & \underline{1} \\
    \underline{1}
  \end{ytableau}
  }  \\ \vspace{0.5cm}
  (d) {\small \begin{ytableau}
    \none & \none & \underline{1'} \\
    \none & 1 \\
    \underline{1}
  \end{ytableau}
  } \hspace{1cm}
  (e) {\small \begin{ytableau}
    \none & \none & \underline{1} \\
    \none & 1' \\
    \underline{1}
  \end{ytableau}
  } \hspace{1cm}
  (f) {\small \begin{ytableau}
    \none & \none & \underline{1'} \\
    \none & \underline{1'} \\
    1
  \end{ytableau}
  }
  \end{center}
  
There are six tableaux in $\LRyb$.  Applying the switching algorithm beginning
with each produces the following steps:

\begin{center}
{\setlength\tabcolsep{.2in}
\begin{tabular}{cp{2.4in}}
\small 
\begin{ytableau}
  \none & \none & 1 \\
  \none & 1 \\
  \times 
\end{ytableau}
\raisebox{-0.5cm}{$\xrightarrow{\text{Phase 2(b)}}$}
\begin{ytableau}
  \none & \none & 1 \\
  \none & \times \\
  1 
\end{ytableau}
\raisebox{-0.5cm}{$\xrightarrow{\text{Phase 2(b)}}$}
\begin{ytableau}
  \none & \none & \times \\
  \none & 1 \\
  1 
\end{ytableau}
& traversing (b) and (a) in standardization order;
\\[5ex]
\small 
\begin{ytableau}
  \none & \none & 1' \\
  \none & 1 \\
  \times 
\end{ytableau}
\raisebox{-0.5cm}{$\xrightarrow{\text{Phase 1}}$}
\begin{ytableau}
  \none & \none & \times \\
  \none & 1 \\
  1 
\end{ytableau}
\raisebox{-0.5cm}{$\xrightarrow{\text{Phase 1}}$}
\begin{ytableau}
  \none & \none & 1 \\
  \none & \times \\
  1 
\end{ytableau}
& traversing (d) and (a) in reverse standardization order;
\\[5ex]
\small 
\begin{ytableau}
  \none & \none & 1 \\
  \none & \times \\
  1'
\end{ytableau}
\raisebox{-0.5cm}{$\xrightarrow{\substack{\text{Exceptional} \\ \text{Phase 2(a)}}}$}
\begin{ytableau}
  \none & \none & 1 \\
  \none & 1' \\
  \times 
\end{ytableau}
\raisebox{-0.5cm}{$\xrightarrow{\text{Phase 2(b)}}$}
\begin{ytableau}
  \none & \none & \times \\
  \none & 1' \\
  1 
\end{ytableau}
& traversing (b) in reverse standardization order and (e) in standardization order;
\\[5ex]
\small 
\begin{ytableau}
  \none & \none & 1' \\
  \none & \times \\
  1 
\end{ytableau}
\raisebox{-0.5cm}{$\xrightarrow{\text{Phase 1}}$}
\begin{ytableau}
  \none & \none & \times \\
  \none & 1' \\
  1 
\end{ytableau}
\raisebox{-0.5cm}{$\xrightarrow{\text{Phase 1}}$}
\begin{ytableau}
  \none & \none & 1 \\
  \none & 1 \\
  \times 
\end{ytableau}
& traversing (e) and (f) in reverse standardization order;
\\[5ex]
\small 
\begin{ytableau}
  \none & \none & \times \\
  \none & 1 \\
  1'
\end{ytableau}
\raisebox{-0.5cm}{$\xrightarrow{\text{Phase 2(a)}}$}
\begin{ytableau}
  \none & \none & 1' \\
  \none & 1 \\
  \times 
\end{ytableau}
\raisebox{-0.5cm}{$\xrightarrow{\text{Phase 2(b)}}$}
\begin{ytableau}
  \none & \none &  1'\\
  \none & \times \\
  1 
\end{ytableau}
& traversing (d) and (c) in standardization order;
\\[5ex]
\small 
\begin{ytableau}
  \none & \none & \times \\
  \none & 1' \\
  1'
\end{ytableau}
\raisebox{-0.5cm}{$\xrightarrow{\text{Phase 2(a)}}$}
\begin{ytableau}
  \none & \none & 1' \\
  \none & \times \\
  1'
\end{ytableau}
\raisebox{-0.5cm}{$\xrightarrow{\substack{\text{Exceptional} \\ \text{Phase 2(a)}}}$}
\begin{ytableau}
  \none & \none & 1' \\
  \none & 1' \\
  \times 
\end{ytableau}
& traversing (f) in standardization order and (c) in 
reverse standardization order.
\end{tabular}
}
\end{center}
Note that since the exceptional steps have index $1$ instead of $1'$, they traverse genomic tableaux in which the repeated gene has two $1$s.  Also note that (d), which has both a $1$ and $1'$ in the repeated gene, is always traversed by a step of index $1'$.  Here, $\ybox_1$ and $\ybox_2$ are first both treated as $1'$, and we obtain (d) when we put the result in canonical form.
\end{example}

\begin{proof}
Let $(T, \{\ybox_1, \ybox_2\})$ be a genomic tableau 
in $\K(\gamma^c/\alpha;\beta)$.
First, suppose that both entries of the repeated gene $\{\ybox_1, \ybox_2\}$ 
are equal to $i$, with $\ybox_1$ preceding
$\ybox_2$ in standardization order.  For $j=1,2$, let $T_j$ be the tableau 
obtained by replacing $\ybox_j$ by $\ybox$ in $T_{ss}$.  Both $T_1$ and
$T_2$ are semistandard and ballot, omitting $\ybox$.

It is clear that $T_2 \in Z_{i'}(\alpha,\beta,\gamma)$, and 
 $T_2 \leadsto T_1$ is a hop.  Therefore, by Theorem~\ref{thm:resume-esh},
the first step of $\partesh_i(T_2)$ is either the hop $T_2 \leadsto T_1$
or an exceptional step
(which occurs in the case where $\ybox_1$ and $\ybox_2$ are the first 
two letters in the $i,i'$-reading word).  In either case this first
step of $\partesh_i(T_2)$ traverses
$(T, \{\ybox_1, \ybox_2\})$ in reverse standardization order, and $T_2$ is the
only tableau in $Z_{i'}$ with this property.  
By Theorem~\ref{thm:pause-esh}, this implies that there 
is a unique tableau in $\LRyb$
for which the switching algorithm traverses $(T, \{\ybox_1, \ybox_2\})$
in standardization order.

Let $T'_1$ be the tableau obtained starting from $T_1$ and performing
as many inverse hops of index $i$ as possible.  Note that this sequence
of inverse hops must have length at least one, since $T_1 \leadsto T_2$ is an
inverse hop.  We have $T'_1 \in Z_i(\alpha,\beta,\gamma)$, since otherwise 
there is an $i$ immediately right of $\ybox$ in $T'_1$; but then these are
adjacent in the reading word, so it possible to switch them, 
contradicting the maximality of the sequence of inverse hops.  
Since it is not possible
to perform an inverse hop or an inverse exceptional switch starting at $T'_1$, Theorem~\ref{thm:reverse-resume-esh}, tells us that the $\partesh_i^{-1}(T'_1)$ begins in Case C, by reversing this sequence of inverse hops. Therefore, some step of $\partesh_i^{-1}(T'_1)$ 
traverses $(T, \{\ybox_1, \ybox_2\})$ in standardization onder, and 
$T'_2$ is the only tableau in $Z_i$ with this property.  
By Theorem~\ref{thm:pause-esh}, this implies that there 
is a unique tableau in $\LRyb$
for which the switching algorithm traverses $(T, \{\ybox_1, \ybox_2\})$
in reverse standardization order.

If ther repeated gene has at least one $i'$, then the argument is 
essentially identical, with $\partesh_{i'}: Z_{i-1} \to Z_{i'}$ in place of
$\partesh_i : Z_{i'} \to Z_i$, and using the column reading word $T'_2$ to 
prove that $T'_2 \in Z_{i'}$.
\end{proof}

\subsection{Geometric consequences}

Using the connection with genomic tableaux developed above, we can now deduce several facts about \emph{complex} geometry of Schubert curves.  

\begin{definition}
  Let $\eta(S)$ be the number of real connected components of the Schubert curve $S$ (which is equal to the number of orbits of $\omega$), let $h^0(S)=\dim_{\mathbb{C}}H^0(\mathcal{O}_S)$ be the number of complex connected components, let $\iota(S)$ be the number of irreducible components, and let $\chi(\mathcal{O}_S)$ be its holomorphic Euler characteristic.
\end{definition}

We first prove two basic geometric results relating these quantities.  The type A versions of these facts were shown in \cite{bib:Levinson}, and the proofs are similar, but we include sketches of the proofs here for the reader's convenience.

\begin{lemma}\label{lem:chi-formula}
  We have $$\chi(\mathcal{O}_S)=|\LRyb|-|\K(\gamma^c/\alpha;\beta)|.$$
\end{lemma}

\begin{proof}
  Since $\chi$ is additive on the K-theory ring $\K(\OG)$, we have 
  \begin{eqnarray*}
  \chi(\mathcal{O}_S)&=&\chi([\mathcal{O}_\alpha]\cdot [\mathcal{O}_\beta]\cdot[\mathcal{O}_\gamma]) \\
  &=& \chi\left(-k^{\stair}_{\alpha\beta\gamma} [\mathcal{O}_{\stair}]+k^{\stair/\tinybox}_{\alpha\beta\gamma}[\mathcal{O}_{\stair/\tinybox}]\right) \\
  &=& -k^{\stair}_{\alpha\beta\gamma}+k^{\stair/\tinybox}_{\alpha\beta\gamma}
  \end{eqnarray*}
  since the Schubert curve $X_{\stair/\tinybox}$ is a copy of $\mathbb{P}^1$ and so $\chi(\mathcal{O}_\stair)=\chi(\mathcal{O}_{\stair/\tinybox})=1$.  Furthermore, since $|\alpha|+|\beta|+|\gamma|=|\stair/\tinybox|$ the coefficient $k^{\stair/\tinybox}_{\alpha\beta\gamma}$ is equal to the corresponding ordinary Littlewood-Richardson coefficient and is enumerated by $|\LRyb|$.  Finally, the other coefficient, $k^{\stair}_{\alpha\beta\gamma}$, counts chains of ballot genomic tableaux completely filling $\stair$ of contents $\alpha,\beta,\gamma$, which is equivalently enumerated by $|\K(\gamma^c/\alpha;\beta)|$. 
\end{proof}

\begin{lemma}\label{lem:inequality}
  We have $\chi(\mathcal{O}_S)\le h^0(S)\le \iota(S)\le \eta(S)$.
\end{lemma}

\begin{proof}
  The first inequality follows from the fact that $S$ is one-dimensional so $\chi(\mathcal{O}_S)=h^0(S)-h^1(S)$.  We clearly have $h^0(S)\le \iota(S)$.    Finally, since the map $S\to \mathbb{P}^1$ is flat (Lemma \ref{lem:geometry-schubert-curves-type-B}), every irreducible component of $S$ dominates the $\mathbb{P}^1$.  Since the fibers in $S$ over $\mathbb{RP}^1\subset \mathbb{P}^1$ consist entirely of real points, it follows that every irreducible component of $S$ contains a real point.  Thus we have $\iota(S)\le \eta(S)$.
\end{proof}

\subsubsection{Schubert curves with trivial monodromy}

We now use these connections along with the bijections above to analyze the case when $\omega$ is the identity map.

\begin{thm}\label{thm:identity}
  The monodromy operator $\omega$ acts on $\LRyb$ as the identity if and only if $|\K(\gamma^c/\alpha;\beta)|=0$.  Furthermore, if $\omega$ is the identity map then the complex Schubert curve $S$ is isomorphic to a disjoint union of copies of $\mathbb{P}^1$, and the map $S\to \mathbb{P}^1$ is locally an isomorphism.
\end{thm}

\begin{proof}
  Suppose $\omega$ acts as the identity map.  Let $T\in \LRyb$.  Then $\omega(T)=\sh(\esh(T))=T$, so the switching algorithm applied to $T$ is reversed by JDT.  
  
  Assume for contradiction that there is a non-adjacent step in the switching algorithm computation of $\esh(T)$, that is, a step in which $\ybox$ switches with an entry in a square that is not adjacent to it.  Let $t$ be the first such entry that it switches with; Notice that in any phase of the algorithm, any given $i'$ strip or $i$ strip has all its switches with $\ybox$ in only one direction (reading order or reverse reading order). Thus some entry in the $i'$-strip or $i$-strip containing $t$ is switched with $\ybox$ and not replaced by a letter of the same type.  Since JDT only involves adjacent move, it cannot move this entry back to its original position, and $T$ is not fixed by $\omega$, a contradiction. It follows that there are no non-adjacent steps in the computation and so, by Theorem \ref{thm:k-theory-correspondence}, we have $|\K(\gamma^c/\alpha;\beta)|=0$.   
  
  Conversely, if $|\K(\gamma^c/\alpha;\beta)|=0$, there are no non-adjacent steps in the computation, and so the JDT slides undo every step of the algorithm since the tableau at every step of the algorithm is semistandard.  Hence $\omega=\mathrm{id}$ if and only if $|\K(\gamma^c/\alpha;\beta)|=0$.
  
  Now, suppose $\omega$ is the identity map.  Then we have $\chi(\mathcal{O}_S)=|\LRyb|-0=|\LRyb|$ by Lemma \ref{lem:chi-formula}.  Note that in this case there are also exactly $|\LRyb|$ real connected components of the Schubert curve $S$.   By Lemma \ref{lem:inequality}, it follows that $$\chi(\mathcal{O}_S)=h^0(S)=\iota(S)=\eta(S)$$ and so there are exactly $|\LRyb|$ complex connected components, each of them irreducible. Furthermore, $\dim_{\mathbb{C}}H^1(S)=h^0(S)-\chi(\mathcal{O}_S)=0$ and so $S$ has genus $0$, implying that each component is isomorphic to $\mathbb{P}^1$.
\end{proof}

\subsubsection{Curves with arbitrarily many complex components}

We can now give a family of Schubert curves having arbitrarily many complex connected components.

\begin{proposition}\label{prop:components-example}
Let $t\in \mathbb{Z}_+$ and let $\alpha=(2t+1,2t-1,\ldots,7,5,3)$, $\beta=(t+2)$, and $\gamma^c=(2t+2,2t,\ldots,6,4,2,1)$.  Then the Schubert curve $S$ has exactly $2^{t}$ connected components.
\end{proposition}

\begin{proof}
  Note that any valid semistandard filling of shape $\gamma^c/\alpha$ with content $\beta$ has $\ybox$ as the leftmost entry in the second row.  The entries in the bottom two rows are determined, and every other entry may be either $1$ or $1'$. (See Figure \ref{fig:components}.) Thus $|\LRyb|=2^t$, and it is easy to see that $\omega=\mathrm{id}$.  The conclusion follows from Theorem \ref{thm:identity}.
\end{proof}

\begin{figure}
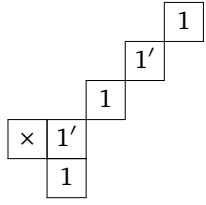

 \begin{center}
$\begin{ytableau}
  \none & \none &  \none & \none & 1 \\
  \none &  \none & \none & 1' \\
  \none &  \none & 1 \\
  \times & 1' \\
  \none  & 1 
\end{ytableau}$
 \end{center}
 \caption{\label{fig:components}A shape and content giving rise to a Schubert curve $S$ with many complex components.}
\end{figure}

\subsubsection{Combinatorial proof of a geometric inequality}

Recall that $\eta(S)$ is the number of real connected components of the Schubert curve $S$; this is equal to $|\mathrm{Orb}(\omega)|$ where $\omega$ is the corresponding monodromy operator and $\mathrm{Orb}$ is the set of orbits of the operator $\omega$. Lemma \ref{lem:chi-formula} allows us to rewrite Lemma \ref{lem:inequality} as the following inequality:
\begin{equation}\label{eqn:inequality}
  |\K(\gamma^c/\alpha;\beta)|\ge |\LRyb|-|\mathrm{Orb}(\omega_{\alpha,\beta,\gamma})|.
\end{equation}

In \cite{bib:GillespieLevinson}, we found a combinatorial proof of this inequality in type A using the bijections between steps of the $\esh$ algorithm and genomic tableaux.  It also relied heavily on a decomposition similar to \eqref{eqn:factorization}.  We now provide a similar proof in the type B setting, omitting some of the details since they closely follow the analysis in \cite{bib:GillespieLevinson}.

\begin{definition}
  Using the notation of \eqref{eqn:factorization}, define 
$\composedpsh_{i'}=\partsh_{1'}\circ \partsh_{1}\circ  \partsh_{2'} \cdots \circ \partsh_{i-1}\circ \partsh_{i'}$ 
and $\composedpsh_{i}= \composedpsh_{i'} \circ \partsh_{i}$ for all $i$.  Then define $\omega_{i'}:Z_0 \to Z_0$ by 
  \[
  \omega_{i'}=\composedpsh_{i-1}\circ (\partsh_{i'} \circ \partesh_{i'}) \circ \composedpsh_{i-1}^{-1}  \hspace{0.5cm} \text{and} \hspace{0.5cm}
  \omega_{i}=\composedpsh_{i'} \circ (\partsh_i \circ \partesh_i) \circ \composedpsh_{i'}^{-1}
  \]
  (where we take $\composedpsh_0$ to be the identity map in the case $i=1$). \end{definition}
  
  Note that we have \begin{equation}\label{eqn:factor-omega}\omega=\omega_{m+1'}\circ \omega_{m}\circ \omega_{m'}\circ \cdots \circ \omega_{1}\circ \omega_{1'}\end{equation} where $m$ is as in \eqref{eqn:factorization}.  We also have that the set of all genomic tableaux arising in the orbits of $\omega$ via the two bijections of Theorem \ref{thm:k-theory-correspondence} coincide with the set of all genomic tableaux arising in each $\omega_{i'}$ and $\omega_{i}$ orbits, at the $\partesh_{i'}$ and $\partesh_{i}$ steps respectively.  
  
  We now analyze the orbits of the $\omega_{i'}$ and $\omega_i$ permutations.  To do so it suffices to understand the corresponding ``loops'' $\partsh_{i'}\circ \partesh_{i'}$ and $\partsh_i \circ \partesh_i$.
  
  \begin{proposition}\label{prop:mini-orbits}
    Let $O_{i'}$ and $O_i$ be orbits (cycles) of the permutations $\partsh_{i'}\circ \partesh_{i'}$ and $\partsh_i \circ \partesh_i$ respectively.  Then the computation of $O_{i'}$, respectively $O_i$, generates exactly $|O_{i'}|-1$, respectively $|O_{i}|-1$, genomic tableaux in each direction as in Theorem \ref{thm:k-theory-correspondence}, via the maps $\partesh_{i'}$ and $\partesh_i$ respectively.
  \end{proposition}
  
  \begin{proof}
  (Sketch.)  It is not hard to verify, using the Removable Pairs Lemma (\ref{lem:removable-pairs}) and the definition of $\partesh_{i'}$, that the orbits $O_{i'}$ always have the pattern of moving $\ybox$ up one column of $i'$ entries at a time before jumping back down to the starting column, as illustrated in Figure \ref{fig:omega-i-prime}.  Each non-adjacent step of $\ybox$ to the next column of $i'$ entries generates one genomic tableau in reverse standardization order, and only one step of the orbit has no such step.  Similarly the special step of the orbit generates $|O_{i'}|-1$ genomic tableaux in standardization order.
    
    The orbits $O_i$ can either look similar to the loops $O_{i'}$ (but on the horizontal $i$-strips rather than the vertical $i'$-strips), or if an exceptional move is involved in some $\partesh_i$ step then they can have the form of the example in Figure \ref{fig:omega-i-exceptional}.  In either case, the result is again easily verified.
  \end{proof}

\begin{figure}
\begin{center}
\small 
\begin{ytableau}
 \none & \none & \none & \none & 1' & 1 \\
 \none & \none & \none & \none & 1' \\
 \none & \none & \none & 1' \\
 \none & \none & \none & 1' \\
 \none & \none & \times \\
 \none & 1' & 2' \\
 1 & 2' \\
 2
\end{ytableau} \hspace{0.5cm}
\begin{ytableau}
 \none & \none & \none & \none & \times & 1 \\
 \none & \none & \none & \none & 1' \\
 \none & \none & \none & 1' \\
 \none & \none & \none & 1' \\
 \none & \none & 1' \\
 \none & 1' & 2' \\
 1 & 2' \\
 2
\end{ytableau} \hspace{0.5cm}
\begin{ytableau}
 \none & \none & \none & \none & 1' & 1 \\
 \none & \none & \none & \none & 1' \\
 \none & \none & \none & \times  \\
 \none & \none & \none & 1' \\
 \none & \none & 1' \\
 \none & 1' & 2' \\
 1 & 2' \\
 2
\end{ytableau}

\end{center}

\caption{\label{fig:omega-i-prime} The elements, in order, of an $\omega_{i'}=\partsh_{1'} \circ \partesh_{1'}$ orbit.
}
\end{figure}

\begin{figure}
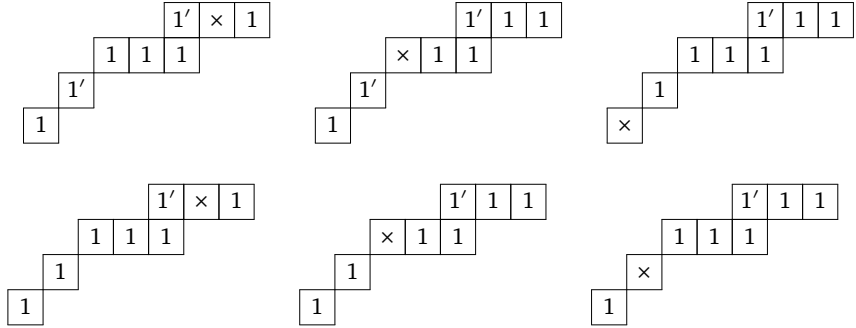

\begin{center}
\small
\begin{ytableau}
 \none & \none & \none & \none & 1' & \times & 1 \\
 \none & \none & 1 & 1 & 1 \\
 \none & 1' \\
 1 
\end{ytableau}\hspace{0.5cm}
\begin{ytableau}
 \none & \none & \none & \none & 1' & 1 & 1 \\
 \none & \none & \times & 1 & 1 \\
 \none & 1' \\
 1 
\end{ytableau}\hspace{0.5cm}
\begin{ytableau}
 \none & \none & \none & \none & 1' & 1 & 1 \\
 \none & \none & 1 & 1 & 1 \\
 \none & 1 \\
 \times 
\end{ytableau} \vspace{0.5cm}

\begin{ytableau}
 \none & \none & \none & \none & 1' & \times & 1 \\
 \none & \none & 1 & 1 & 1 \\
 \none & 1 \\
 1 
\end{ytableau}\hspace{0.5cm}
\begin{ytableau}
 \none & \none & \none & \none & 1' & 1 & 1 \\
 \none & \none & \times & 1 & 1 \\
 \none & 1 \\
 1 
\end{ytableau}\hspace{0.5cm}
\begin{ytableau}
 \none & \none & \none & \none & 1' & 1 & 1 \\
 \none & \none & 1 & 1 & 1 \\
 \none & \times \\
 1 
\end{ytableau}\hspace{0.5cm}

\end{center}

\caption{\label{fig:omega-i-exceptional} An $\partsh_1 \circ \partesh_1$-orbit involving an exceptional move.  The exceptional move occurs on the step from the last element to the first element shown above.}
\end{figure}

Write $\mathrm{refl}(\pi)$ for the \textbf{reflection length} of a permutation $\pi$, the minimum number of transpositions required to generate the permutation. Then since the reflection length of any orbit $\mathcal{O}$ of $\omega$ is equal to $|\mathcal{O}|-1$, by equation \eqref{eqn:factor-omega} we have $$|\LRyb|-|\mathrm{Orb}(\omega)|=\mathrm{refl}(\omega)\le  \sum_i \mathrm{refl}(\omega_i)+\mathrm{refl}(\omega_{i'}).$$
 The right hand side of the above equation is equal to $|\K(\gamma^c/\alpha;\beta)|$ by Proposition \ref{prop:mini-orbits}, and \eqref{eqn:inequality} now follows.

\section{Conjectures}

We have seen that many of the geometric and combinatorial properties of type A Schubert curves have a natural type B analog.  Based on further results and conjectures in type A \cite{bib:GillespieLevinson}, we also make the following conjectures in this setting.

\subsection{Orbit-by-orbit inequality}

In type A, we conjectured an ``orbit-by-orbit'' version of the inequality \eqref{eqn:inequality} as follows.  Note that the right hand side, $|\LRyb|-|\mathrm{Orb}(\omega_{\alpha,\beta,\gamma})|$, can be written as the summation $$\sum_{\mathcal{O}\in \mathrm{Orb}(\omega)} \left(|\mathcal{O}|-1\right).$$
We conjecture in particular that for each orbit $\mathcal{O}$ of $\omega$, the number of genomic tableaux that are traversed in standardization order, or in reverse standardization order, in this orbit is at least $|\mathcal{O}|-1$.  Computer calculations indicate that this refined inequality may be true in type B as well.  We state this precisely as follows.

\begin{conjecture}
 Let $\mathcal{O}\in \mathrm{Orb}(\omega)$.  There are at least $|\mathcal{O}|-1$ genomic tableaux traversed in standardization order and at least $|\mathcal{O}|-1$ traversed in reverse standardization order in the computation of $\mathcal{O}$ using the switching algorithm.
\end{conjecture}

\subsection{High-genus Schubert curves}

In type A, the authors exhibit a family of Schubert curves in in $\Gr$, determined by three partitions $\alpha,\beta,\gamma$, with arbitrarily high arithmetic genus as follows.  If $\omega$ has only one orbit then there is only one real connected component and so $S$ is connected and integral (since it is reduced by Lemma \ref{lem:geometry-schubert-curves-type-B}).  Then we have 
\begin{eqnarray*}
\chi(\mathcal{O}_S)&=&|\LRyb|-|\K(\gamma^c/\alpha;\beta)| \\ 
  &=& \dim_{\mathbb{C}}H^0(\mathcal{O}_S)-\dim_{\mathbb{C}} H^1(\mathcal{O}_S) \\
  &=& 1-g_a(S).
\end{eqnarray*}

Thus $g_a(S)=|\K(\gamma^c/\alpha;\beta)|-|\LRyb|+1$.  It follows that, to find a family of irreducible Schubert curves with arbitrarily high arithmetic genus, it suffices to find a family of partitions $\alpha,\beta,\gamma$ such that $\omega$ has a single orbit and the difference $|\K(\gamma^c/\alpha;\beta)|-|\LRyb|$ grows arbitrarily large. We found such a family in type A using the combinatorial connection between orbits and genomic tableaux. Since the above computations go through verbatim in type B, and we have a similar bijection to K-theory tableaux in this setting, we make the following conjecture.

\begin{conjecture}
  The Schubert curves in $\OG$ attain arbitrarily high arithmetic genus $g_a(S)$.
\end{conjecture}

\subsection{Number of steps in the switching algorithm}

We conclude with a surprising combinatorial conjecture about the length of the switching algorithm:

\begin{conjecture}
Let $(\ybox,T)\in \LRyb$.  The number of steps (switches) in the switching algorithm for computing $\esh(\ybox,T)$ is equal to $$2s+\beta_s-1$$ where $s$ is the index of the first Phase 2 step. In particular, there are exactly as many steps as in the promotion step of the rectification algorithm described in Section \ref{sec:esh}. \end{conjecture}

In Type A, the analogous statement is an easy consequence of the proof of the switching algorithm, due to a natural coplactic decomposition into switches past horizontal and vertical strips \cite{bib:GillespieLevinson}. No such decomposition exists for Phase 2 in type B (see Proposition \ref{prop:phase2-indecomposable}), yet this conjecture indicates, remarkably, that the length of the evacuation shuffle path of the $\ybox$ is coplactic. Indeed, the conjecture suggests that the path itself somehow corresponds to the promotion path through the rectified tableau. We have verified computationally that the conjecture holds for all tableaux with $|\alpha| + |\beta| + |\gamma| \leq 10$ (approximately 913000 words). It would be interesting to see a proof of this conjecture and any accompanying tableau decompositions.

\end{document}